\documentclass[a4paper,12pt]{article}
\usepackage{amsthm,amsmath,amssymb,amsfonts,mathrsfs,color,fancyhdr,xypic,mathrsfs,mathtools, hyperref, enumitem}
\usepackage{tikz,standalone,float}
\usepackage[labelfont=bf]{caption}
\usepackage{eurosym}
\hypersetup{
    colorlinks=true, 
    linktoc=all,     
    linkcolor=blue,  
citecolor=black}
\usepackage[english]{babel}
\usetikzlibrary{intersections,decorations.pathmorphing}
\usepackage{german}
\def\bigboxplus{\mathop{\raisebox{-2pt}{\scalebox{1.6}{$\kern1pt\boxplus\kern1pt$}}}}
\def\Aut{\mathop{\rm Aut}\nolimits}
\def\Stab{\mathop{\rm Stab}\nolimits}

\def\End{\mathop{\rm End}\nolimits}
\def\deg{\mathop{\rm deg}\nolimits}
\def\mod{\mathop{\rm mod}\nolimits}

\def\Image{\mathop{\rm Im}\nolimits}

\def\Quot{\mathop{\rm Quot}\nolimits}

\def\PGL{\mathop{\rm PGL}\nolimits}

\def\id{{\rm id}}

\def\Span{\mathop{\rm Span}\nolimits}

\def\Spec{\mathop{\rm Spec}\nolimits}
\def\Proj{\mathop{\rm Proj}\nolimits}
\def\Sym{\mathop{\rm Sym}\nolimits}
\def\max{\mathop{\rm max}}

\def\op{\mathop{\rm op}\nolimits}

\let\phi\varphi
\let\epsilon\varepsilon
\let\setminus\smallsetminus
\let\emptyset\varnothing

\let\leq\leqslant
\let\geq\geqslant

\newcommand{\BA}{{\mathbb{A}}}

\newcommand{\BF}{{\mathbb{F}}}
\newcommand{\BG}{{\mathbb{G}}}

\newcommand{\BN}{{\mathbb{N}}}

\newcommand{\BP}{{\mathbb{P}}}

\newcommand{\BZ}{{\mathbb{Z}}}

\newcommand{\Fm}{{\mathfrak{m}}}

\newcommand{\CB}{{\cal B}}
\newcommand{\CC}{{\cal C}}

\newcommand{\CE}{{\cal E}}
\newcommand{\CF}{{\cal F}}
\newcommand{\CG}{{\cal G}}

\newcommand{\CL}{{\cal L}}
\newcommand{\CM}{{\cal M}}

\newcommand{\CO}{{\cal O}}

\newcommand{\CU}{{\cal U}}

\newcommand{\CZ}{{\cal Z}}

\def\UX{{\underline{X}}}
\def\UT{{\underline{T}}}
\def\Ut{{\underline{t}}}

\def\os{{\overline{s}}}

\def\ov{{\overline{v}}}
\def\oV{{\overline{V}}}
\def\oW{{\overline{W}}}
\def\oW{{\overline{W}}}
\def\oC{{\overline{C}}}

\def\oCalE{{\overline{\mathcal{E}}}}

\def\mrV{\mathring{V}}
\def\OE{{\overline{E}}}

\let\longto\longrightarrow

\let\into\hookrightarrow

\let\onto\twoheadrightarrow

\newtheorem{theorem}{Theorem}[section]
\newtheorem*{Mtheorem*}{Main Theorem}
\newtheorem{lemma}[theorem]{Lemma}

\newtheorem*{lemma*}{Lemma}
\newtheorem*{Klemma*}{Key Lemma}
\newtheorem*{theorem*}{Theorem}
\newtheorem*{satz*}{Satz}
\newtheorem{proposition}[theorem]{Proposition}
\newtheorem{corollary}[theorem]{Corollary}
\newtheorem{definition}[theorem]{Definition}

\newenvironment{remark}[1][Remark.]{\begin{trivlist}
\item[\hskip \labelsep {\bfseries #1}]}{\end{trivlist}}

\def\OV{{\Omega_V}}

\def\OF{{\overline\CF}}
\def\OE{{\overline\CE}}
\def\Fern{\mathop{\underline{\rm Fern}}\nolimits}
\def\SchFqOp{\mathop{\underline{\rm Sch}_{\BF_q}^{\op}}\nolimits}
\def\SchFq{\mathop{\underline{\rm Sch}_{\BF_q}}\nolimits}
\def\Set{\mathop{\underline{\rm Set}}\nolimits}
\def\Mor{\mathop{\rm Mor}\nolimits}
\def\ch{\mathop{\rm char}\nolimits}
\def\SE{\mathscr{E}}
\numberwithin{equation}{section}

\numberwithin{equation}{section}
\begin{document}


\begin{titlepage}

\begin{center}
    Diss. ETH No. 25610
\end{center}
 
\medskip
 
\medskip
 
\medskip
 
\begin{center}
\begin{huge}
Compactification of the finite Drinfeld period domain as a moduli space of ferns
\end{huge}
\end{center}
 
\medskip
 
\medskip
 
\medskip
 
\begin{center}
    A thesis submitted to
 attain the degree of 
\end{center}

\begin{center}
    \begin{large}
    DOCTOR OF SCIENCES of ETH ZURICH
    \end{large}
\end{center}

\begin{center}
    (Dr. sc. ETH Zurich)
\end{center}
 
\medskip
 
\medskip
 
\begin{center}
    presented by
\end{center}

\begin{center}
    ALEXANDRE PUTTICK\\
    
    Master, Mathematics, Universit\'e de Paris-Sud 11\\
    
    born August 22, 1989\\
    
    citizen of the United States of America
\end{center}
 
\medskip
 
\medskip
 
\medskip
 
\begin{center}
    accepted on the recommendation of
\end{center}
 
\begin{center}
    Prof. Dr. Richard Pink, examiner\\
    Prof. Dr. Ching-Li Chai, co-examiner
\end{center}

\medskip
 
\medskip
 
\medskip
\begin{center}
    2018
\end{center}
\end{titlepage}
\pagenumbering{roman}
\setcounter{section}{-2}
\section*{Summary}
\addcontentsline{toc}{section}{Summary}
Let $\BF_q$ be a finite field with $q$ elements and let $V$ be a vector space over $\BF_q$ of dimension $n>0$. Let $\OV$ be the Drinfeld period domain over $\BF_q$. This is an affine scheme of finite type over $\BF_q$, and its base change to $\BF_q(t)$ is the moduli space of Drinfeld $\BF_q[t]$-modules with level $(t)$ structure and rank $n$.  In this thesis, we give a new modular interpretation to Pink and Schieder's smooth compactification $B_V$ of $\OV$. 

Let $\hat V$ be the set $V\cup\{\infty\}$ for a new symbol $\infty$. We define the notion of a $V$-fern over an $\BF_q$-scheme $S$, which consists of a stable $\hat V$-marked curve of genus $0$ over $S$ endowed with a certain action of the finite group $V\rtimes \BF_q^\times$. Our main result is that the scheme $B_V$ represents the functor that associates an $\BF_q$-scheme $S$ to the set of isomorphism classes of $V$-ferns over $S$. Thus $V$-ferns over $\BF_q(t)$-schemes can be regarded as generalizations of Drinfeld $\BF_q[t]$-modules with level $(t)$ structure and rank $n$. 

To prove this theorem, we construct an explicit universal $V$-fern over $B_V$. We then show that any $V$-fern over a scheme $S$ determines a unique morphism $S\to B_V$, depending only its isomorphism class, and that the $V$-fern is isomorphic to the pullback of the universal $V$-fern along this morphism.

We also give several functorial constructions involving $V$-ferns, some of which are used to prove the main result. These constructions correspond to morphisms between various modular compactifications of Drinfeld period domains over $\BF_q$. We describe these morphisms explicitly.

\newpage
 \sloppy
\section*{Zusammenfassung}
\addcontentsline{toc}{section}{Zusammenfassung}
Seien $\BF_q$ ein endlicher K\"orper mit $q$ Elementen und $V$ ein Vektorraum \"uber $\BF_q$ endlicher Dimension $n>0$. Sei $\OV$ der Drinfeld'sche Periodenbereich \"uber $\BF_q$. Dieser ist ein affines Schema von endlichem Typ \"uber $\BF_q$, dessen Basis-Wechsel nach $\BF_q(t)$ der Modulraum von Drinfeld $\BF_q[t]$-Moduln vom Rang $n$ mit Niveau $(t)$ Struktur ist. In dieser Arbeit geben wir Pinks und Schieders glatter Kompaktifizierung $B_V$ von $\OV$ eine neue modulare Interpretation. 

Sei $\hat V$ die Menge $V\cup\{\infty\}$ f"ur ein neues Symbol $\infty$. Wir f"uhren den Begriff eines $V$-Farns \"uber einem $\BF_q$-Schema ein. Ein solcher besteht aus einer stabilen $\hat V$-markierten Kurve vom Geschlecht 0 \"uber $S$ zusammen mit einer gewissen Wirkung der endlichen Gruppe $V\rtimes\BF_q^\times$. Unser Hauptsatz ist, dass das Schema $B_V$ den Funktor repr"asentiert, der einem $\BF_q$-Schema die Menge der Isomorphieklassen von $V$-Farnen \"uber $S$ zuordnet. Deshalb k\"onnen $V$-Farne \"uber $\BF_q(t)$-Schemen als Verallgemeinerungen von Drinfeld $\BF_q[t]$-Moduln vom Rang $n$ mit Niveau $(t)$ Struktur betrachtet werden.

Um diesen Satz zu beweisen, konstruieren wir einen expliziten universellen $V$-Farn \"uber $B_V$. Danach zeigen wir, dass ein beliebiger $V$-Farn "uber einem Schema $S$ einen eindeutigen Morphismus $S\to B_V$ bestimmt, der nur von der Isomorphieklasse abh"angt, derart, dass der $V$-Farn isomorph zur Zur"uckziehung des universellen $V$-Farns entlang dieses Morphismus ist.

Wir f"uhren auch verschiedene funktorielle Konstruktionen mit $V$-Farnen ein. Einige davon werden verwendet im Beweis des Hauptsatzes. Diese Konstruktionen entsprechen Morphismen zwischen unterschiedlichen modularen Kompaktifizierungen von Drinfeld'schen Periodenbereichen "uber $\BF_q$, die wir explizit beschreiben.

\fussy
\newpage
\subsection*{Acknowledgments}
\addcontentsline{toc}{section}{Acknowledgments}
My greatest thanks go to Prof. Dr. Richard Pink, who suggested the thesis topic and provided guidance and many invaluable ideas and insights. This includes inventing the concept of a $V$-fern and suggesting the construction of the universal $V$-fern. His mentoring pushed me to think much more thoroughly and carefully about both mathematics and the way it is presented, all while keeping the big picture in sight. Indeed, his influence has affected my thinking in all domains. It is a gift to have been stretched so far beyond my former horizons.
\vspace{2mm}

I would like to thank Prof. Dr. Ching-Li Chai for accepting to be the co-examiner for my doctoral thesis and reading through my work.
\vspace{2mm}

I also extend my gratitude to my fellow PhD students Jennifer-Jayne Jakob, Simon H"aberli, Nicolas M"uller and Felix Hensel for interesting mathematical discussions along with a good deal of emotional support. Special thanks especially to Simon H"aberli for proofreading my introduction and providing several useful suggestions.
\vspace{2mm}

Finally, I'd like to thank my family and friends for inspiration, support, love and joy throughout the course of my PhD studies.
\newpage
\tableofcontents
\vspace{10mm}
\pagenumbering{arabic}


\setcounter{section}{-1}
\section{Introduction}
Let $\BF_q$ be a finite field with $q$ elements. This thesis was motivated by a question involving Drinfeld modules and their moduli spaces. For an overview of the theory, see Drinfeld \cite{Drin}, Goss \cite{Go} and Deligne-Husem\"oller \cite{Del-Hu}. More specifically, we are interested in compactifications of such moduli spaces and will focus on the special case of Drinfeld $\BF_q[t]$-modules of generic characteristic and rank $n$ with level $(t)$ structure. In this case, the corresponding fine moduli space is obtained via base change from a \emph{Drinfeld period domain over a finite field} $\OV$, which is an affine algebraic variety over $\BF_q$ (see below for more details). In \cite{P-Sch}, Pink and Schieder construct two projective compactifications of $\OV$ and give each a modular interpretation: the normal \emph{Satake compactification} $Q_V$, and a smooth compactification $B_V$ which dominates $Q_V$. The aim of this thesis is to give $B_V$ a new modular interpretation in terms of geometric objects that we call $V$-ferns. This will hopefully aid in future efforts to construct useful compactifications of more general Drinfeld modular varieties.

\subsection{Drinfeld Modular Varieties}Let $C$ be a smooth projective geometrically irreducible curve over $\BF_q$. Let $\infty\in C$ be a closed point and let $A:=\CO_C(C\setminus\{\infty\}) $. Such an $A$ is called an \emph{admissible coefficient ring}. Let $F:=\Quot(A)$ and let $K$ be a field extension of $F$. The \emph{ring of additive polynomials} $K[\tau]$ is the ring of polynomials in $\tau$ over $K$ subject to the relation $\tau a=a^q\tau$ for all $a\in K$. A \emph{Drinfeld $A$-module} (of generic characteristic) over $K$ is a homomorphism of $\BF_q$-algebras
$$\psi\colon A\to K[\tau],\,\,a\mapsto \psi_{a,0}+\psi_{a,1}\tau+\psi_{a,2}\tau^2+\ldots$$ such that $\psi_{a,0}=a$ for all $a\in A$, and $\psi(A)\not\subset K.$

Let $\psi$ be a Drinfeld $A$-module over $K$. There exists an $n\in\BZ_{>0}$ such that $\deg_\tau(\psi_a)=n\deg_A (a)$ for all $a\in A$, where $\deg_A (a):=\dim_{\BF_q}A/(a).$ The integer $n$ is called the \emph{rank} of $\psi.$ Let $I\subset A$ be a non-zero proper ideal. A \emph{level $I$ structure} on $\psi$ is an $A$-module isomorphism $$\lambda\colon(I^{-1}A/A)^r\stackrel\sim\longto\bigcap_{i\in I}\ker(\psi_i)=:\psi[I]\subset K.$$ Here $\ker(\psi_i)$ denotes the set of zeros of $\psi_i$ (a priori contained in an algebraic closure of $K$), and the definition of $\lambda$ presupposes that $\psi[I]\subset K$.
\vspace{2mm}

One can generalize these notions and define Drinfeld $A$-modules of rank $n$ and level $I$ over an arbitary $F$-scheme. The corresponding moduli functor is represented by a smooth irreducible $(n-1)$-dimensional affine algebraic variety $M^n_{A,I}$ over $F$. It is then natural to search for meaningful compactifications of $M^n_{A,I}$. An example is the \emph{Satake compactification} of Pink \cite{Pink}, which he denotes by $\overline{M}^n_{A,I}$. This is normal integral proper algebraic variety over $F$, but is singular when $n\geq 3$ (see \cite[Theorem 7.8]{Pink}). It is our hope that generalizing the theory we describe in this thesis will lead to a ``nice'' (e.g. smooth, toroidal...) compactification $M^n_{A,I}\subset B^n_{A,I}$ dominating $\overline{M}^n_{A,I}$, in such a way that $B^n_{A,I}$ itself has a meaningful geometric modular interpretation. 
\vspace{2mm}

\subsection{The case $A=\BF_q[t]$ and $I=(t)$}
In the special case of Drinfeld $\BF_q[t]$-modules of rank $n$ and level $(t)$, the fine moduli space $M^n_{\BF_q[t],(t)}$ is obtained via base change to $\Spec\BF_q(t)$ from:
$$\Omega^n:=\BP^{n-1}_{\BF_q}\setminus\bigcup_{H}H,$$
where $H$ runs over the set of $\BF_q$-hyperplanes in $\BP^{n-1}_{\BF_q}$.
The scheme $\Omega^n$, often referred to as the \emph{Drinfeld period domain} over $\BF_q$, has been studied by Rapoport \cite{Rap}, Orlik-Rapoport \cite{Orl} and Pink-Schieder \cite{P-Sch}. As in \cite{P-Sch}, it will be useful to adopt a coordinate free formulation in order to describe the compactifications of $\Omega^n$ that we are interested in. Let $V$ be a finite dimensional vector space over $\BF_q$ of dimension $n$. Let $P_V:=\Proj\big(\Sym V\big)$, and define
$$\Omega_V:=P_V\setminus\bigcup_{H}H,$$
where $H$ runs over the set of $\BF_q$-hyperplanes in $P_V$.
Any basis of $V$ induces an isomorphism $P_V\stackrel\sim\to\BP^{n-1}_{\BF_q}$ sending $\OV$ to $\Omega^n$.
In \cite{P-Sch}, a smooth compactification $B_V$ of $\OV$ is constructed that differs in general from the tautological compactification $P_V$. The scheme $B_V$ is projective and defined as a closed subscheme of $\prod_{0\neq V'\subset V}P_{V'},$ where $V'$ runs over all non-zero subspaces of $V$. It is shown that $B_V$ possesses a natural stratification indexed by the flags of subspaces of $V$ and that the complement $B_V\setminus\OV$ is a divisor with normal crossings. Further properties of $B_V$ are established by Linden \cite{Lind}. For further details, see Section \ref{OVBV}.
\vspace{2mm}

\subsection{Smooth $V$-ferns}
We now turn to $V$-ferns, whose definition we motivate in the following paragraphs. 
Keeping in mind the ever-fruitful analogy between Drinfeld modules and elliptic curves, these $V$-ferns should correspond to generalizations of Drinfeld modules in the spirit of the generalized elliptic curves of Deligne-Rapoport \cite{Del-Rap}. Indeed, the latter are used to accomplish a task analogous to what one hopes to accomplish for Drinfeld modular varieties, namely, to construct smooth algebraic compactifications of moduli of elliptic curves with given level. On the other hand, we rely heavily on the work of Knudsen \cite{Knud} concerning the moduli space of stable $n$-pointed curves of a given genus. We recall the definition (with a formulation that is more suited to our purposes) in the case of genus~0:
\begin{definition}Let $K$ be a field, and let $I$ be a finite set. A \emph{stable $I$-marked curve of genus $0$} over $K$ is a pair $(C,\lambda)$ where
\begin{enumerate}
\item $C$ is a geometrically reduced, geometrically connected projective curve of genus 0 over $K$ with at worst nodal singularities, and
\item $\lambda\colon I\into C^{\mathop{\rm sm}}(K),\,\,i\mapsto\lambda_i,$ where $C^{\mathop{\rm sm}}$ denotes the smooth locus of $C$, is an injective map such that for each irreducible component $E\subset C$, the number of marked points $\lambda_i$ on $E$ plus the number of singular points on $E$ is $\geq 3$.
\end{enumerate}
\end{definition}
We call the map $\lambda$ an \emph{$I$-marking} of $C$. Since we never consider curves of higher genus, we use the phrase ``stable $I$-marked curve'' in place of  ``stable $I$-marked curve of genus 0'' from now on. We visualize such curves as trees of copies of $\BP^1_K$. Stable $I$-marked curves and their generalization to arbitrary base schemes are discussed in detail in Section \ref{stablemarkedcurves}. 
\vspace{2mm} 

Let $V:=\bigl(t^{-1}\BF_q[t]/\BF_q[t]\bigr)^n$.
A Drinfeld $\BF_q[t]$-module of rank $n$ and level $(t)$ over a field extension $K$ of $\BF_q(t)$ may be viewed as a pair $(\psi,\lambda)$ consisting of a homomorphism 
\begin{eqnarray*}
\psi\colon\BF_q[t]&\longto&\End_K(\BG_{a})\cong K[\tau]\\
a&\mapsto& \psi_{a}
\end{eqnarray*}
 along with an injective $\BF_q$-linear map 
\begin{equation*}\label{lvlstr}\lambda\colon V\stackrel\sim\to\ker(\psi_t)(K)\subset\BG_a(K)=(K,+).\end{equation*}
In fact $\psi$ is determined by $\psi_t$, which is in turn determined by its zeros and the condition that $\psi_{t,0}=t$, so the level structure $\lambda$ \textbf{completely determines} $\psi$. 

Let $\hat V:=V\cup\{\infty\}$ for a new symbol $\infty$. We add a ``point at infinity'' via the open embedding 
$$\BG_{a,K}\into\BP^1_K,\,\,t\mapsto (t:1),$$ and obtain a map $\hat\lambda\colon \hat V\to\BP^1_K(K)$ from $\lambda$ by sending $\infty$ to $(1:0)$. The pair $(\BP^1_K,\lambda)$ is then a smooth stable $\hat V$-marked curve. One obtains a left action of the group $$G:=V\rtimes\BF_q^\times$$ on $\BP^1_K$ via the homomorphism
$$\phi\colon G\to\PGL_2(K),\,\,(v,\xi)\mapsto\begin{pmatrix}
\xi& \lambda_v\\0 & 1
\end{pmatrix}.
$$
This leads to the following definition:
\begin{definition}\label{smoothferndef} A \emph{\textbf{smooth} $V$-fern} over $K$ is a tuple $(C,\lambda,\phi)$ consisting of a smooth stable $\hat V$-marked curve $(C,\lambda)$ and a left $G$-action
$$\phi\colon G\to \Aut_K(C),\,\,(v,\xi)\mapsto \phi_{v,\xi}$$ such that
\begin{enumerate}
\item $\phi_{v,\xi}(\lambda_w)=\lambda_{\xi w+v}$ for all $w\in \hat V$, and
\item there exists an isomorphism $C\stackrel\sim\to\BP^1_K$ with $\lambda_0\mapsto(0:1)$ and $\lambda_\infty\mapsto(1:0)$ under which $\phi_{\xi}$ corresponds to $\begin{pmatrix}\xi &0\\0&1
\end{pmatrix}\in\PGL_2(K)$ for all $\xi\in\BF_q^\times$.
\end{enumerate}
\end{definition}
The upshot of the preceding discussion is that one can naturally associate a smooth $V$-fern over $K$ to a Drinfeld $\BF_q[t]$-module of rank $n$ and level $(t)$ over $K$. In the above notation, this is given by $(\psi,\lambda)\mapsto \lambda\mapsto (\BP^1_K,\hat\lambda,\phi)$. After passing to isomorphism classes, the notions are in fact equivalent. 

We can generalize to the notion of a smooth $V$- fern over an arbitrary $\BF_q$-scheme $S$. Roughly speaking, a smooth $V$- fern over $S$ is a flat family over $S$ of objects of the type in Definition \ref{smoothferndef}. One can show the following with relative ease:\footnote{Without using Proposition \ref{smoothferns0} as an intermediary, we will prove a more general result directly and obtain Proposition \ref{smoothferns0} as a corollary.}
\begin{proposition}[see Corollary \ref{7-cor-smoothferns}]\label{smoothferns0}The scheme $\OV$ represents the functor that associates to an $\BF_q$-scheme $S$ the set of isomorphism classes of smooth $V$-ferns over $S$.\end{proposition}

\begin{remark}While the notion of a smooth $V$-fern makes sense over any $\BF_q$-scheme, Drinfeld modules of generic characteristic only occur over $\BF_q(t)$-schemes. To generalize the equivalence between smooth $V$-ferns over $K$ and Drinfeld $\BF_q[t]$-modules of rank $n$ and level $(t)$ over $K$, we must restrict ourselves to base schemes defined over $\Spec(\BF_q(t))$. However, we only restricted ourselves to Drinfeld modules of generic characteristic for simplicity. In fact, one may drop the requirement that the Drinfeld module be of generic characteristic, instead requiring that the characteristic of the Drinfeld module does not divide $(t)$. One then gets an equivalence for base schemes defined over $\Spec(\BF_q[t,t^{-1}])$. One can go even further and generalize the notion of level structure to that of a \emph{Drinfeld level structure} (see, for instance, \cite[Definition 6.1]{Del-Hu}). This allows one to define Drinfeld modules of rank $n$ and level $(t)$ over arbitrary $\BF_q[t]$-schemes, though the correspondence with smooth $V$-ferns then breaks down.
\end{remark}

\subsection{General $V$-ferns}
To motivate the general definition of a $V$-fern, let $R$ be a discrete valuation ring over $\BF_q$, and suppose we are given a smooth $V$-fern $(C,\lambda,\phi)$ over the generic point of $\Spec R$. Knudsen shows in \cite{Knud} that there is a fine moduli space $\CM_{\hat V}$ of stable $\hat V$-marked curves and that $\CM_{\hat V}$ is proper. The valuative criterion for properness (\cite[\S II.4]{Ha}) then says that the morphism $\Spec K\to \CM_{\hat V}$ corresponding to $(C,\lambda)$ extends uniquely to a morphism $\Spec R\to \CM_{\hat V}$. On other words, the curve $(C,\lambda)$ extends uniquely to a stable $\hat V$-marked curve $(\tilde C,\tilde\lambda)$ over $\Spec R$. By the functoriality of the extension, the $G$-action $\phi$ can be extended uniquely to $\tilde C$. The corresponding tuple $(\tilde C,\tilde\lambda,\tilde\phi)$ is characterized by certain conditions generalizing those in Definition \ref{smoothferndef}. It is important to note that the special fiber of $\tilde C$ can be singular, and hence the resulting object is \textbf{not} in general a smooth $V$-fern over $\Spec R$. 

It is then natural to define a $V$-fern over a general $\BF_q$-scheme $S$ as a tuple $(C,\lambda,\phi)$, where $(C,\lambda)$ is a stable $\hat V$-marked curve over $S$ and $\phi\colon G\to\Aut_S(C)$ is a left group action of $G$ on $C$ satisfying certain conditions generalizing those of Definition \ref{smoothferndef} (see Definition \ref{FernDef}). Note that we often write $C$ in place of $(C,\lambda,\phi)$. A smooth $V$-fern is then just a $V$-fern with smooth fibers. Since this is an open condition, Proposition \ref{smoothferns0} implies that the moduli space of $V$-ferns, if it exists, contains $\OV$ as an open subscheme. 
The argument that a smooth $V$-fern over the generic point of $\Spec R$ extends uniquely to a $V$-fern over $\Spec R$ applies just as well in the case of a general $V$-fern over the generic point. The moduli space of $V$-ferns would thus satisfy the valuative criterion for properness, so the notion of $V$-ferns points toward a modular compactification of $\OV$.

\subsection{Constructions involving $V$-ferns}
We will use several constructions involving $V$-ferns repeatedly and in combination. For convenience, we briefly introduce each of them here. Let $S$ be a scheme over $\BF_q$ and let $(C,\lambda,\phi)$ be a $V$-fern over $S$. Consider a subspace $0\neq V'\subsetneq V.$
\subsubsection*{Contraction} Contraction associates a $V'$-fern $(C',\lambda',\phi')$ to $C$, along with a morphism $C\to C'$. Conceptually, one obtains $C'$ from $C$ by forgetting the $(V\setminus V')$-marked points and contracting irreducible components until the curve is again stable. The contraction $C'$ inherits a $\hat V'$-marking from $C$, and the $G$-action on $C$ induces a left $(G':=V'\rtimes\BF_q^\times)$-action on $C'$. If $C$ is smooth, then $(C',\lambda',\phi')=(C,\lambda|_{V'},\phi|_{G'})$, and the contraction morphism $C\to C'$ is the identity.

\subsubsection*{Grafting}Let $\oV:=V/V'$. Grafting yields a $V$-fern from a given $V'$-fern $(C',\lambda',\phi')$ and $\oV$-fern $(\oC,\overline\lambda,\overline\phi)$. Intuitively, one does this by gluing a copy of $C'$ to each of the $\oV$-marked points of $\oC$. The resulting curve possesses a natural $\hat V$-marking (up to choosing a decomposition $V=V'\oplus U$) and left $G$-action.

\subsubsection*{Contraction to the $v$-component}
This construction associates to $C$ a pair $(C^v,\lambda^v)$ consisting of a $\BP^1$-bundle over $S$, i.e., a scheme that is Zariski-locally isomorphic to $\BP^1_S$,\footnote{Some authors use the term ``$\BP^1$-bundle'' more generally to mean \'etale locally (equivalently flat locally) trivial.} along with a map
$\lambda^v\colon \hat V\to C^v(S)$ and a morphism $C\to C^v$. We call $\lambda^v$ a \emph{$\hat V$-marking} and say $C^v$ is a (smooth) \emph{$\hat V$-marked curve.}  The construction is related to the notion of contraction described above. Intuitively, one obtains $C^v$ from $C$ fiberwise by contracting all irreducible components not containing the $v$-marked point. We will mostly be interested in the case $v=\infty$, but also consider the case $v=0$. When $C$ is smooth, we obtain $C^v$ by simply forgetting  the $G$-action on~$C$.

\subsubsection*{Line bundles}
We also provide constructions associating line bundles with certain extra structure to $(C,\lambda,\phi)$. One such construction assigns to $C$ a pair $(L,\lambda)$ consisting of a line bundle $L$ over $S$ along with a fiberwise non-zero (see Definition \ref{fiberwise}) $\BF_q$-linear map
$\lambda\colon V\to L(S).$
To obtain $L$, we take the contraction to the $\infty$-component $C^\infty$ and define $L$ to be the complement of the image of the $\infty$-section.

Let $\mrV:=V\setminus\{0\}.$ An analogous construction yields a pair $(\hat L,\rho)$ consisting of a line bundle $\hat L$ over $S$ and a fiberwise non-zero reciprocal map (see Definition \ref{recipmap})
$\rho\colon \mrV\to \hat L(S).$ This construction differs from the preceding one in that we instead take the contraction to the $0$-component $C^0$ and define $\hat L$ to be the complement of the $0$-marked point.

\subsection{Main theorem and outline of the proof}
Our main result is the following:
\begin{Mtheorem*}[see Theorem \ref{rep}] The scheme $B_V$ represents the functor that associates an $\BF_q$-scheme $S$ to the set of isomorphism classes of $V$-ferns over $S$.
\end{Mtheorem*}
 We denote the functor referred to in the Main Theorem by
$$\begin{array}{cccl}
\Fern_V\colon&\SchFqOp&\longto&\Set\\
&S&\mapsto&\left\{\substack{\mbox{Isomorphism classes} \\
    \mbox{of $V$-ferns over $S$}}\right\}.
\end{array}$$
Roughly speaking, our strategy for proving the theorem is to begin by defining a universal $V$-fern $(\CC_V,\lambda_V,\phi_V)$ over $B_V$. For an arbitrary $\BF_q$-scheme $S$ and a $V$-fern $(C,\lambda,\phi)$ over $S$, we then show that there exists a unique morphism $f_C\colon S\to B_V$ such that $C$ is isomorphic to the pullback $f_C^*\CC_V$. We now give a more detailed outline of the proof.
\subsubsection*{Step 1: Construct the universal family}
Let $\Sigma:=V\times (V\setminus\{0\})$ and let $P^\Sigma:=\prod_{(v,w)\in \Sigma}\BP^1.$ We define the universal $V$-fern $\pi\colon \CC_V\to B_V$ as the scheme theoretic closure of the image of a certain morphism from $\OV\times\BA^1$ to $B_V\times P^\Sigma$. 

Recall that a \emph{flag} $\CF$ of $V$ is a set $\{V_0,\ldots,V_m\}$ of subspaces of $V$ with $m\in \BZ_{>0}$ such that $V_0=\{0\}$ and $V_m=V$ and $V_i\subsetneq V_j$ for $i<j$. We say that $\CF$ is \emph{complete} if $m=\dim V$. For each flag $\CF$ of $V$, there is an open subscheme $U_\CF$ of $B_V$ which represents a certain subfunctor of the functor in \cite{P-Sch} which is represented by $B_V$. As $\CF$ varies over all (complete) flags of $V$, the $U_\CF$ form an open covering of $B_V$.

In order to endow $\CC_V$ with the structure of a $V$-fern, we separately consider a scheme $\CC_\CF$, which is a closed subscheme of $U_\CF\times P^\Sigma$ defined by explicit polynomial equations. We endow $\CC_\CF$ with the structure of a $V$-fern and then show that $\CC_\CF= \pi^{-1}(U_\CF)$. This allows us to glue the $V$-fern structures on the various $\CC_\CF$ to obtain one on $\CC_V.$

\subsubsection*{Step 2: Reduction to the case of $\CF$-ferns}
 We show that for a $V$-fern $(C,\lambda,\phi)$ over a scheme $S$ one can associate to each $s\in S$ a flag $\CF_s$ coming from the fiber $C_s$. For a fixed flag $\CF$, we define an \emph{$\CF$-fern} to be a $V$-fern for which $\CF_s\subset\CF$. The locus $S_\CF:=\{s\in S\mid \CF_s\subset \CF\}$ is open in $S$ and the $S_\CF$ cover $S$ as $\CF$ varies over all (complete) flags of $V$. We reduce the proof of the Main Theorem to showing that $U_\CF$ represents $\Fern_\CF$, where $\CF$ is a complete flag, with the universal $\CF$-fern is given by $\CC_\CF$.

\subsubsection*{Step 3: Morphism to $U_\CF$ representing an $\CF$-fern}
Let $\CF$ be a complete flag and let $(C,\lambda,\phi)$ be an $\CF$-fern over a scheme $S$. Using one of the constructions described above, we associate a pair $(L,\lambda)$ to $C$ consisting of a line bundle $L$ over $S$ and an $\BF_q$-linear map $\lambda\colon V\to L(S)$ such that the image of $V$ generates the sheaf of sections $\CL$ of $L$. Using the well-known description of the functor of points of projective space (\cite[\S II.7]{Ha}), the isomorphism class of the pair $(L,\lambda)$ corresponds to a unique morphism $S\to P_V$. By repeatedly contracting and applying the same construction, we obtain morphisms $S\to P_{V'}$ for each $0\neq V'\subset V$, and hence a morphism $$f_C\colon S\to\prod_{0\neq V'\subset V}P_{V'}.$$ The scheme $U_\CF$ is a locally closed subscheme of $\prod_{0\neq V'\subset V}P_{V'}$, and we show that $f_C$ factors through $U_\CF$.

\subsubsection*{Step 4: Key lemma for induction}Write $\CF=\{V_0,\ldots V_n\}$ and $\CF'=\{V_0,\ldots V_{n-1}=:V'\}$. We prove the following lemma:
\begin{Klemma*}[see Lemma \ref{key}]Let $(C,\lambda,\phi)$ and $(D,\mu,\psi)$ be $\CF$-ferns such that the $V'$-contractions $C'$ and $D'$ are isomorphic $\CF'$-ferns. Suppose further that $C^\infty\cong D^\infty$ as $\hat V$-marked curves. Then $C$ and $D$ are isomorphic.
\end{Klemma*}

\subsubsection*{Step 5: Induction on $\dim V$ and the remainder of the proof}
We prove the Main Theorem in the $\dim V=1$ case. In the general case, the Key Lemma allows us to use induction on $\dim V$ to show that $C\cong f_C^*\CC_\CF$ as follows. There is a natural morphism $p\colon U_\CF\to U_{\CF'}$, and we show that the pullback $p^*\CC_{\CF'}$ is isomorphic to the $V'$-contraction $(\CC_\CF)'$ of $\CC_\CF$. We consider the $V'$-contraction $C'$ of $C$ and the corresponding morphism $f_{C'}$ constructed in the same manner as $f_C$. We show that $f_{C'}=p\circ f_C$. Using induction, we obtain a chain of isomorphisms $$C'\cong f_{C'}^*\CC_{\CF'}\cong (p\circ f_C)^*\CC_{\CF'}\cong f_C^*(p^*\CC_{\CF'})\cong f_C^*(\CC_\CF)'.$$ Using the fact that contraction commutes with pullback, the last scheme is isomorphic to the $V'$-contraction of $f_C^*\CC_\CF.$ We then show that the contractions to the $\infty$-components $C^\infty$ and $(f_C^*\CC_\CF)^\infty$ are isomorphic. The may thus apply the Key Lemma to deduce that $C$ and $f_C^*\CC_\CF$ are isomorphic.

Finally, we show that the pullbacks of the universal family under two distinct morphisms from$S$ to $B_V$ yield non-isomorphic $V$-ferns over $S$, concluding the proof of the main theorem.

\subsection{Results on morphisms of moduli spaces} After proving the Main Theorem, we provide some additional results concerning certain natural morphisms amongst the schemes $B_V$ and $Q_V$ and $P_V$, where $V$ may also vary. More precisely, we relate them to the various constructions involving $V$-ferns described above. Contraction and grafting correspond to morphisms $B_V\to B_{V'}$ and $B_{V'}\times B_{\oV}\to B_V$ respectively, whereas the two line bundle constructions give morphisms $B_V\to P_V$ and $B_V\to Q_V$. We will show that all of these agree with the natural morphisms between the same schemes described in \cite{P-Sch} (see Propositions \ref{PVmorph}, \ref{contrmorph}, \ref{graftingmorph} and \ref{QVmorph}). 

\subsection{Outlook for the general case} We conclude the introduction with a short word on the generalization to Drinfeld modular varieties $M^n_{A,I}$ for arbitrary $A$ and $(0)\neq I\subsetneq A$. In \cite{Pink}, Pink defines the Satake compactification of $M^n_{A,I}$ to be a normal integral proper variety $\overline{M}^n_{A,I}$ with an open embedding $M^n_{A,I}\into \overline{M}^n_{A,I}$, which satisfies a certain universal property. In proving the existence of such a compactification, he reduces to the case of Drinfeld $\BF_q[t]$-modules with level $(t)$, where the Satake compactification is given by the base change of $Q_V$ to $\Spec \BF_q(t)$. This reduction is mainly accomplished using the following two observations. First, if $I'\subset I\subset A$ are proper non-zero ideals, then $\overline{M}^n_{A,I}$ can be realized as a quotient of $\overline{M}^n_{A,I'}$ under the action of a finite group. Second, for certain $A\subset A'$ and $I':=IA'$ and $n=n'\cdot[F'/F]$, where $F$ and $F'$ are the corresponding quotient fields, there is a finite injective morphism $M^{n'}_{A',I'}\to M^{n}_{A,I}$. Then $\overline{M}^{n'}_{A',I'}$ is the normalization of $\overline{M}^{n}_{A,I}$ in the function field of $M^{n'}_{A',I'}$. In proving the existence of the Satake compactification from the base case, the first observation allows one to increase the level, and the second allows one to enlarge the admissible coeffecient ring.
\vspace{2mm}

One might use similar reasoning to generalize $B_V$. First, appropriately generalize the notion of a $V$-fern and consider the corresponding moduli functor. Then exploit the morphisms $M^n_{I',A}\to M^n_{I,A}$ and $M^{n'}_{I',A'}\to M^n_{A,I}$ to reduce the question of representabiliy to the case $A=\BF_q[t]$ and $I=(t)$. It is our hope that the foundations laid here will lead to the success of such an approach.

\section*{Outline}
\addcontentsline{toc}{section}{Outline}
\noindent\textbf{Section 1.} We review the notion of \emph{stable marked curves} of genus 0 and recall the concepts of contraction and stabilization from \cite{Knud}. We describe a related construction, the contraction to the $i$-component, of which the contraction to the $v$-component mentioned above is a special case.
\vspace{2mm}

\noindent\textbf{Section 2.} We recall the construction of $B_V$ from \cite{P-Sch} and gather additional results from \emph{loc. cit.} which will be of use to us. In particular, we describe the open subscheme $U_\CF\subset B_V$ associated to a flag $\CF$ of $V$.
\vspace{2mm}

\noindent\textbf{Section 3.} The concept of $V$-ferns is formally introduced. We then demonstrate several of their properties. First, we discuss $V$-ferns over a field, and show that there is a natural way to associate a flag of $V$ to each such fern. As a result, for a $V$-fern over an arbitrary $\BF_q$-scheme $S$, we can associate a flag $\CF_s$ of $V$ to each $s\in S$. We then discuss $V$-ferns for $\dim V=1$ and $2$. In the first case, we show that there is exactly one $V$-fern over $S$ up to (unique) isomorphism. This provides the base case for the induction in our proof of the Main Theorem.
\vspace{2mm}

\noindent\textbf{Section 4.} We describe the contraction, grafting and line bundle constructions associated to $V$-ferns in detail.
\vspace{2mm}

\noindent\textbf{Section 5.} Let $\CF$ be a flag of $V$. We define the notion of an $\CF$-fern. For a scheme $S$ and $V$-fern $C$ over $S$, we also define the locus $S_\CF\subset S$ over which $C$ is an $\CF$-fern. We show that $S_\CF$ is open and that the $S_\CF$ cover $S$ as $\CF$ varies.
\vspace{2mm}

\noindent\textbf{Section 6.}We construct the universal family $\CC_V$ over $B_V$. This constitutes the technical bulk of the thesis.
\vspace{2mm}

\noindent\textbf{Section 7.} Here we prove the Main Theorem.
\vspace{2mm}

\noindent\textbf{Section 8.} We consider the morphisms corresponding the each of the constructions described in Section 4 between various moduli schemes.
\vspace{2mm}


\section*{Notation and conventions}
\addcontentsline{toc}{section}{Notation and conventions}
We gather some notation and conventions here for easy reference:
\begin{eqnarray*}
 \BF_q& &\mbox{ a finite field of $q$ elements};\\
V&&\mbox{ a finite dimensional vector space over $\BF_q$ of dimension $n>0$};\\
\hat V&&\mbox{ the set $V\cup \{\infty\}$ for a new symbol $\infty$.}\\
\mrV&&\mbox{ the set $V\setminus \{0\}$.}\\
G&&\mbox{ the finite group $V\rtimes\BF_q^\times$};\\
S_V&&\mbox{ the symmetric algebra of $V$ over $\BF_q,$}\\
RS_V&&\mbox{ the localization of $S_V$ obtained by inverting all $v\in V\setminus\{0\}$.}\\
P_V&&\mbox{ the scheme }\Proj(S_V);\\
\OV&&\mbox{ the complement of the union of all $\BF_q$-hyperplanes in $P_V$}\\
\Sigma&&\mbox{ the set $V\times\mrV$}\\
P^\Sigma&&\mbox{ the scheme $\prod_{(v,w)\in\Sigma}\BP^1$}.
\end{eqnarray*}
For a morphism $X\to S$ and $s\in S$, we denote the fiber over $s$ by $X_s$. Given another morphism $f\colon T\to S$, we often denote the base change $X\times_S T$ by $f^*X$. For an open immersion $j\colon U\into S$, we will usually write $X|_U$ in place of $j^*X$.

We always view the irreducible components of a scheme as closed subschemes endowed with the induced reduced scheme structure. 


%



\section{Stable marked curves of genus 0}\label{stablemarkedcurves}
In this section we collect some useful facts and constructions involving stable marked curves of genus 0. The main references are \cite{Arb} and \cite{Knud}. Let $S$ be an arbitrary scheme, and let $I$ be a finite set. 
\subsection{Stable $I$-marked curves of genus 0}
\begin{definition}\label{stablecurve}
An \emph{$I$-marked curve} $(C,\lambda)$ of genus 0 over $S$ is a flat and proper scheme $C$ over $S$, together with a map $$\lambda\colon I\to C(S),\, i\mapsto \lambda_i$$ such that for every geometric point $\os$ of $S$,
\begin{enumerate}
\item the geometric fiber $C_\os$ is a reduced and connected curve with at most ordinary double point singularities,
\item the equality $\dim H^1(C_\os,\CO_{C_\os})=0$ holds.
\end{enumerate}
A closed point $p\in C_{\os}$ is called \emph{marked} if $p=\lambda_i(\os)$ for some $i\in I$. The point $p$ is called \emph{special} if it is singular or marked. We say that $(C,\lambda)$ is \emph{stable} if, in addition to the above,
\begin{enumerate}
 \setcounter{enumi}{2}
 \item the curve $C_\os$ is smooth at the marked points $\lambda_i(\os)$,
 \item we have $\lambda_i(\os)\neq \lambda_j(\os)$ for all $i\neq j$,
 \item each irreducible component of $C_\os$, contains at least 3 special points.
\end{enumerate}
\end{definition}
\begin{remark} Conditions (1) and (2) imply that each geometric fiber is a tree of copies of $\BP^1$. We will see later that this already holds on (scheme-theoretic) fibers (Proposition \ref{P1fibs}).
\end{remark}
We call the map $\lambda\colon I\to C(S)$ an \emph{$I$-marking} on $C$. From now on, we abbreviate the expression ``$I$-marked curve of genus 0'' by $I$-marked curve. We will often write $C$ for the pair $(C,\lambda)$ if confusion is unlikely.
\vspace{2mm}

In \cite{Knud} a morphism of stable $I$-marked curves $f\colon (C,\lambda)\to(D,\mu)$ over a scheme $S$ is defined to be an isomorphism over $S$ such that $f\circ\lambda_i=\mu_i$ for all $i\in I$. In the following paragraphs, we will define a more general notion of morphisms between (not necessarily stable) $I$-marked curves that is equivalent to Knudsen's definition when the curves are stable. One of the main results in \cite{Knud}, written here in the context of stable $I$-marked curves is the following:
\begin{theorem}[\cite{Knud}, Theorem 2.7]\label{moduli}The functor $\CM_I$ associating to a scheme $S$ the set of isomorphism classes of stable $I$-marked curves over $S$ is represented by a scheme $M_I$, which is smooth and proper over $\Spec\BZ$.
\end{theorem}
\begin{remark}Knudsen's result in \cite{Knud} actually states that the Deligne-Mumford moduli stack $\CM_I$ is smooth and proper over $\Spec\BZ$. For $I:=\{1,\ldots,n\}\subset\BN$, write $\CM_n:=\CM_I$. Knudsen proves that the universal stable $k$-pointed curve $\CZ_k$ is isomorphic to $\CM_{k+1}$. Moreover, if $\CM_k$ is representable by a scheme, then so is $\CZ_k$. To see that the stack $\CM_n$ is a scheme, it thus suffices by induction to observe that $\CM_3$ is representable by $\Spec \BZ$.
\end{remark}
We will often use the following consequence of the existence of a fine moduli scheme:
\begin{corollary}[Uniqueness of morphisms]\label{uniquemorph}For any stable $I$-marked curves $(C,\lambda)$ and $(D,\mu)$ over a scheme $S$, there exists at most one morphism of stable $I$-marked curves between them.
\end{corollary}

\begin{lemma}\label{fibiso} Let $X$ and $Y$ be schemes over $S$, and let $f\colon X\to Y$ be an $S$-morphism. Assume
\begin{enumerate}[label=(\alph*)] 
\item $X$ is locally of finite presentation over $S$,
\item $X$ is flat over $S$,
\item $Y$ is locally of finite type over $S$, and
\item for all $s\in S$, the induced morphism $f_s\colon X_s\to Y_s$ on fibers is an isomorphism. 
\end{enumerate}
Then $f$ is an isomorphism. 
\end{lemma}
\begin{proof} Conditions (a) and (c) together imply that $f$ is locally of finite presentation (\cite[Tag 02FV]{Stacks}). In particular, the morphism $f$ is locally of finite type, so we may apply Proposition 17.2.6 in \cite{EGAIV}, which says that $f$ is a monomorphism if and only if for all $y\in Y$, the fiber $f^{-1}(y)$ is empty or isomorphic to $\Spec k(y)$. Fix $y\in Y$, and let $s$ be the image of $y$ in $S$. Since $f_s$ is an isomorphism, we have $f^{-1}(y)\cong f_s^{-1}(y)\cong \Spec k(y)$, and thus $f$ is a monomorphism and in fact bijective. We next apply the \emph{crit\`ere de platitude par fibre} (\cite{Stacks}, Tag 039A), which says that if (a)-(c) hold and $f_s$ is flat for all $s\in S$, then $f$ is flat. 
A flat monomorphism that is locally of finite presentation is an open immersion (\cite[Theorem 17.9.1]{EGAIV}). Since a bijective open immersion is an isomorphism, this concludes the proof.
\end{proof}
Let $(C,\lambda)$ and $(D,\mu)$ be $I$-marked curves over a scheme $S$, and let $f\colon C\to D$ be an $S$-morphism such that $f\circ\lambda_i=\mu_i$ for all $i\in I$.
Consider the set
 \begin{equation}\label{dim1fiber}Z:=\{d\in D\mid \dim f^{-1}(d)=1\}.
 \end{equation}
 For each $s\in S$, we also define 
\begin{equation}\label{1-eq-dim1fiber}Z_s:=\{d\in D_s\mid \dim f_s^{-1}(d)=1\}.\end{equation}
\begin{proposition}\label{1-prop-Zclosedfinite}The set $Z$ is closed in $D$ and finite over $S$ when endowed with the induced reduced subscheme structure. Moreover, for every $s\in S$, we have $Z\cap D_s=Z_s$.
\end{proposition}
\begin{proof}We first show that $Z$ is a closed subset of $D$. According to \cite[Corollary 13.1.5]{EGAIV}, given a proper morphism of schemes $g\colon X\to Y$, the function $Y\to\BZ,\,\,y\mapsto \dim X_y$ is upper semicontinuous.
Since $C$ is proper over $S$, the morphism $f$ is proper, and we deduce by upper semicontinuity that the set of points $d\in D$ such that $\dim f^{-1}(d)\geq 1$ is closed. The dimension of $f^{-1}(d)$ is bounded above by 1 for all $d\in D$, so this set is precisely $Z$. Hence $Z$ is closed. 
\vspace{1mm}

Let $s\in S$ and let $d\in D$ be a point lying over $s$. Since $f_s^{-1}(d)\cong f^{-1}(d)$, we immediately deduce that $Z\cap D_s=Z_s$.
We claim that $Z_s$ is a finite set. This follows easily from the fact that each irreducible component of $C_s$ maps onto an irreducible closed subset of $D_s$, i.e., an irreducible component or a closed point of $D_s$. The set $Z_s$ consists precisely of the images of irreducible components of $C_s$ mapping to a closed point. Since $C_s$ has finitely many irreducible components, the finiteness of $Z_s$ follows.
\vspace{1mm}

 Endow $Z$ with the induced reduced scheme structure inherited from $D$. Since $D$ is proper over $S$, so is $Z$. Moreover, the fact that each $Z_s$ is finite implies that $Z$ is quasi-finite over $S$. Since a proper morphism is finite if and only if it is quasi-finite (\cite[Corollary 12.89]{GW}), we conclude that $Z$ is finite over $S$. 
\end{proof}
\begin{lemma}\label{connfibers}Let $Y\into D$ be a closed subscheme which is finite over $S$, and suppose that $f$ induces an isomorphism $C\setminus f^{-1}(Y)\stackrel\sim\to D\setminus Y$. Let $s\in S$ and let $\os$ be a geometric point centered at $s$. The fiber $f_\os^{-1}(y)$ is connected for every $y\in Y_\os$.
\end{lemma}
\begin{proof}Suppose $y\in Y_\os$ and that $f_\os^{-1}(y)$ is disconnected. Since $C_\os$ is connected, there exists a unique chain of irreducible components in $C_\os\setminus f^{-1}_\os(y)$ connecting any two connected components of $f^{-1}_\os(y)$. The image of any such chain forms a loop in $D_\os$ (viewed as a graph) at $y$ which contradicts the fact that $D_\os$ has genus 0.
\end{proof}
\begin{proposition}\label{morphprops}The following are equivalent:
\begin{enumerate}[label=(\alph*)]
\item There exists a closed subscheme $Y$ of $D$ which is finite over $S$ such that $f$ induces an isomorphism $C\setminus f^{-1}(Y)\stackrel\sim\to D\setminus Y$.
\item For each $s\in S$, the morphism $f_s$ induces an isomorphism $C_s\setminus f_s^{-1}(Z_s)\stackrel\sim\to D_s\setminus Z_s$.
\item The morphism $f$ induces an isomorphism $C\setminus f^{-1}(Z)\stackrel\sim\to D\setminus Z.$
\end{enumerate}
\end{proposition}
\begin{proof}
\vspace{2mm}Since we aim to apply Lemma \ref{fibiso}, we first show that $C$ and $C'$ are $S$-schemes of finite presentation. Both $C$ and $C'$ are obtained via base change from schemes of finite type over a noetherian base by Theorem \ref{moduli}. Since finite type and finite presentation are equivalent over a noetherian base, and we conclude by observing that finite presentation is preserved under base change.
\vspace{1mm}

(a)$\Rightarrow$(b): Let $s\in S$. The assumption in (a) implies that $f_s$ induces an isomorphism $C_s\setminus f^{-1}_s(Y_s)\stackrel\sim\to D_s\setminus Y_s$; hence $Z_s\subset Y_s$. If $Z_s=Y_s$, then (b) immediately follows. Suppose $Z_s\subsetneq Y_s$. It suffices to show that $f_s$ induces an isomorphism $C_\os\setminus f_\os^{-1}(Z_\os)\stackrel\sim\to D_\os\setminus Z_\os$, where $\os$ is a geometric point centered at $s$. Let $d\in Y_\os\setminus Z_\os.$ Then $f_\os^{-1}(d)$ must be $0$-dimensional by the definition of $Z_s$. By Lemma \ref{connfibers}, the fiber $f_\os^{-1}(d)$ is connected and thus consists of a single point. Let $U_\os:=C_\os\setminus f_\os^{-1}(Y_\os\setminus\{d\})$. Since $f_\os^{-1}(d)$ is a point, by applying \cite[Proposition 17.2.6]{EGAIV} in the same way as in the proof of Lemma \ref{fibiso}, we deduce that the restriction 
$f_\os|_{U_\os}\colon U_\os\to D_\os\setminus (Y_\os\setminus \{d\})$ is a monomorphism. It is also proper and hence a (bijective) closed immersion (\cite[Corollary 12.92]{GW}). Since $C_\os$ and $D_\os$ are reduced, it follows that $f_\os|_{U_\os}$ is an isomorphism. Iterating this process for every point in $Y_\os\setminus Z_\os$, we find that $f_\os$ restricts to an isomorphism $C_\os\setminus f_\os^{-1}(Z_\os)\stackrel\sim\to D_\os\setminus Z_\os$, as desired.
\vspace{2mm}

(b)$\Rightarrow$(c): This follows directly from Lemma \ref{fibiso} and the fact that $Z_s=Z\cap D_s$ by Proposition \ref{1-prop-Zclosedfinite}.
\vspace{2mm}

(c)$\Rightarrow$(a): Endowing $Z$ with the induced reduced subscheme structure inherited from $D$, we deduce from Proposition \ref{1-prop-Zclosedfinite} that (a) holds with $Y:=Z$ .
\end{proof}
\begin{definition}We call a morphism $f\colon (C,\lambda)\to(D,\mu)$ satisfying any of the equivalent properties in Proposition \ref{morphprops} a \emph{morphism of $I$-marked curves}.
\end{definition}
We observe that the identity is a morphism of $I$-marked curves and that the composite of two morphisms of $I$-marked curves is again a morphism of $I$-marked curves.
\begin{proposition}\label{1-prop-category-stable-curves}Let $S$ be a scheme. The $I$-marked curves over $S$ and the morphisms between them form a category.
\end{proposition}
In order to simplify the proof of the next proposition, we introduce the concept of a dual graph:
\begin{definition}\label{1-def-dualgraph}Let $(C,\lambda)$ be a stable $I$-marked curve over an algebraically closed field. The \emph{dual graph} $\Gamma_C$ of $C$ is defined to have
\begin{enumerate}
\item a vertex for each irreducible component of $C$;
\item an edge for each node of $C$, joining the corresponding vertices;
\item a labeled half edge for each $I$-marked point, emanating from the corresponding vertex.
\end{enumerate}
\end{definition}
The graph $\Gamma_C$ is a connected tree and the definition of stability translates to each vertex having degree $\geq 3$. Given a non-empty closed subset $E\subset C$ such that every connected component of $E$ is 1-dimensional, we define the \emph{subgraph $\Gamma_E$ generated by $E$} to be the subgraph of $\Gamma_C$ with vertices corresponding to the irreducible components of $E$ and whose edges consist of the edges of $\Gamma_C$ incident to those vertices. In particular, the graph $\Gamma_E$ has a half edge for each irreducible component of $E$ that intersects an irreducible component of $\overline{C\setminus E}$. We call the half edges in $\Gamma_E$ the \emph{external edges} and denote the number of external edges in $\Gamma_E$ by $n_{E}^{\mathrm{ext}}$. We call the remaining edges of $\Gamma_E$ the \emph{internal edges}, and we call a vertex that has at most one internal edge incident to it a \emph{leaf}. A leaf of $\Gamma_C$ corresponds to a irreducible component of $F$ of $C$ such that $|F\cap \overline{C\setminus F}|\leq 1$. We call such components the $\emph{tails}$ of $C$.
\begin{figure}[H]
\centering
  \caption{A stable $\{1,2,3,4\}$-marked curve $C$ and its dual graph. The closed subset $E$ and the subgraph $\Gamma_E$ generated by it are represented by thick lines.}
\begin{tikzpicture}

    \draw[color=blue, very thick](0,-0)  .. controls (1,.25-0) and (3,.25-0) .. 
  node[at end,right] {\color{black}\scriptsize$E$} 
  (4,-0) 
  node[color=black,circle,pos=0.9, fill=black,inner sep=0pt,minimum size=2pt, label={above:{\color{black}\scriptsize$\lambda_4$}}]{}
   node[color=black,circle,pos=0.5, fill=black,inner sep=0pt,minimum size=2pt, label={above:{\color{black}\scriptsize$\lambda_3$}}]{};


  \draw (.75,-0.25) .. controls (1,.25-0) and (1,0.75-0) .. (.75,1.75-0)
node[at end,left] {\scriptsize$C$}
node[color=black,circle,pos=0.9, fill=black,inner sep=0pt,minimum size=2pt, label={right:{\scriptsize$\lambda_1$}}]{}
node[color=black,circle,pos=0.7, fill=black,inner sep=0pt,minimum size=2pt, label={right:{\scriptsize$\lambda_2$}}]{}
  ;

\draw[->,
line join=round,
decorate, decoration={
    zigzag,
    segment length=4,
    amplitude=.9,post=lineto,
    post length=2pt}]
 (2,-.5)--(2,-1.10);

\node at (.65,-1.25) {\scriptsize$\Gamma_C$};
\node at (4,-2.25) {\scriptsize$\Gamma_E$};

\begin{scope}[yshift=-1]
\draw (1.5,-2)--(2,-2);
\draw[color=blue, very thick] (2,-2)--(2.5,-2);
\node[color=black,circle, fill=black,inner sep=0pt,minimum size=4pt] at (1.5,-2) {};
\node[color=blue,circle, fill=blue,inner sep=0pt,minimum size=4pt] at (2.5,-2) {};
\draw (1.5,-2)--(.75,-1.5)
 node[at end,left] {\scriptsize$\lambda_1$} ;
\draw (1.5,-2)--(.75,-2.5)
node[at end,left] {\scriptsize$\lambda_2$};
\draw[color=blue, very thick] (2.5,-2)--(3.25,-1.5)
node[at end,right] {\color{black}\scriptsize$\lambda_3$};
\draw[color=blue, very thick] (2.5,-2)--(3.25,-2.5)
node[at end,right] {\color{black}\scriptsize$\lambda_4$};
\end{scope}

\end{tikzpicture}
  \label{fig:dualgraph}
  \end{figure}
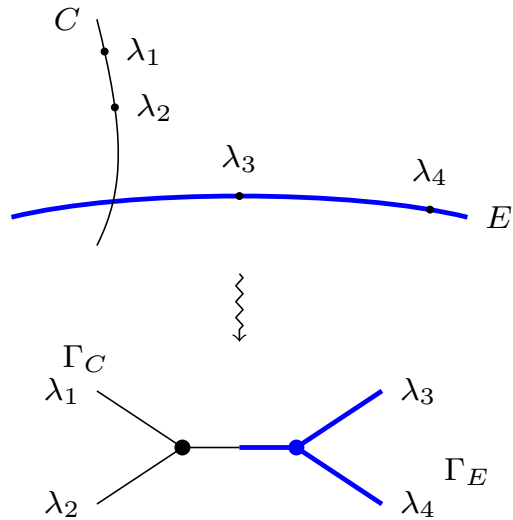

\begin{lemma}\label{1-lem-number-external-edges}For every $E$, we have $n_E^{\mathrm{ext}}\geq 3$.
\end{lemma}
\begin{proof}We proceed by induction on the number of irreducible components $m$ of $E$. If $m=1$, then the external edges of $\Gamma_E$ correspond precisely to the special points on $E$. The lemma then follows directly from the definition of a stable $I$-marked curve. Suppose $m>1$. Let $F\subset E$ be an irreducible component. Then $\Gamma_E=\Gamma_{\overline{E\setminus{F}}}\cup \Gamma_F\subset\Gamma_C$. By the induction hypothesis, both $\Gamma_{\overline{E\setminus{F}}}$ and $\Gamma_F$ have at least three external edges. Since all of the graphs are trees, the graphs $\Gamma_{\overline{E\setminus{F}}}$ and $\Gamma_F$ have at most one shared edge. It follows that $n_E^{\mathrm{ext}}\geq n_{\overline{E\setminus F}}^{\mathrm{ext}}+n_F^{\mathrm{ext}}-2\geq 4$ by induction.
\end{proof}

\begin{proposition}\label{stableiso}
Let $f\colon (C,\lambda)\to (D,\mu)$ be a morphism of $I$-marked curves over $S$. If $C$ and $D$ are both stable, then $f$ is an isomorphism.
\end{proposition}
\begin{proof}By Lemma \ref{fibiso}, the morphism $f$ is an isomorphism if and only if $f_s$ is an isomorphism for all $s\in S$, which in turn holds if and only if $f_\os$ is an isomorphism for any geometric point $\os$ centered on $s$. We may thus assume without loss of generality that $S=\Spec k$ for an algebraically closed field $k$. Suppose $Z\subset D$ as defined in (\ref{dim1fiber})  is non-empty and consider $d\in Z$. Let $\CE$ be the set of irreducible components in $\overline{C\setminus f^{-1}(d)}$ which have non-empty intersection with $f^{-1}(d)$. Since $f$ is an isomorphism over a deleted neighborhood of $d$, it follows that the elements of $\CE$ map isomorphically onto distinct irreducible components of $D$ which all intersect at the point $d$. Since $D$ has at worst nodal singularities and $f$ is non-constant, it follows that $0\leq|\CE|\leq 2$.

Suppose $|\CE|=2$. Consider the subgraph $\Gamma_{f^{-1}(d)}\subset\Gamma_C$. Each external edge of $\Gamma_{f^{-1}(d)}$ corresponds either to the intersection of an irreducible component of $f^{-1}(d)$ with an irreducible component of $\overline{C\setminus f^{-1}(d)}$ or to an $I$-marked point contained in $f^{-1}(d)$. Since $|\CE|=2$, there are exactly 2 external edges of the first kind, and it follows from Lemma \ref{1-lem-number-external-edges} that $f^{-1}(d)$ contains a marked point. But then the corresponding marked point of $D$ is the singular point $d$, which is not allowed. If $|\CE|=1$, then Lemma \ref{1-lem-number-external-edges} implies that $f^{-1}(d)$ contains at least two marked points, which again correspond to $d\in D$. Since the marked points of $D$ are required to be distinct, this is a contradiction. Finally, if $|\CE|=0$, then $f$ is constant, which is similarly prohibited. Thus $Z$ is empty and $f$ is an isomorphism, as desired.
\end{proof}

Consider a morphism $T\to S$, and let $(D,\mu)$ be an $I$-marked curve over $T$. We define a \emph{morphism} from $(D,\mu)\to (C,\lambda)$ to be a commutative diagram
\begin{equation}\label{morph}\begin{gathered}\xymatrix{D\ar[d]\ar[r]& C\ar[d]\\
T\ar[r]& S}\end{gathered}
\end{equation}
such that the induced morphism $D\to C\times_S T$ is a morphism of $I$-marked curves.
\begin{corollary}\label{cartesian}If $(D,\mu)$ and $(C,\lambda)$ are both stable, then (\ref{morph}) is cartesian.
\end{corollary}
\begin{remark}Proposition \ref{stableiso} and Corollary \ref{cartesian} show that when $(D,\mu)$ and $(C,\lambda)$ are both stable then our notion of a morphism of stable $I$-marked curves is identical to that of \cite{Knud}, where a morphism is defined to be a cartesian diagram of the form in (\ref{morph}) respecting the $I$-markings of $D$ and $C$. Our results also demonstrate that the category of stable $I$-marked curves over $S$ forms a full subcategory of the category of $I$-marked curves over $S$.
\end{remark}

\subsection{Contractions}\label{stablecontractions}
Note that for stable $I$-marked curves to exist, we must have $|I|\geq 3$. Let $(C,\lambda)$ be a stable $I$-marked curve over a scheme $S$. Consider a subset $I'\subset I$ with $|I'|\geq 3$.
\begin{definition}\label{cntrct}
Let $(C',\lambda')$ be an $I$-marked curve over $S$, together with a morphism $f\colon (C,\lambda)\to (C',\lambda')$. We call $C'$ a \emph{contraction of $C$ with respect to $I'$} if the pair $(C',\lambda'|_{I'})$ is a stable $I'$-marked curve.
\end{definition}
Intuitively, the contraction is obtained by viewing $C$ as an $I'$-marked curve and then (on geometric fibers) contracting irreducible components containing fewer than three special points until one obtains a stable $I'$-marked curve. This is indeed the definition given by Knudsen in \cite{Knud}, where contractions are only explicitly defined when $I'=I\setminus\{i\}$ for some $i\in I$. The following proposition shows that the two definitions are equivalent.

\begin{proposition}\label{equivdefs}Let $S=\Spec k$ for an algebraically closed field $k$, and let $I':=I\setminus\{i\}$ for a fixed $i\in I$. Let $(C,\lambda)$ and $(C',\lambda')$ be $I$-marked curves, and suppose that $(C,\lambda)$ and $(C',\lambda'|_{I'})$ are stable. Consider a morphism $f\colon C\to C'$ of schemes over $S$ satisfying $f\circ\lambda_i=\lambda'_i$ for all $i\in I$. Then $C'$ is a contraction if and only if
\begin{enumerate}[label=(\alph*)]
\item $C$ is stable as an $I'$-marked curve and $f$ is an isomorphism, or
\item $C$ is not stable as an $I'$-marked curve and $f$ sends the irreducible component $E_i$ of $C$ containing the $i$-marked point to a point $c'\in C$ and induces an isomorphism $C\setminus E_i\stackrel\sim\to C'\setminus \{c'\}$.
\end{enumerate}
\end{proposition}
\begin{proof}Under our assumptions, the morphism $f$ is a contraction if and only if it satisfies any of the equivalent conditions in Proposition \ref{morphprops}. The ``if'' direction thus follows from the fact that (a) and (b) both imply condition (b) in Proposition \ref{morphprops}.
\vspace{2mm}

For the converse, suppose $f\colon (C,\lambda)\to (C',\lambda')$ is a contraction. If $C$ is stable as an $I'$-marked curve, then $f$ is an isomorphism by Proposition \ref{stableiso}, so (i) holds. If $C$ is not stable, we must show that the closed subset $Z\subset C'$ as defined in (\ref{dim1fiber}) consists exclusively of the point $c'_i:=\lambda'_i(S)$ and that $f^{-1}(c'_i)=E_i$. Let $c'\in Z$ and consider the set $\CE$ of irreducible components of $\overline{C\setminus f^{-1}(c')}$ which have non-empty intersection with $f^{-1}(c')$. By the same reasoning as in the proof of Proposition \ref{stableiso}, we must have $0<|\CE|\leq 2$.

 If $|\CE|=2$, then, by the same reasoning as in the proof of Proposition \ref{stableiso}, the stability of $C$ implies that the fiber $f^{-1}(c')$ contains an $I$-marked point. We claim that the only marked point in $f^{-1}(c')$ is the $i$-marked point. Indeed, in this case $c'\in C'$ is a nodal singularity in $C'$. Since $C'$ is stable, the $I'$-marked points are smooth by definition. The fiber $f^{-1}(c')$ therefore cannot contain an $I'$-marked point, which yields the claim. The stability of $C$ then implies that $f^{-1}(c')$ is irreducible and hence equal to $E_i$, as desired. 
 
 By similar reasoning, if $|\CE|=1$, the fiber $f^{-1}(c')$ must contain precisely two marked points, and these correspond to $i$ and some $i'\in I'$. The stability of $C$ again implies that $f^{-1}(c')$ is irreducible and hence equal to $E_i$.
\end{proof}

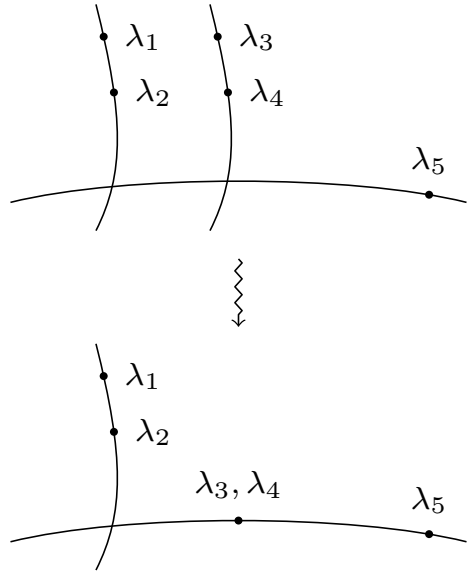
\begin{figure}[H]
\centering
  \caption{Example of a contraction for $I:=\{1,2,3,4,5\}$ and $I':=\{1,2,3,5\}$.}
\begin{tikzpicture}
  \draw(0,0)  .. controls (1,.25) and (3,.25) .. 
  (4,0) 
  node[color=black,circle,pos=0.9, fill=black,inner sep=0pt,minimum size=2pt, label={above:{\scriptsize$\lambda_5$}}]{};
  
    \draw(0,-3)  .. controls (1,.25-3) and (3,.25-3) .. 
  (4,-3) 
  node[color=black,circle,pos=0.9, fill=black,inner sep=0pt,minimum size=2pt, label={above:{\scriptsize$\lambda_5$}}]{}
   node[color=black,circle,pos=0.5, fill=black,inner sep=0pt,minimum size=2pt, label={above:{\scriptsize$\lambda_3,\lambda_4$}}]{};

  \draw (.75,-0.25) .. controls (1,.25) and (1,0.75) .. (.75,1.75)
node[color=black,circle,pos=0.9, fill=black,inner sep=0pt,minimum size=2pt, label={right:{\scriptsize$\lambda_1$}}]{}
node[color=black,circle,pos=0.7, fill=black,inner sep=0pt,minimum size=2pt, label={right:{\scriptsize$\lambda_2$}}]{}
  ;
    \draw (1.75,-0.25) .. controls (2,.25) and (2,0.75) .. (1.75,1.75)
node[color=black,circle,pos=0.9, fill=black,inner sep=0pt,minimum size=2pt, label={right:{\scriptsize$\lambda_3$}}]{}
node[color=black,circle,pos=0.7, fill=black,inner sep=0pt,minimum size=2pt, label={right:{\scriptsize$\lambda_4$}}]{}    
    ;

  \draw (.75,-3.25) .. controls (1,.25-3) and (1,0.75-3) .. (.75,1.75-3)
node[color=black,circle,pos=0.9, fill=black,inner sep=0pt,minimum size=2pt, label={right:{\scriptsize$\lambda_1$}}]{}
node[color=black,circle,pos=0.7, fill=black,inner sep=0pt,minimum size=2pt, label={right:{\scriptsize$\lambda_2$}}]{}
  ;

\draw[->,
line join=round,
decorate, decoration={
    zigzag,
    segment length=4,
    amplitude=.9,post=lineto,
    post length=2pt}]
 (2,-.5)--(2,-1.10);

\end{tikzpicture}
  \label{fig:contraction}
  \end{figure}
\begin{proposition}\label{1-prop-uniquenesscontractions}Let $(C,\lambda)$ be a stable $I$-marked curve, and let $I'$ be a subset of $I$ with $|I'|\geq 3.$ There exists a contraction $f\colon(C,\lambda)\to (C',\lambda')$ of $C$ with respect to $I'$. The tuple $(C',\lambda',f)$ is unique up to unique isomorphism.
\end{proposition}
\begin{proof}
The existence and uniqueness of contractions when $I'=I\setminus\{i\}$ for some $i\in I$ is proven in \cite[Proposition 2.1]{Knud}. Iterating this (using that the composite of morphisms of $I$-marked curves is again such a morphism) and noting that the result is independent of the order in which we forget sections \cite[Lemma 10.6.10]{Arb},  yields the desired contraction for general $I'$.
\end{proof}
Let $(C',\lambda')$ be the contraction of $(C,\lambda)$ with respect to $I'$. For each $i\in I$, we denote the image of the $i$-marked section by $Z_i\subset C'$. Since $C'$ is separated over $S$, each $Z_i$ is closed in $C'$. Consider
$$Z_{(I\setminus I')}:=\bigcup_{i\in (I\setminus I')}Z_i.$$ The following property of contractions will be used repeatedly in later sections.

\begin{proposition}\label{contriso}The contraction morphism $f\colon C\to C'$ induces an isomorphism $$C\setminus f^{-1}(Z_{(I\setminus I')})\stackrel\sim\to C'\setminus Z_{(I\setminus I')}.$$ 
\end{proposition}
\begin{proof}
Let $s\in S$, and let $c'\in C'_s$ be a point with $\dim f_s^{-1}(c')=1$. We consider the set of irreducible components of $C_s\setminus f_s^{-1}(c')$ which intersect $f_s^{-1}(c')$ and apply the same argument as in the proof
of Proposition \ref{equivdefs} to deduce that $f_s^{-1}(c')$ contains an $(I\setminus I')$-marked point. It follows that $Z\subset C'$ as defined in (\ref{dim1fiber}) is contained in $Z_{(I\setminus I')}$, and the proposition follows by the definition of a morphism of $I$-marked curves.
\end{proof}

\subsubsection*{Irreducible components of stable $I$-marked curves}
The following lemma and proposition give a nice consequence of the existence of contractions. Let $(C,\lambda)$ be a stable $I$-marked curve over $\Spec k$, where $k$ is a field, and fix an algebraic closure $\overline{k}$ of $k$. 
\begin{lemma}\label{1-lem-irred-comps}
Every irreducible component of $C$ is geometrically irreducible. 
\end{lemma}
\begin{proof}
We proceed by induction on $|I|$. If $|I|=3$, then it follows directly from the definition of stability that the base change $C_{\overline k}$ must be isomorphic to $\BP^1_{\overline k}$; hence $C$ is geometrically irreducible. Suppose $|I|=n>3$. Let $i\in I$ and consider the contraction $C'$ of $C$ with respect to  $I':=I\setminus\{i\}.$ Let $Z\subset C'$ be as defined in (\ref{dim1fiber}). Since $|I'|=n-1$, it follows from the induction hypothesis that every irreducible component of $C'$ is geometrically irreducible. If $Z=\emptyset$, then the contraction morphism $f\colon C\to C'$ is an isomorphism by definition, so we deduce the same for the irreducible components of $C$. Otherwise, the base change $Z_{\overline k}$ must consist of a single point, whose image in $C'$ we denote by $c'$ (so that $Z=\{c'\}$). Then $f^{-1}(c')_{\overline k}$ is irreducible, so $f^{-1}(c')$ is geometrically irreducible. By definition, the morphism $f$ induces an isomorphism $C\setminus f^{-1}(c')\stackrel\sim\to C'\setminus \{c'\}$. The induction hypothesis then implies that every irreducible component of $C\setminus f^{-1}(c')$ is geometrically irreducible. Since $C=\bigr(C\setminus f^{-1}(c')\bigl)\cup f^{-1}(c')$, this proves the lemma.
\end{proof}

\begin{proposition}\label{P1fibs}Let $E\subset C$ be an irreducible component. Then $E\cong\BP^1_k$.
\end{proposition}
\begin{proof}If $E$ contains an $I$-marked point, then it is isomorphic to $\BP^1_k$ because any connected smooth projective curve of genus $0$ over $\Spec k$ possessing a $k$-rational point is isomorphic to $\BP^1_k$. Suppose $E$ contains no $I$-marked points. The base change $E_{\overline k}$ is irreducible by Lemma \ref{1-lem-irred-comps} and also contains no $I$-marked points. It follows that $E_{\overline{k}}$ contains as least three singular points. This implies that $C_{\overline k}\setminus E_{\overline k}$ has at least three connected components. Since $C$ is stable, we can choose a marked point on each of these components. The chosen points correspond to some $I':=\{i,j,k\}\subset I$. Let $C'$ denote the contraction of $C$ with respect to $I'$. Then $C'$ contains a marked point and is thus isomorphic to $\BP^1_k$. The contraction morphism $f\colon C\to C'$ then induces a morphism $E\to C'\cong\BP^1_k$. This becomes an isomorphism after base change to $\overline k$ and is hence itself an isomorphism.
\end{proof}
\begin{remark}It follows from Proposition \ref{P1fibs} that we may replace every instance of the geometric fiber in Definitions \ref{stablecurve} and \ref{cntrct} by the (scheme-theoretic) fiber, and we will often do so in the remainder of the text.
\end{remark}
%
%
\subsection{Contraction to the $i$-component}\label{contractionicomp}
Let $(C,\lambda)$ be a stable $I$-marked curve over a scheme $S$, and fix an element $i\in I$. In this subsection we construct an $I$-marked curve $(C^i,\lambda^i)$ along with a morphism $f\colon C\to C^i$ which, on fibers, contracts the irreducible components of $C$ not meeting the $i$-marked point.
We first formalize what we mean for a morphism to contract a given set of irreducible components.
\begin{definition} Suppose $S=\Spec k$, where $k$ is a field.  Let $\SE$ be a set of irreducible of components of $C$  such that $$\bigcup_{E\in\SE} E\subsetneq C.$$ Let $C'$ be an $I$-marked curve over $k$. We say that a morphism $f\colon C\to C'$ \emph{contracts} $\SE$ if $f(E)$ is a closed point in $C'$ for every $E\in\SE$, and $f$ induces an isomorphism from $C\setminus \bigl(\bigcup_{E\in\SE}E\bigl)$ to $C'\setminus f\bigl(\bigcup_{E\in\SE}E\bigl)$.
\end{definition}

For a morphism $X\to S$ of schemes, a \emph{relative effective Cartier divisor} is an effective Cartier divisor on $X$ which is flat over $S$ when regarded as a closed subscheme of $X$ (\cite[Definition 3.4]{Mil}). In what follows, we do not distinguish between an effective Cartier divisor and the associated closed subscheme. 

\begin{lemma}The closed subscheme $D_i:=\lambda_i(S)$ of $C$ is a relative effective Cartier divisor on $C/S$.
\end{lemma}
\begin{proof}The scheme $D_i$ is clearly flat over $S$. By the definition of stable $I$-marked curves, for each $s\in S$, the fiber $D_{i,s}$ corresponds to a smooth point of $C_s$ and is hence an effective Cartier divisor on $C_s$. The lemma then follows directly from \cite[Proposition 3.8]{Mil}, which says that if $X/S$ is flat and $D$ is a closed subscheme of $X$ that is flat over $S$ and such that $D_s$ is an effective Cartier divisor on the fiber $X_s$ for all $s\in S$, then $D$ is a relative effective Cartier divisor on $X/S$. 
\end{proof}
Suppose $S=\Spec k$, where $k$ is a field.
\begin{lemma}\label{ggs}We have the following:
\begin{enumerate}[label=(\alph*)]
\item $h^0\bigl(C,\CO_C(D_i)\bigr)=2$,
\item $H^1\bigl(C,\CO_C(D_i)\bigr)=0$,
\item $\CO_C(D_i)$ is generated by its global sections.
\end{enumerate}
\end{lemma}
\begin{proof}
For the first statement, let $f_i$ denote the closed embedding $D_i\into C$, and consider the short exact sequence
\begin{equation}\label{ses}0\to\CO_C\to\CO_C(D_i)\to (f_i)_*k\to0.
\end{equation}
Since $H^1(C,\CO_C)=0$ by assumption, the first part of the associated long exact sequence on cohomology reads
$$\xymatrix@C=1em@R=1em{0\ar[r]& H^0(C,\CO_C)\ar[r]& H^0\bigl(C,\CO_C(D_i)\bigr)\ar[r]& H^0\bigl(C, (f_i)_*k\bigr)\ar[r]&0}.$$
Since $H^0(C,\CO_C)$ and $H^0\bigl(C,(f_i)_*k\bigr)\cong H^0(D_i,k)$ are both isomorphic to $k$, it follows that $h^0\bigl(C,\CO_C(D_i)\bigr)=2$, as desired. 
\vspace{2mm}

For (b), we observe that the long exact sequence on cohomology induced by (\ref{ses}) also yields
$$0\to H^1\bigl(C,\CO_{C}(D_{i})\bigr)\to H^1\bigl(C,(f_i)_*k\bigr)\to\ldots.$$
Since $H^1(C,(f_i)_*k)\cong H^1(D_i, k)$ and the latter vanishes because $D_i$ is zero-dimensional, we conclude that $ H^1\bigl(C,\CO_{C}(D_{i})\bigr)=0$.
\vspace{2mm}

The third statement follows from the exact sequence (\ref{ses}) and the fact that $\CO_C$ and $(f_i)_*k$ are both generated by their global sections: We have a commutative diagram
$$\xymatrix@C=1em@R=1em{0\ar[r]&\CO_C\ar[r]\ar[d]_f&\CO_C\oplus\CO_C\ar[r]\ar[d]_h&\CO_C\ar[r]\ar[d]_g&0\\
0\ar[r]&\CO_C\ar[r]& \CO_C(D_i)\ar[r]&  (f_i)_*k\ar[r]&0}$$
with exact rows and such that $f$ and $g$ are surjective. It follows from the Five Lemma that $h$ is also surjective; hence $\CO_C(D_i)$ is generated by global sections.
\end{proof}
\vspace{2mm}

We now return to the case where $S$ is an arbitrary scheme. Let $\pi\colon C\to S$ denote the structure morphism.
\begin{lemma}\label{pushfwd}We have the following:
\begin{enumerate}[label=(\alph*)]
\item The sheaf $\pi_*\CO_C(D_i)$ is locally free of rank $2$.
\item The adjunction morphism $\pi^*\pi_*\CO_C(D_i)\to\CO_C(D_i)$ is surjective.
\end{enumerate}
\end{lemma}
\begin{proof}$\,$

\noindent\emph{Reduction to noetherian base:}  By Theorem \ref{moduli}, the stable $I$-marked curve $C$ is isomorphic to the base change of the universal stable $I$-marked curve $C_I$ over $M_I$, which is noetherian. Let $D_{i,I}$ denote the effective Cartier divisor on $C_I$ corresponding to the $i$-marked section. Corollary 1.5 of \cite{Knud} says that if the base scheme is noetherian, then condition (c) of Lemma \ref{ggs} implies that the formation of $\pi_*\CO_C(D_i)$ commutes with base change. Thus, given the cartesian diagram
\begin{equation}\label{univpullback}\begin{gathered}\xymatrix{C\ar[r]^g\ar[d]^{\pi}&C_I\ar[d]^{\rho}\\S\ar[r]^f&M_I,}\end{gathered}\end{equation}
there is a natural isomorphism $$f^*\rho_*\CO_{C_I}(D_{i,I})\cong \pi_*g^*\CO_{C_I}(D_{i,I}).$$ Since $g^*\CO_{C_I}(D_{i,I})\cong \CO_{C}(D_i)$, we deduce that $\pi_*\CO_C(D_i)$ is locally free of rank 2 if $\rho_*\CO_{C_I}(D_{i,I})$ is. We also have a chain of natural isomorphisms 
$$\pi^*\pi_*\CO_{C}(D_i)\cong\pi^*\pi_*g^*\CO_{C_I}(D_{i,I})\cong \pi^*f^*\rho_*\CO_{C_I}(D_{i,I})\cong g^*\rho^*\rho_*\CO_{C_I}(D_{i,I}).
$$
The adjunction morphism $\pi^*\pi_*\CO_C(D_i)\to\CO_C(D_i)$ is thus obtained by applying $g^*$ to the adjunction morphism $\rho^*\rho_*\CO_{C_I}(D_{i,I})\to\CO_{C_I}(D_{i,I})$. Since $g^*$ is right exact, surjectivity of the latter implies surjectivity of the former. It follows that if the lemma holds for the universal family $C_I$, then it holds for $C$ as well. We may thus assume that $S$ is noetherian.
\vspace{2mm}

\noindent\emph{Proof for noetherian base:} Suppose $S$ is noetherian. For any closed point $s\in S$, the pullback of $\CO_C(D_i)$ to the fiber $C_s$ is equal to $\CO_{C_{s}}\bigl(D_{i,s}\bigr)$, and $D_{i,s}$ is the effective Cartier corresponding to the $i$-marked point on $C_s$. By Lemma \ref{ggs}, the following two conditions hold:
\begin{enumerate}
\item[(i)] For every closed point $s\in S$, we have $$H^1\bigl(C_s,\CO_{C_s}\bigl(D_{i,s}\bigr)\bigr)=0.$$
\item[(ii)] For every closed point $s\in S$, the sheaf $\CO_{C_s}\bigl(D_{i,s}\bigr)$ is generated by its global sections.
\end{enumerate}
The discussion in \cite[Chp. 0\S5]{Mum}, says that if condition (i) holds, then $\pi_*\CO_C(D_i)$ is locally free. The fact that the rank is $2$ then follows directly from Lemma \ref{ggs}. Corollary 1.5 of \cite{Knud} says that if both (i) and (ii) hold, then the adjunction morphism $\pi^*\pi_*\CO_C(D_i)\to\CO_C(D_i)$ is surjective, as desired.
\end{proof}
Let $\pi\colon C\to S$ denote the structure morphism, and define $$C^i:=\underline{\Proj}\bigl(\Sym\pi_*\CO_C(D_i)\bigr).$$ 
We say that a scheme over $S$ is a \emph{$\BP^1$-bundle} if it is Zariski locally isomorphic over $S$ to $\BP^1_S$. By Lemma \ref{pushfwd}.1, the direct image $\pi_*\CO_C(D_i)$ is locally free of rank $2$; hence $C^i$ is a $\BP^1$-bundle over $S$. Morphisms $C\to C^i$ over $S$ correspond to surjective morphisms of sheaves $\pi^*\pi_*\CO_C(D_i)\to \CL$, where $\CL$ is an invertible sheaf on $C$ (\cite[II.7.12]{Ha}). 
Lemma \ref{pushfwd}.2 thus implies that the adjunction homomorphism $\pi^*\pi_*\CO_C(D_i)\to\CO_C(D_i)$ induces a natural $S$-morphism 
$$f^i\colon C\to C^i.$$ Endowing $C^i$ with the $I$-marking $\lambda^i:=f^i\circ\lambda$ makes $(C^i,\lambda^i)$ into an $I$-marked curve.

\begin{proposition}\label{1-prop-characterizationCi}We have the following:
\begin{enumerate}
\item[(a)]for each $s\in S$, the morphism $f^i_s\colon C_s\to C^i_s$ contracts all irreducible components not meeting $D_{i,s}$;
\item[(b)]$f^i$ is a morphism of $I$-marked curves.
\end{enumerate}
\end{proposition}
\begin{proof}For (a), we may assume that $S=\Spec k$ for a field $k$. Let $i\colon E\into C$ be an irreducible component. By \cite[Lemma 8.3.29]{Liu}, the image of $E$ under $f^i$ is a point if and only if $i^*\CO_{C}(D_i)\cong \CO_E$. This occurs precisely when $E$ does not meet $D_i$. Suppose $E$ is the unique irreducible component of $C$ meeting $D_i$. Since $f^i(E)$ is closed, not a equal to a point, and $C^i\cong\BP^1_k$ is irreducible, we must have $f^i(E)=C^i$. Thus $f^i$ induces a finite morphism $E\to C^i.$ Since $D_i$ has degree $1$, this is an isomorphism, proving (a).

 Statement (a) implies that $f^i_s$ induces an isomorphism $C_s\setminus (f^i_s)^{-1}(Z_s)\stackrel\sim\to C^i_s\setminus Z_s,$ where $Z_s\subset C^i_s$ is as defined in (\ref{1-eq-dim1fiber}). Hence (a) implies (b).
\end{proof}

\begin{definition}\label{i-comp}We call $(C^i,\lambda^i)$ along with the morphism $f^i\colon (C,\lambda)\to (C^i,\lambda^i)$ the \emph{contraction to the $i$-component} of $(C,\lambda)$.
\end{definition}
We now relate the contraction to the $i$-component to the notion of contraction from the preceding subsection. For each $I':=\{i,j,k\}\subset I$ of cardinality $3$, define
\begin{equation}\label{Salpha}S_{I'}:=\{s\in S\mid\lambda^i(i)\neq\lambda^i(j)\neq\lambda^i(k)\neq\lambda^i(i)\}.\end{equation}
Since the morphism $C^i\to S$ is separated and $S_{I'}$ is defined to be the locus where finitely many sections are not equal, it follows that $S_{I'}$ is an open subscheme of $S$.
\begin{proposition}\label{loccontr}
The $S_{I'}$ form an open cover of $S$ as $I'$ varies. Moreover, the $I$-marked curve $C^{i}\times_S S_{I'}$ is the contraction of $C\times_S S_{I'}$ with respect to ${I'}$.
\end{proposition}
\begin{proof} Let $s\in S$, and let $E_i$ denote the irreducible component of $C_s$ containing the $i$-marked point. Since $C_s$ is stable, it satisfies exactly one of the following:
\begin{enumerate}
\item[(i)] The curve $C_s$ is smooth.
\item[(ii)] There are at least two connected components in $C_s\setminus E_i$, each containing at least two marked points.
\item[(iii)] The complement of $E_i$ is connected, contains at least two marked points, and there exists a marked point on $E_i$ distinct from the $i$-marked point.
\end{enumerate}
In case (i), let $j,k\in I\setminus\{i\}$ be arbitrary. In case (ii), choose $j$ and $k$ corresponding to marked points in distinct connected components of $C_s\setminus E^i$. In case (iii), take $j\neq i$ corresponding to a marked point on $E^i$ and $k$ corresponding to a marked point on any other irreducible component. In each case, for ${I'}:=\{i,j,k\}$, we have $s\in S_{I'}$. The $S_{I'}$ thus cover $S$, as desired. Since the $i$- and $j$- and $k$-marked sections of $C^{i}\times_S S_{I'}$ are distinct, it is stable as an $I'$-marked curve. The second statement in the lemma then follows directly from the definition of contractions.
\end{proof}
In fact, the contraction to the $i$-component $(C^i,\lambda^i)$ is uniquely determined by the fact that, on fibers, in contracts the irreducible components not meeting $D_i$.
\begin{corollary}\label{1-cor-uniquenessCi}Let $\overline{f}\colon (C,\lambda)\to(\overline{C},\overline\lambda)$ be a morphism of $I$-marked curves satisfying property (a) of Proposition \ref{1-prop-characterizationCi}. Then there exists a unique isomorphism $(C^i,\lambda^i)\stackrel\sim\to (\overline{C},\overline{\lambda})$. Moreover, the resulting diagram commutes:
$$\xymatrix{&C\ar[rd]^{\overline{f}}\ar[ld]_{f^i}&\\C^i\ar[rr]_\sim& &\overline C.}
$$
\end{corollary}
\begin{proof}
By Proposition \ref{loccontr}, the $I$-marked curve $(C^i,\lambda^i)$ is locally a contraction of $C$ with respect to $\{i,j,k\}\subset I$ for some $j,k\in I$. The exact same argument shows that the same holds for $(\overline{C},\overline\lambda)$. The statement then follows directly from the uniqueness of contractions (Proposition \ref{1-prop-uniquenesscontractions}).
\end{proof}
For each $j\in I$, we denote the image of $\lambda^i(j)$ by $Z_j\subset C^i$. 
\begin{corollary}Let $Z_{i^c}:=\bigcup_{j\in I\setminus\{i\}}Z_j$. The morphism $f^i$ induces an isomorphism $$C\setminus (f^i)^{-1}(Z_{i^c})\stackrel\sim\to C^i\setminus Z_{i^c}.$$
\end{corollary}
\begin{proof}This follows directly from Propositions \ref{contriso} and \ref{loccontr}.
\end{proof}

We conclude this subsection by showing that the construction of the contraction to the $i$-component commutes with base change. For any morphism $\phi\colon S'\to S$ we consider the following pullback diagram:
\begin{equation}\label{basechangediag}
\begin{gathered}
\xymatrix{C':=\phi^*C\ar[rr]^-{\phi'}\ar[d]^{\pi'}&&C\ar[d]^{\pi}\\S'\ar[rr]^\phi&&S.}
\end{gathered}
\end{equation}
Let $D_i':=(\phi')^*D_i$.
\begin{lemma}\label{basechangeiso}The formation of $\pi_*\CO(D_i)$ commutes with base change, i.e., there is a natural isomorphism $(\pi')_*\CO_C(D_i')\cong \phi^*\pi_*\CO_C(D_i).$
\end{lemma}
\begin{proof}
We proved the lemma in the case where $S$ is noetherian in the proof of Lemma \ref{pushfwd}. For general $S$, consider the unique morphisms $g\colon S\to M_I$ and $g'\colon S'\to M_I$ and $h\colon C\to C_I$ and $h'\colon  C'\to C_I$ yielding commutative diagrams of the form (\ref{univpullback}). By uniqueness, we must have $g'=g\circ \phi$ and $h'=h\circ \phi'$. We again use $D_{i,I}$ to denote the image of the $i$-marked section of $C_I$. Since $M_I$ is noetherian, we can apply the compatibility with base change in the noetherian case to obtain the chain of natural isomorphisms
\begin{eqnarray*}
(\pi')_*\CO_C(D_i')&\cong&(\pi')_*(h')^*\CO_{C_I}(D_{i,I})\cong
\\
g'^*\rho_*\CO_{C_I}(D_{i,I})&\cong& \phi^*g^*\rho_*\CO_{C_I}(D_{i,I})\cong\\
\phi^*\pi_*h^*\CO_{C_I}(D_{i,I})&\cong& \phi^*\pi_*\CO_C(D_i).
\end{eqnarray*}
\end{proof}
\begin{proposition}[Compatibility with base change]\label{basechangeCi}There is an isomorphism 
$(C')^i\stackrel\sim\to (C^i)':=\phi^*(C^i)$ of $I$-marked curves over $S'$ such that the following diagram commutes:
\begin{equation}\label{AA}\begin{gathered}\xymatrix{&C'\ar[rd]\ar[ld]&\\(C')^i\ar[rr]_\sim& &(C^i)'.}\end{gathered}\end{equation}
\end{proposition}
\begin{proof} The isomorphism of sheaves from Lemma \ref{basechangeiso} yields an isomorphism $(C')^i\stackrel\sim\to (C^i)'$ over $S'$. The commutativity of (\ref{AA}) is equivalent to the commutativity of the following diagram:
$$\xymatrix{&\CO_{C'}(D_i')&\\(\pi')^*(\pi')_*\CO_{C'}(D_i')\ar[ur]^{(1)}& &(\phi')^*\pi^*\pi_*\CO_C(D_i)\ar[ul]_{(2)}\ar[ll]_{(3)}^\sim.}$$
Here (1) is the adjunction homomorphism, and (2) is $(\phi')^*$ applied to the adjunction composed with the natural isomorphism $(\phi')^*\CO_C(D_i)\stackrel\sim\to \CO_{C'}(D_i').$ The homomorphism (3) is from Lemma \ref{basechangeiso}, and the diagram commutes by the naturality of (1)-(3). Finally, we note that the $I$-markings on $(C')^i$ and $(C^i)'$ are both induced by the $I$-marking on $C'$; hence (\ref{AA}) is a diagram of $I$-marked curves. This proves the proposition.
\end{proof}

\subsection{Stabilizations}\label{stab}

The stabilization construction provides an inverse to the contraction obtained by forgetting a single section (\cite[Corollary 2.6]{Knud}). Fix an $I$-marked curve $(C,\lambda)$ and $i\in I$ such that $C$ is stable as an $(I\setminus\{i\})$-marked curve.
\begin{definition} A \emph{stabilization} of $(C,\lambda)$ is a stable  $I$-marked curve $(C^s,\lambda^s)$ with a morphism $f\colon C^s\to C$ making $C$ a contraction of $C^s$ with respect to $I\setminus\{i\}$.
\end{definition}

\begin{proposition}[\cite{Knud}, Theorem 2.4] There exists a stabilization of $(C,\lambda)$ and it is unique up to unique isomorphism.
\end{proposition}
 Note that stabilization does \textbf{not} provide a way to obtain a unique stable $I$-marked curve from an $I$-marked curve that is stable as an $I'$-marked curve for arbitrary $I'\subset~I$. The result of iterating the stabilization process depends at each step on how the remaining sections are lifted. See the figure below:
\begin{figure}[H]
\centering
  \caption{Example for $I=\{1,\ldots,5\}$. After stabilizing with respect to $\lambda_4$, a lift of $\lambda_5$ is either distinct from $\lambda_3$ and $\lambda_4$ or equal to one of them. In the first case, the resulting stabilization is already a stable $I$-marked curve. In the latter, we must stabilize further to obtain a stable $I$-marked curve.}
\begin{tikzpicture}
  \draw[name path=E, opacity=0](0,0)  .. controls (1,.25) and (3,.25) .. 
  (4,0) 
  node[color=black,circle,pos=0.9, fill=black,inner sep=0pt,minimum size=2pt, label={above:{\scriptsize$\lambda_5$}}]{};

  \draw[name path=D] (.75,-0.25) .. controls (1,.25) and (1,0.75) .. (.75,1.75)
node[color=black,circle,pos=0.9, fill=black,inner sep=0pt,minimum size=2pt, label={right:{\scriptsize$\lambda_1$}}]{}
node[color=black,circle,pos=0.7, fill=black,inner sep=0pt,minimum size=2pt, label={right:{\scriptsize$\lambda_2$}}]{}
  ;

     
\draw[->,
line join=round,
decorate, decoration={
    zigzag,
    segment length=4,
    amplitude=.9,post=lineto,
    post length=2pt}]
 (2,-.5)--(2,-1.10);
 
 \draw[->,
line join=round,
decorate, decoration={
    zigzag,
    segment length=4,
    amplitude=.9,post=lineto,
    post length=2pt}]
 (2+5,-.5)--(2+5,-1.10);
 
 \draw[->,
line join=round,
decorate, decoration={
    zigzag,
    segment length=4,
    amplitude=.9,post=lineto,
    post length=2pt}]
 (3.5,1)--(4.5,1);
 
 
 \draw[name intersections={of=D and E}] 
     (intersection-1) node[circle, draw, color=black, fill=black,inner sep=0pt,minimum size=2pt,label={right:{\scriptsize$\lambda_3,\lambda_4,\lambda_5$}}]{};
     
    \draw(5,0)  .. controls (1+5,.25) and (3+5,.25) .. 
  (9,0) 
   node[color=black,circle,pos=0.5, fill=black,inner sep=0pt,minimum size=2pt, label={above:{\scriptsize$\lambda_3$}}]{}
     node[color=black,circle,pos=0.7, fill=black,inner sep=0pt,minimum size=2pt, label={above:{\scriptsize$\lambda_4,\lambda_5$}}]{};
     
       \draw (.75+5,-.25) .. controls (1+5,.25) and (1+5,0.75) .. (5.75,1.75)
node[color=black,circle,pos=0.9, fill=black,inner sep=0pt,minimum size=2pt, label={right:{\scriptsize$\lambda_1$}}]{}
node[color=black,circle,pos=0.7, fill=black,inner sep=0pt,minimum size=2pt, label={right:{\scriptsize$\lambda_2$}}]{}
  ;
    
    \draw(5,0-3)  .. controls (1+5,.25-3) and (3+5,.25-3) .. 
  (9,0-3) 
   node[color=black,circle,pos=0.5, fill=black,inner sep=0pt,minimum size=2pt, label={above:{\scriptsize$\lambda_3$}}]{}
    ;
     
       \draw (.75+5,-.25-3) .. controls (1+5,.25-3) and (1+5,0.75-3) .. (5.75,1.75-3)
node[color=black,circle,pos=0.9, fill=black,inner sep=0pt,minimum size=2pt, label={right:{\scriptsize$\lambda_1$}}]{}
node[color=black,circle,pos=0.7, fill=black,inner sep=0pt,minimum size=2pt, label={right:{\scriptsize$\lambda_2$}}]{}
  ;
  
  \draw (.75+7,-.25-3) .. controls (1+7,.25-3) and (1+7,0.75-3) .. (7.75,1.75-3)
node[color=black,circle,pos=0.9, fill=black,inner sep=0pt,minimum size=2pt, label={right:{\scriptsize$\lambda_4$}}]{}
node[color=black,circle,pos=0.7, fill=black,inner sep=0pt,minimum size=2pt, label={right:{\scriptsize$\lambda_5$}}]{}
  ;
    
    \draw(0,-3)  .. controls (1,.25-3) and (3,.25-3) .. 
  (4,-3) 
  node[color=black,circle,pos=0.9, fill=black,inner sep=0pt,minimum size=2pt, label={above:{\scriptsize$\lambda_5$}}]{}
   node[color=black,circle,pos=0.5, fill=black,inner sep=0pt,minimum size=2pt, label={above:{\scriptsize$\lambda_3$}}]{}
     node[color=black,circle,pos=0.7, fill=black,inner sep=0pt,minimum size=2pt, label={above:{\scriptsize$\lambda_4$}}]{};
     
       \draw (.75,-3.25) .. controls (1,.25-3) and (1,0.75-3) .. (.75,1.75-3)
node[color=black,circle,pos=0.9, fill=black,inner sep=0pt,minimum size=2pt, label={right:{\scriptsize$\lambda_1$}}]{}
node[color=black,circle,pos=0.7, fill=black,inner sep=0pt,minimum size=2pt, label={right:{\scriptsize$\lambda_2$}}]{}
  ;

\end{tikzpicture}
  \label{fig:stabilization}
  \end{figure}
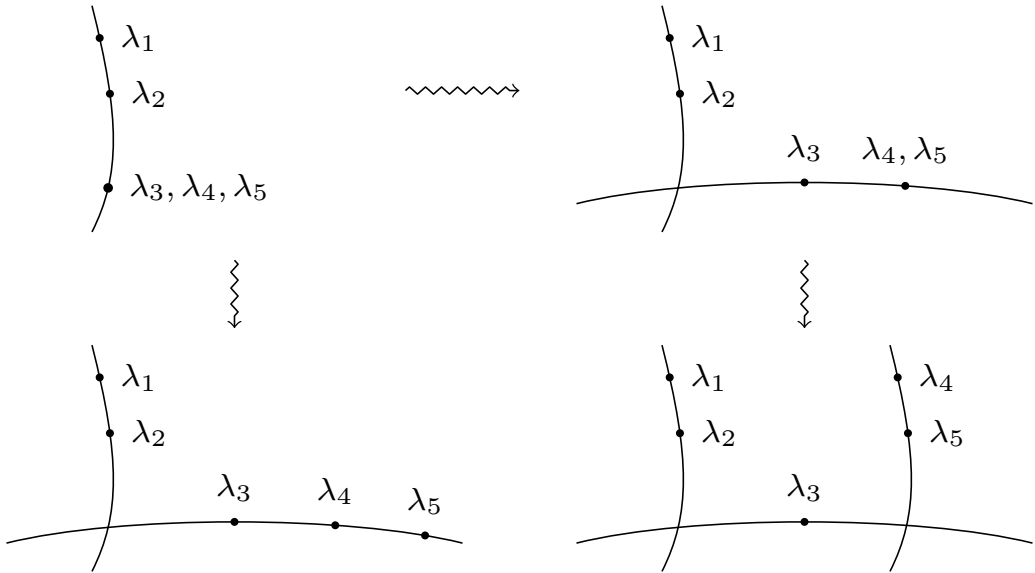


\section{The scheme $B_V$ and some properties}\label{OVBV}
For the remainder of the text, we work within the category $\SchFq$ of schemes over $\Spec\BF_q$. All tensor products and fiber products will be taken over $\BF_q$ unless otherwise stated. 
Let $V$ be a finite dimensional $\BF_q$-vector space. In this section we recall the definition of $B_V$ and list several of its properties for later use. 
 Most of these facts are collected from \cite{P-Sch}, and the corresponding proofs can be found there.
\subsection{Preliminaries on $P_V$ and $\OV$}
The \emph{symmetric algebra} $S_V:=\Sym(V)$ of $V$ over $\BF_q$ is the quotient of the tensor algebra 
$T(V):=\bigoplus_{i=0}^\infty V^{\otimes i}$ by the ideal generated by all elements of the form $v\otimes w-w\otimes v$, where $v,w \in V$. Any choice of ordered basis of $V$ yields an isomorphism $S_V\stackrel\sim\to \BF_q[X_1,\ldots, X_n].$ Let $RS_V$ denote the localization of $S_V$ obtained by inverting all $v\in V\setminus\{0\}$. We then have
$$\OV\cong\Proj RS_V\cong \Spec RS_{V,0},$$
where $RS_{V,0}$ denotes the degree $0$ part of $RS_V$.
\begin{definition}\label{fiberwise} Let $S$ be a scheme and let  $\CL$ be an invertible sheaf on $S$. Consider a set $I$ and a map $\lambda\colon I\to \CL(S)$. If, for all $s\in S$, the composite map $$\xymatrix@C-=0.5cm{I\ar[r]^-\lambda&\CL(S)\ar[r]&\CL\otimes_{\CO_S}k(s)}$$ is
\begin{enumerate}
\item non-zero, we call $\lambda$ \emph{fiberwise non-zero,}
\item injective, we call $\lambda$ \emph{fiberwise injective.}
\end{enumerate}
\end{definition}
The scheme $P_V=\Proj(S_V)$ admits a natural fiberwise non-zero $\BF_q$-linear map 
$$\lambda_V\colon V\stackrel\sim\to\CO_{P_V}(1)(P_V).$$
\begin{proposition}[\cite{P-Sch}, Proposition 7.8-7.9]
$\,$
\begin{enumerate}[label=(\alph*)]
\item The scheme $P_V$ with the universal family $(\CO_{P_V}(1),\lambda_V)$ represents the functor which associates to a scheme $S$  over $\BF_q$ the set of isomorphism classes of pairs $(\CL,\lambda)$ consisting of an invertible sheaf $\CL$ on $S$ and a fiberwise non-zero $\BF_q$-linear map $\lambda\colon V\to\CL(S)$.
\item The open subscheme $\OV\subset P_V$ represents the subfunctor of fiberwise injective linear maps.
\end{enumerate}
\end{proposition}
Let $S$ be a scheme. Given a fiberwise non-zero linear map $\lambda\colon V\to \CL(S),$ we obtain a surjection of coherent sheaves 
$$\lambda\otimes\id\colon V\otimes\CO_S\onto\CL,$$ whose kernel we denote by $\CE_V$.
\begin{proposition}[\cite{P-Sch}, Proposition 10.1]\label{PVrep}The scheme $P_V$ represents the functor which to an $\BF_q$-scheme $S$ associates the set of coherent subsheaves $\CE_V\subset V\otimes\CO_S$ such that $(V\otimes\CO_S)/\CE_V$ is invertible.
\end{proposition}
\subsection{The scheme $B_V$}
Let $V'$ run over all non-zero $\BF_q$-subspaces of $V$. The product $\prod_{0\neq V'\subset V}P_{V'}$ represents tuples $\CE_\bullet=(\CE_{V'})_{V'}$ such that $(V'\otimes \CO_S)/\CE_{V'}$ is invertible for each $V'$.
\begin{proposition}[\cite{P-Sch}, \S 10]\label{BVrep} There exists a unique closed subscheme 
$$B_V\subset \prod_{0\neq V'\subset V}P_{V'}$$ representing the subfunctor of all $\CE_\bullet$ satisfying the closed condition
\begin{eqnarray}\label{BVcond}
\CE_{V''}\subset\CE_{V'}&\mbox{ for all $0\neq V''\subset V'\subset V.$}
\end{eqnarray}
\end{proposition}
Slightly abusing notation, we also denote the functor $S\mapsto \CE_\bullet$ from Propositon \ref{BVrep} by $B_V$. For later use, we give the following alternative interpretation of $B_V$. Let $$\underline{\CB}_V\colon \mathop{\rm \underline{Sch}}\nolimits_{\BF_q}^{\op}\to \mathop{\rm \underline{Set}}$$
be the functor sending an $\BF_q$-scheme $S$ to the set of isomorphism classes of tuples $(\CL,\lambda,\psi)$ where $$\CL:=(\CL_{V'})_{0\neq V'\subset V}$$ is a family of invertible sheaves over $S$ and $\lambda:=(\lambda_{V'})_{0\neq V'\subset V}$ is a family of fiberwise non-zero $\BF_q$-linear maps 
$$\lambda_{V'}\colon V\to \CL_{V'}(S)$$
 and $\psi:=(\psi^{V'}_{V''})_{0\neq V''\subset V'\subset V}$ is a family of homomorphisms $$\psi^{V'}_{V''}\colon \CL_{V''}\to\CL_{V'}$$ such that the following diagram commutes:
\begin{equation*}\label{BVfunc}\xymatrix{V'\ar[r]^-{\lambda_{V'}}&\CL_{V'}(S)\\
V''\ar@{^{(}->}[u]\ar[r]^-{\lambda_{V''}}&\CL_{V''}(S)\ar[u]_{\psi_{V''}^{V'}(S)}.}\end{equation*}
Two such tuples $(\CL,\lambda,\psi)$ and $(\CM,\mu,\omega)$ are isomorphic if for all subspaces $0\neq V'\subset V$ there are isomorphisms $\CL_{V'}\stackrel\sim\to\CM_{V'}$ which are compatible with $\lambda_{V'}$ and $\mu_{V'}$. The data of an isomorphism class $[(\CL,\lambda,\psi)]$ and a tuple $\CE_\bullet$ satisfying (\ref{BVcond}) are equivalent, which yields the following lemma.
\begin{lemma}\label{AltRep}The scheme $B_V$ represents the functor $\underline{\CB}_V$.
\end{lemma}
%

\subsection{The open subscheme $U_\CF$ and the boundary stratum $\Omega_\CF$}
Recall that we defined a flag of $V$ to be a set $\{V_0,\ldots, V_m\}$ of subspaces of $V$ such that $V_0=\{0\}$ and $V_m=V$ and $V_i\subsetneq V_j$ whenever $i<j$. For each flag $\CF$ of $V$, it is shown in \cite[\S 10]{P-Sch} that there exists a unique open subscheme $U_\CF\subset B_V$ representing the subfunctor satisfying (\ref{BVcond}) and the open condition 
\begin{equation}\label{UFRep}\begin{array}{cl}
V'\otimes \CO_S=\CE_{V'}+(V''\otimes\CO_S)&\mbox{ for all $0\neq V''\subset V'\subset V$ such that there}
\\ &\mbox{ exists no $W\in \CF$ with $V''\subset W$ and $V'\not\subset W$.}
\end{array}
\end{equation}
\begin{lemma}[\cite{P-Sch}, Lemma 10.7]\label{contains} For any two flags $\CF'\subset\CF$ of $V$, we have $U_{\CF'}\subset U_\CF$.
\end{lemma}
\begin{proposition}[\cite{P-Sch}, Corollary 10.10]\label{covering}
For varying $\CF$, the $U_\CF$ form an open covering of $B_V$. The $U_\CF$ where $\CF$ is complete form an open subcovering.
\end{proposition}

There is also a unique closed subscheme $B_\CF\subset B_V$ representing the subfunctor of $\CE_\bullet$ satisfying (\ref{BVcond}) and the closed condition
\begin{equation}
\begin{array}{cl}
V''\otimes\CO_S\subset\CE_{V'}&\mbox{ for all $0\neq V''\subset V'\subset V$ such that there}
\\ &\mbox{ exists $W\in \CF$ with $V''\subset W$ and $V'\not\subset W$.}
\end{array}
\end{equation}

Consider the locally closed subscheme $\Omega_\CF:=U_\CF\cap B_\CF$. For varying $\CF$, the $\Omega_\CF$ stratify $B_V$ in the following sense:
\begin{proposition}[\cite{P-Sch}, Theorem 10.8, Corollary 10.10]\label{strata}We have the following equalities of underlying sets:
\begin{eqnarray*}
B_V&=&\bigsqcup_{\CF}\Omega_\CF,\\
B_\CF&=&\bigsqcup_{\CF\subset\CF'}\Omega_{\CF'},\\
U_\CF&=&\bigsqcup_{\CF'\subset\CF}\Omega_{\CF'}.
\end{eqnarray*}

\end{proposition}
\begin{proposition}[\cite{P-Sch}, Proposition 10.12] For the trivial flag $\CF_0:=\{0,V\}$, we have $\Omega_{\CF_0}=U_{\CF_0}$ and a natural isomorphism 
$$\Omega_{\CF_0}\stackrel\sim\to\OV,\,\CE_\bullet\mapsto\CE_V.$$
\end{proposition}
We identify $\Omega_{\CF_0}$ with $\Omega_V$ via this isomorphism.
Thus $U_\CF$ contains $\OV$ as an open subscheme by Lemma \ref{contains}. 
\subsection{Flag subquotients and some natural isomorphisms}
Let $\CF:=\{V_0,\ldots, V_m\}$ be a flag of $V$.
\begin{proposition}[\cite{P-Sch}, Proposition 10.18]\label{grafting-subquot} There is a natural isomorphism
$$B_\CF\stackrel\sim\to B_{V_1/V_0}\times\ldots \times B_{V_{m}/V_{m-1}},\,\,\CE_\bullet\mapsto\Bigl(\bigl(\CE_{V'}/(V_{i-1}\otimes\CO_S)\bigr)_{V'/V_{i-1}}\Bigr)^m_{i=1},$$
where $V'$ runs through all subspaces $V_{i-1}\subsetneq V'\subset V_i$.
\end{proposition}

\begin{proposition}[\cite{P-Sch}, Proposition 10.19]\label{stratum} The isomorphism from Proposition \ref{grafting-subquot} induces an isomorphism
$$\Omega_\CF\stackrel\sim\to \Omega_{V_1/V_0}\times\ldots \times\Omega_{V_{m}/V_{m-1}},\,\,\CE_\bullet\mapsto\bigl(\CE_{V_i}/(V_{i-1}\otimes\CO_S)\bigr)^m_{i=1}.$$
\end{proposition}
For later use (Proposition \ref{graftingmorph}), we describe the inverse of the isomorphism in Proposition \ref{grafting-subquot} explicitly in the case where $\CF:=\{0,V',V\}$. Let $\oV:=V/V'$. Let $S$ be a scheme and consider tuples $\CE'_\bullet$ and $\overline{\CE}_\bullet$ corresponding to $S$-valued points of $B_{V'}$ and $B_{\oV}$ respectively. Define a tuple $\CE_\bullet$ as follows: For each $W\subset V'$ we set $\CE_{W}:=\CE'_{W}$. For each $W\not\subset V'$ consider the natural surjection 
$$\pi_W\colon W\otimes\CO_S\onto (W+V')/V'\otimes\CO_S$$
 and define $\CE_W:=\pi^{-1}_W\big(\overline{\CE}_{(W+V')/V'}\big)$. By the proof of Proposition 10.18 in \cite{P-Sch}, the tuple $\CE_\bullet$ is an $S$-valued point of $B_V$ and the induced morphism $B_{V'}\times B_{\oV}\to B_V$ is a closed embedding with image equal to $B_\CF$.

\subsection{Affine coordinates on $U_\CF$}
Let $\CF$ be a flag, and choose a complete flag $\overline\CF=\{V_0,\ldots,V_n\}$ containing $\CF$. Fix an ordered basis $\CB:=(b_1,\ldots,b_n)$ of $V$ such that for all $1\leq j\leq n$ we have $V_j=\Span(b_1,\ldots,b_j)$.
\begin{definition}We call $\CB$ a \emph{flag basis associated to $\CF$}. 
\end{definition}
Consider an $S$-valued point $\CE_\bullet$ of $U_\CF$. For each $1\leq i\leq n$, define $U_i:=\BF_q\cdot b_i$. The natural inclusion $U_i\into V_i$ induces an isomorphism $U_i\otimes \CO_s\stackrel\sim\to(V_i\otimes\CO_S)/\CE_{V_i}.$ It follows from (\ref{UFRep}) that $V_i\otimes\CO_S=\CE_{V_i}\oplus (U_i\otimes\CO_S)$. Since $V_i=V_{i-1}\oplus U_i$, this shows that $\CE_{V_i}$ is the graph of a $\CO_S$-linear homomorphism $V_{i-1}\otimes\CO_S\to U_i\otimes\CO_S$. This homomorphism sends $b_{i-1}\otimes 1$ to $b_{i}\otimes t_{i-1}$ for a unique $t_i\in\CO_S(S)$. In other words, there is a unique section $t_i\in\CO_S(S)$ such that 
\begin{equation}\label{coord}b_{i-1}\otimes 1-b_{i}\otimes t_{i-1}\in\Gamma(S,\CE_{V_{i}}).\end{equation}
Varying $i=2,\ldots,n$, we thereby obtain a morphism 
\begin{eqnarray}\label{affcoord}U_\CF&\longto&\BA_{\BF_q}^{n-1},\\
\nonumber\CE_\bullet&\mapsto&(t_1,\ldots,t_{n-1}).
\end{eqnarray}
\begin{proposition}[\cite{P-Sch}, Proposition 10.14]\label{Coords}This is an open embedding.
\end{proposition}
On the other hand, the choice of $\CB$ also yields the composite morphism 
\begin{equation}\label{projcoord}\OV\into P_V\stackrel\sim\to\BP^{n-1}_{\BF_q}=\Proj\BF_q[X_1,\ldots X_n].\end{equation}
This is determined by the global sections $b_1,\ldots, b_n$ of the invertible sheaf $\CO_{\OV}(1).$ Since $\OV$ is the complement of the union of all $\BF_q$-hyperplanes in $P_V$, the image of $\OV$ in $\BP^{n-1}_{\BF_q}$ is contained in the open subscheme defined by $X_i\neq 0$ for all $i$. 
\begin{lemma}\label{coordchange} Identify $\OV$ with its image in $\BP^{n-1}_{\BF_q}$ via (\ref{projcoord}). The composite 
$$\xymatrix@C-=0.5cm{\OV\ar@{^{(}->}[r]&U_\CF\ar@{^{(}->}[r]^-{(\ref{affcoord})}&\BA_{\BF_q}^{n-1}}$$
 is determined by the tuple $(T_1,\ldots, T_{n-1}),$ where $T_i:=\frac{X_i}{X_{i+1}}\in\Gamma(\OV,\CO_\OV)$.
\end{lemma}
\begin{proof}
Let $\CE_\bullet$ be an $S$-valued point of $\OV\subset U_\CF$. For each $1\leq i\leq n-1$, the sheaf $\CE_{V_{i+1}}$ is the kernel of a homomorphism $\lambda\colon V_{i+1}\otimes\CO_S\onto \CL$, for some invertible sheaf $\CL$ on $S$. The composite $S\to\OV\into\BA^{n-1}_{\BF_q}$ corresponds to $(t_1,\ldots,t_n)\in\CO_S(S)^{n-1}$, and (\ref{coord}) implies that $\lambda(b_{i})-\lambda(b_{i+1})t_{i}=0$ for each $i$. Thus $t_{i}=\frac{\lambda(b_{i})}{\lambda(b_{i+1})}$. Since $\frac{\lambda(b_{i})}{\lambda(b_{i+1})}$ is the pullback to $S$ of $\frac{X_{i}}{X_{+1}}$, this proves the lemma.
\end{proof}
\begin{remark}Let $S$ be an $\BF_q$-scheme and $\CL$ an invertible sheaf on $S$. Let $\lambda\colon V\to\Gamma(S,\CL)$ be a fiberwise injective linear map corresponding to an $S$-valued point of $\OV$. The open embedding (\ref{affcoord}) sends $\lambda$ to the tuple $\bigl(\frac{\lambda_{b_1}}{\lambda_{b_2}},\ldots,\frac{\lambda_{b_{n-1}}}{\lambda_{b_n}}\bigr)$, while the open embedding (\ref{projcoord}) sends $\lambda$ to the tuple $(\lambda_{b_1},\ldots,\lambda_{b_n})$.
\end{remark}
Write $\CF=\{V_{i_0},\ldots,V_{i_m}\}$ with $\{0\}=V_{i_0}\subsetneq V_{i_1}\subsetneq\cdots\subsetneq V_{i_m}=V$.
\begin{lemma} \label{kpoints} The image of $\Omega_\CF$ in $ \BA_{\BF_q}^{n-1}=\Spec(\BF_q[T_1,\ldots T_{n-1}])$ under the morphism (\ref{affcoord}) is the locally closed subscheme defined by
 \begin{enumerate}[label=(\alph*)]
 \item $T_i=0$ for $i=i_1,\ldots i_{m-1}$,
 \item for each $1\leq k\leq m$, any non-trivial $\BF_q$-linear combination of the elements $$\Big\{\textstyle{\prod_{i=i_{k-1}+1}^{i_{k}-1}T_i,\prod_{i=i_{k-1}+2}^{i_{k}-1}T_i,\ldots,T_{i_{k}-1},1}\Big\}$$ is non-zero. 
 \end{enumerate}
\end{lemma}
\begin{proof} Let $\Omega_\CF^\sharp$ denote the locally closed subscheme of $\BA^{n-1}_{\BF_q}$ defined by (a) and (b). For each $1\leq k\leq m$, let $$W_k:=\Span(b_{i_{k-1}+1},\ldots, b_{i_k}).$$ Let $n_k:=\dim(W_k)$ and define $\Omega_{W_k}$ in analogy to $\OV$. The basis $\{ b_{i_{k-1}+1},\ldots, b_{i_k}\}$ of $\Omega_{W_k}$ determines an immersion $\Omega_{W_k}\into\BA^{n_k-1}_{\BF_q}$. Composing with the closed embedding $\BA^{n_k-1}_{\BF_q}\into\BA_{\BF_q}^{n-1}$ given by $$(T_1,\ldots,T_{n_k-1})\mapsto (0,\ldots, 0,\underbrace{T_1,\ldots,T_{n_k-1}}_{\mathclap{i_{k-1}+1,\ldots, i_k-1-\mbox{components}}},0,\ldots,0),$$ we obtain an immersion $\Omega_{W_k}\into \BA_{\BF_q}^{n-1}$. For varying $k$, these combine to an immersion $\Omega_{W_1}\times\cdots\times\Omega_{W_m}\into \BA^{n-1}_{\BF_q}$. By Proposition \ref{stratum}, we have a natural isomorphism $\Omega_{\CF}\cong\Omega_{W_1}\times\cdots\times\Omega_{W_m}$. The resulting composite $\Omega_{\CF}\cong\Omega_{W_1}\times\cdots\times\Omega_{W_m}\into \BA^{n-1}_{\BF_q}$ is equal to the immersion $\Omega_\CF\into U_\CF\into\BA_{\BF_q}^{n-1}$ determined by $\CB$. It follows that (a) holds on the image of $\Omega_\CF$.
\vspace{2mm}

Let $S$ be a scheme. Fix $k$ and let $\lambda_k\in\Omega_{W_k}(S)$. Since $\lambda_k$ is fiberwise injective, for all $s\in S$ the image in $k(s)$ of the set
$$\Bigl\{\frac{\lambda_k(b_{i_{k-1}+1})}{\lambda_k(b_{i_{k}})},\ldots, \frac{\lambda_k(b_{i_k-1})}{\lambda_k(b_{i_k})}\Bigr\}\subset\CO_S(S)$$
is $\BF_q$-linearly independent. For all $1\leq j\leq n_k$ we have
$$\frac{\lambda_k(b_{i_{k-1}+j})}{\lambda_k(b_{i_{k}})}=\prod_{i=i_{k-1}+j}^{i_{k}-1}\frac{\lambda_k(b_{i_{k-1}+j})}{\lambda_k(b_{i_{k-1}+j+1})}.$$ The $j$-th term on the left-hand side is precisely the $j$-th coordinate of the open embedding $\Omega_{W_k}\into\BA_{\BF_q}^{n_k-1}$. All together, this implies that the restriction of (\ref{affcoord}) to $\Omega_\CF$ factors through $\Omega_\CF^\sharp$.
\vspace{1mm}

On the other hand, given a scheme $S$ and a tuple $(t_1,\ldots, t_{n-1})\in\CO_S(S)^{n-1}$ satisfying (a) and (b), i.e., a morphism $S\to\Omega_\CF^\sharp$, the map 
$\lambda_k\colon W_k\to \CO_S(S)$ defined via $b_j\mapsto\prod_{i=j}^{i_k-1} t_i$ for $j=i_{k-1}+1,\ldots,i_k$ and extending by linearity is fiberwise injective and hence determines a morphism $S\to\Omega_{W_k}$. Varying $k$ and composing with the isomorphism $\Omega_{W_1}\times\cdots\times\Omega_{W_m}\stackrel\sim\to\Omega_\CF$, we obtain a morphism $\Omega_{\CF}^\sharp\to\Omega_\CF$ which is clearly inverse to the restriction of (\ref{affcoord}) to $\Omega_\CF$. Thus (\ref{affcoord}) induces an isomorphism $\Omega_\CF\stackrel\sim\to\Omega_\CF^\sharp$.
\end{proof}


 \section{$V$-ferns: definition and properties}
Recall that we defined $V$ to be a finite dimensional vector space over $\BF_q$ and $\hat V$ to be the set $V\cup\{\infty\}$.
\subsection{The group $G:=V\rtimes \BF_q^\times$}
We define $G$ to be the finite group $V\rtimes \BF_q^\times$. The identity element is given by $(0,1)$ and for $(v,\xi)$, $(w,\nu)\in G$, we have $(v,\xi)\cdot(w,\nu)=(\xi w+v,\xi\nu)$ and $(v,\xi)^{-1}=(-\xi^{-1}v,\xi^{-1})$. The group $G$ has a natural left action on the normal subgroup $V$ via $(v,\xi)\cdot w=\xi w+v.$ Defining $(v,\xi)\cdot\infty:=\infty$ for each $(v,\xi)\in G$ extends this to a left action on $\hat V$.
\subsection{Definition of a $V$-fern}Let $S$ be a scheme, and let $(C,\lambda)$ be a stable $\hat V$-marked curve over $S$. Let $s\in S$ and consider the fiber $C_s$.
\begin{definition} A \emph{chain} in $C_s$ is a sequence of pairwise distinct irreducible components $E_1,\ldots,E_m$ of $C_s$ such that $E_i\cap E_j\neq \emptyset$ if and only if $i=j\pm 1$. The integer $m$ is called the \emph{length} of the chain.
\end{definition}
Denote by $E_{\{0\}}\subset C_s$ the irreducible component containing $\lambda_0(s)$ and by $E_{\{\infty\}}\subset C_s$ the irreducible component containing $\lambda_\infty(s)$. There is a unique chain $$\SE_s:=E_{\{0\}}=:E_1,E_2\ldots,E_{m_s}:=E_{\{\infty\}}$$ connecting $E_{\{0\}}$ and $E_{\{\infty\}}.$
\begin{definition}We call $\SE_s$ the \emph{chain from $0$ to $\infty$} at $s$.
\end{definition}
Note that if two elements of $\SE_s$ have non-empty intersection, then they intersect at a single closed point of $C_s$.
\begin{definition}For each $1\leq i\leq m_s$, consider
$$\begin{array}{ccc}x_i:=\left\{\begin{array}{cl}
\lambda_0(s)&\mbox{if $i=1$,}\\
E_{i-1}\cap E_i&\mbox{otherwise,}
\end{array}\right.&,&y_i:=\left\{\begin{array}{cl}
\lambda_\infty(s)&\mbox{if $i=m_s$,}\\
E_{i}\cap E_{i+1}&\mbox{otherwise.}\end{array}\right.\end{array}$$
We call $x_i$ and $y_i$ the \emph{distinguished special points} of $E_i$.
\end{definition}
\begin{remark}
As points of $C_s$, we have $x_i=y_{i-1}$ for each $2\leq i\leq m$.
\end{remark}

\begin{proposition}The natural left action of $G$ on $V$ induces a right action on the set of $\hat V$-markings on the underlying curve $C$ as follows: For each $(v,\xi)\in G$ and $\hat V$-marking $\lambda$, define $\lambda\cdot(v,\xi)$ to be the $\hat V$-marking sending $w\in V$ to $\lambda_{(v,\xi)\cdot w}$ and $\infty$ to $\lambda_\infty$. \end{proposition}
\begin{proof}Let $(u,\eta)\in G$. To verify that this indeed defines a right action, we must check that $$\bigl(\lambda\cdot (v,\xi)\bigr)\cdot(u,\eta)=\lambda\cdot\bigl((v,\xi)\cdot(u,\eta)\bigr).$$ Let $w\in V$. Then by definition
$$\begin{array}{lllll}
\Bigl(\bigl(\lambda\cdot (v,\xi)\bigr)\cdot(u,\eta)\Bigr)_w&=&\bigl(\lambda\cdot (v,\xi)\bigr)_{(u,\eta)\cdot w}&=&\lambda_{(v,\xi)\cdot[(u,\eta)\cdot w]}\\&=&\lambda_{[(v,\xi)\cdot(u,\eta)]\cdot w}&=&\Bigl(\lambda\cdot\bigl((v,\xi)\cdot(u,\eta)\bigr)\Bigr)_w.
\end{array}$$
Thus
$$\Bigl(\bigl(\lambda\cdot (v,\xi)\bigr)\cdot(u,\eta)\Bigr)_w=\Bigl(\lambda\cdot\bigl((v,\xi)\cdot(u,\eta)\bigr)\Bigr)_w$$ for all $w\in \hat V$, as desired.
\end{proof}

We now define our main objects of study:
\begin{definition}\label{FernDef}
A \emph{$V$-fern} $(C,\lambda,\phi)$ over $S$ is a stable $\hat V$-marked curve $(C,\lambda)$ endowed with a faithful left group action 
$$\begin{array}{cccc}
\phi\colon &G&\longto& \Aut_S(C)\\
 &(v,\xi)&\mapsto& \phi_{v,\xi}
\end{array}$$
such that
\begin{enumerate}
\item for each $(v,\xi)\in G$, the automorphism $\phi_{v,\xi}$ of $C$ induces an isomorphism of $\hat V$-marked curves $(C,\lambda)\stackrel\sim\to\bigl(C,\lambda\cdot (v,\xi)\bigr)$.
\item for every $s\in S$ and $E_i\in\SE_s$, there exists an isomorphism $E_i\stackrel\sim\to\BP^1_{k(s)}$ with $x_i\mapsto (0:1)$ and $y_i\mapsto(1:0)$, such that the induced action of $\BF_q^\times$ on $\BP^1_{k(s)}$ is $$\xi\mapsto \begin{pmatrix}\xi &0\\0&1\end{pmatrix}\in\PGL_2(k(s)).$$
\end{enumerate}
\end{definition}
For $\phi_{v,\xi}$, we will write $\phi_v$ when $\xi=1$ and $\phi_\xi$ when $v=0$.
\begin{remark}Any two isomorphisms $E_i\to\BP^1_{k(s)}$ sending $x_i$ to $(0:1)$ and $y_i$ to $(1:0)$ differ by multiplication by some $\alpha\in k(s)^\times$. Condition (2) thus implies that $\BF_q^\times$ has the prescribed action under \textbf{any} such isomorphism.
\end{remark}
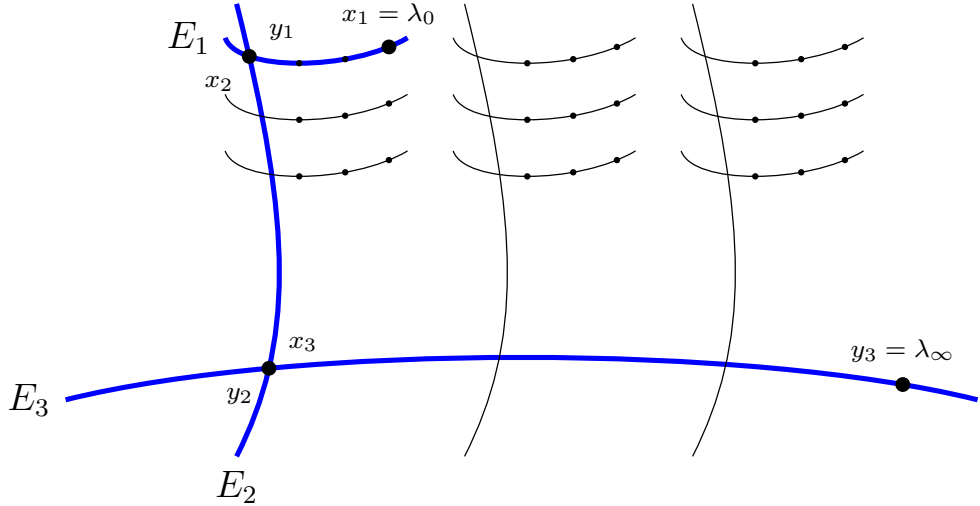
\begin{figure}[H]
\centering
  \caption{Example of a $V$-fern for $\dim V=3$ and $q=3$. The thick lines represent the chain from $0$ to $\infty$, and the large dots represent the distinguished special points.}
\begin{tikzpicture}[scale=2.5]
  \draw[ultra thick, color=blue, name path=E03] (0,0)  .. controls (1,.25) and (3,.25) .. node[at start,left] {{\color{black}$E_{3}$}} (4,0) 
  node[circle,pos=0.9,draw,color=black,thin, fill=black,inner sep=1.5pt,minimum size=2pt, label={[black]above:{\scriptsize$y_3=\lambda_\infty$}}]{};

  \draw[ultra thick, color=blue, name path=E02] (.75,-0.25) .. controls (1,.25) and (1,0.75) .. (.75,1.75)
  node[at start,below] {{\color{black}$E_{2}$}} ;
    \draw (1.75,-0.25) .. controls (2,.25) and (2,0.75) .. (1.75,1.75);
     \draw (2.75,-0.25) .. controls (3,.25) and (3,0.75) .. (2.75,1.75);

     \draw[ultra thick, color=blue, name path=E01] (.7,1.6) .. controls (.75,1.45) and (1.25,1.45) .. (1.5,1.6)
	node[at start,left] {{\color{black}$E_{1}$}}     	
     	 node[circle,pos=0.9, draw, thin, color=black,fill=black,inner sep=1.5pt,minimum size=2pt, label={[black]above:{\scriptsize$x_1=\lambda_0$}}]{}
       	node[circle,pos=0.7, fill=black,inner sep=0pt,minimum size=2pt,]{}
       node[circle,pos=0.5, fill=black,inner sep=0pt,minimum size=2pt, ]{};        
    \draw (.7,1.35) .. controls (.75,1.20) and (1.25,1.20) .. (1.5,1.35)
       		node[circle,pos=0.9, fill=black,inner sep=0pt,minimum size=2pt,]{}
       		node[circle,pos=0.7, fill=black,inner sep=0pt,minimum size=2pt,]{}
       		node[circle,pos=0.5, fill=black,inner sep=0pt,minimum size=2pt,]{};
     \draw (.7,1.1) .. controls (.75,.95) and (1.25,.95) .. (1.5,1.1)
     		node[circle,pos=0.9, fill=black,inner sep=0pt,minimum size=2pt,]{}
       		node[circle,pos=0.7, fill=black,inner sep=0pt,minimum size=2pt,]{}
       		node[circle,pos=0.5, fill=black,inner sep=0pt,minimum size=2pt,]{};;

         \draw (1.7,1.6) .. controls (1.75,1.45) and (2.25,1.45) .. (2.5,1.6)
         		node[circle,pos=0.9, fill=black,inner sep=0pt,minimum size=2pt,]{}
       		node[circle,pos=0.7, fill=black,inner sep=0pt,minimum size=2pt,]{}
       		node[circle,pos=0.5, fill=black,inner sep=0pt,minimum size=2pt,]{};
        \draw (1.7,1.35) .. controls (1.75,1.20) and (2.25,1.20) .. (2.5,1.35)
        		node[circle,pos=0.9, fill=black,inner sep=0pt,minimum size=2pt,]{}
       		node[circle,pos=0.7, fill=black,inner sep=0pt,minimum size=2pt,]{}
       		node[circle,pos=0.5, fill=black,inner sep=0pt,minimum size=2pt,]{};
        \draw (1.7,1.1) .. controls (1.75,.95) and (2.25,.95) .. (2.5,1.1)
        		node[circle,pos=0.9, fill=black,inner sep=0pt,minimum size=2pt,]{}
       		node[circle,pos=0.7, fill=black,inner sep=0pt,minimum size=2pt,]{}
       		node[circle,pos=0.5, fill=black,inner sep=0pt,minimum size=2pt,]{};
        
         \draw (2.7,1.6) .. controls (2.75,1.45) and (3.25,1.45) .. (3.5,1.6)
         		node[circle,pos=0.9, fill=black,inner sep=0pt,minimum size=2pt,]{}
       		node[circle,pos=0.7, fill=black,inner sep=0pt,minimum size=2pt,]{}
       		node[circle,pos=0.5, fill=black,inner sep=0pt,minimum size=2pt,]{};
        \draw (2.7,1.35) .. controls (2.75,1.20) and (3.25,1.20) .. (3.5,1.35)
        		node[circle,pos=0.9, fill=black,inner sep=0pt,minimum size=2pt,]{}
       		node[circle,pos=0.7, fill=black,inner sep=0pt,minimum size=2pt,]{}
       		node[circle,pos=0.5, fill=black,inner sep=0pt,minimum size=2pt,]{};
        \draw (2.7,1.1) .. controls (2.75,.95) and (3.25,.95) .. (3.5,1.1)
        		node[circle,pos=0.9, fill=black,inner sep=0pt,minimum size=2pt,]{}
       		node[circle,pos=0.7, fill=black,inner sep=0pt,minimum size=2pt,]{}
       		node[circle,pos=0.5, fill=black,inner sep=0pt,minimum size=2pt,]{};

		     \draw[name intersections={of=E02 and E01}] 
     (intersection-1) node[circle, draw, color=black, fill=black,inner sep=1.5pt,minimum size=2pt, label={[label distance=.01mm,black]30:{\scriptsize$y_1$}},label={[black]250:{\scriptsize$x_2$}}]{};
     
     \draw[name intersections={of=E03 and E02}] 
     (intersection-1) node[circle, draw, color=black, fill=black,inner sep=1.5pt,minimum size=2pt, label={[black]30:{\scriptsize$x_3$}},label={[black]250:{\scriptsize$y_2$}}]{};
\end{tikzpicture}
  \label{fig:chainofcomps}
  \end{figure}
A \emph{morphism of $V$-ferns} is simply a morphism of the underlying stable $\hat V$-marked curves. We say that two $V$-ferns over $S$ are \emph{isomorphic} if there exists a morphism between them.
\begin{proposition}
Any morphism of $V$-ferns is automatically $G$-equivariant.
\end{proposition}
\begin{proof}By the uniqueness of morphisms of stable $\hat V$-marked curves (Corollary \ref{uniquemorph}), we must have $\psi_{v,\xi}\circ\pi\circ\phi_{v,\xi}^{-1}=\pi$ for all $(v,\xi)\in G$.
\end{proof}
\begin{remark}Proposition \ref{1-prop-category-stable-curves} implies that $V$-ferns and the morphisms between them form a category. The forgetful functor $(C,\lambda,\phi)\mapsto (C,\lambda)$ identifies the category of $V$-ferns over a base scheme $S$ with a full subcategory of the category of stable $\hat V$-marked curves over $S$.
\end{remark}

Let $(C,\lambda,\phi)$ be a $V$-fern over $S$. We will often simply write $C$ for the tuple $(C,\lambda,\phi)$. Let $f\colon S'\to S$ be a morphism. The fiber product $C\times_S S'$ has a natural $V$-fern structure $(C\times_S S',\lambda',\phi')$, where $\lambda':=(\lambda\circ f)\times \id_{S'}$ and $\phi'$ is obtained by pulling automorphisms of $C$ back to automorphisms of $C\times_S S'$. We call this the \emph{pullback of $C$ to $S'$.} Pulling back $V$-ferns preserves isomorphism classes, and we thus obtain a contravariant functor 
$$\Fern_V\colon \SchFq\to\mathop{\underline{\rm Set}}$$ sending a scheme $S$ to the set of isomorphism classes of $V$-ferns of $S$. This is a subfunctor of the functor $\CM_{\hat V}$ of stable $\hat V$-marked curves.
 \subsection{$V$-Ferns over $\Spec$ of a field and the associated flag}\label{GeoFib}
Let $(C,\lambda,\phi)$ be a $V$-fern over $\Spec k$, where $k$ is a field, and let $\SE:=E_1,\ldots, E_m$ be the chain from $0$ to $\infty$. The action of $V$ on $C$ obtained from $\phi|_V$ induces an action on the set of irreducible components of $C$. For each $1\leq i\leq m$ let $V_i:=\Stab_V(E_i)$. Define $V_0:=\{0\}\subset V$.
\begin{lemma}\label{subspace} For each $1\leq i\leq m$ the subgroup $V_i\subset V$ is an $\BF_q$-subspace.
\end{lemma}
\begin{proof} Fix $1\leq i\leq m$. We already know that $V_i$ is a subgroup of $V$ and hence closed under addition. It remains to show that $V_i$ is closed under scalar multiplication, i.e., that $\phi_{\xi v}(E_i)=E_i$ for all $v\in V_i$ and $\xi\in\BF_q$. Fix $v\in V_i$. If $\xi=0$, then $\phi_{\xi v}$ is the identity, and the assertion is trivial. 
Suppose $\xi\in\BF_q^\times$. Since $\phi|_{\BF_q^\times}$ fixes $\lambda_0$ and $\lambda_\infty$, it must stabilize each element of $\SE$. It follows that $\phi_{\xi}(E_i)=\phi_{\xi^{-1}}(E_i)=E_i$. Thus $\phi_{\xi v}(E_i)=\phi_{\xi}\circ\phi_{v}\circ\phi_{\xi^{-1}}(E_i)=E_i$ and hence $\xi v\in V_i$, as desired.
\end{proof}

\begin{lemma}\label{triv} For each $0\leq i\leq m$, the subspace $V_i$ acts trivially on $E_j$ for each $j>i$. If $i>0$, then $V_i/V_{i-1}$ acts faithfully on $E_i$.
\end{lemma}
\begin{proof}The statement is trivial for $i=0$. Fix $0< i<j\leq m$. Since $V_i$ fixes $\lambda_\infty$ and stabilizes $E_i$, it also stabilizes the unique chain from $E_i$ to $E_m$. It follows that $V_i$ sends $E_j$ to itself and fixes the distinguished special points $x_j$ and $y_j$. Choose an isomorphism $E_j\stackrel \sim\to\BP^1_k$ sending $x_j$ to $(0:1)$ and $y_j$ to $(1:0)$. The induced action of $V_i$ on $\BP^1_k$ fixes $0$ and $\infty$, and so each $v\in V_i$ acts via a matrix of the form $\binom{\zeta_v\,\,\, 0}{\,0\,\,\,\,\,\, 1}$ where $\zeta_v\in k^\times$. Let $p:=\ch(k)$. Since $p\cdot v=0$, we must have $(\zeta_v)^p=1$. It follows that $\zeta_v=1$ and $V_i$ acts trivially on $E_j$, as claimed.

For the second statement, we observe that the action of $V_i$ on $E_i$ factors through $V_i/V_{i-1}$ by the above argument. Any element $v\in V_i\setminus V_{i-1}$ sends the distinguished special point $x_i\in E_i$ to the intersection of $E_i$ and $\phi_v(E_{i-1})\neq E_{i-1}$ and hence does not fix $x_i$. The induced action of $V_i/V_{i-1}$ is thus faithful.
\end{proof}

\begin{proposition}\label{FlagFib}For each $0\leq i\leq m$, we have $V_i\subsetneq V_{i+1}$. Moreover, the set $\CF:=\{V_0,V_1,\ldots, V_m\}$ is a flag of $V$.
\end{proposition}

\begin{proof}The preceding lemma shows that $V_i\subset V_j$ whenever $i\leq j$. It only remains to show that $V_i$ is a proper subspace of $V_{i+1}$ for each $0\leq i<m$. For every $v\in V$, we have $\phi_v(\lambda_\infty)=\lambda_\infty$, which implies that $\phi_v(E_m)=E_m$ and hence $V_m=V$. Recall that we defined $V_0:=\{0\}$. If $m=1$, we are already done. Suppose $m>1$.
\vspace{2mm}

\noindent\emph{Case 1.} \underline{$i=0$}: We claim that there exists a $v\in V\setminus\{0\}$ such that $\lambda_v\in E_1$. To see this, we first recall that we defined a tail of $C$ to be an irreducible component corresponding to a leaf in the dual graph $\Gamma_C$. Any tail must contain at least two marked points by stability. Choose a tail $E\subset C$ and let $v\in V$ with $\lambda_v\in E$. Then $\phi_{-v}(E)=E_1$. Hence $E_1$ is itself a tail. There is thus a $w\in \hat V\setminus\{0\}$ with $\lambda_w\in E_1$. Since $m>1$, we must have $w\neq \infty$, and It follows that $w\in V_1$. Thus $V_0\subsetneq V_1$, as desired.
\vspace{2mm}

\noindent\emph{Case 2.} \underline{$i>0$}: We first observe that $E_{i+1}$ contains no non-$\infty$-marked points. Indeed, suppose there exists a $v\in V$ with $\lambda_v\in E_{i+1}$. Then $\{E_{i+1},\ldots, E_m\}$ is the unique chain from $\lambda_v$ to $\lambda_\infty$. But $\phi_v(\SE)$ is also a chain from $\lambda_v$ to $\lambda_\infty$. Since these chains have different lengths, this is a contradiction. 

The stability of $C$ thus implies that $E_{i+1}$ contains a singular point $x$ distinct from the distinguished special points $x_{i+1}$ and $y_{i+1}$. Let $C^\circ$ be the connected component of $\overline{C\setminus E_{i+1}}$ intersecting $E_{i+1}$ at $x$. Again by stability, there is a $v\in V$ with $\lambda_v\in C^\circ$. Since $\phi_v(\SE)$ is the unique chain from $\lambda_v$ to $\lambda_\infty$, and the latter contains $E_{i+1}$ but not $E_i$ it follows that $\phi_v(E_i)\neq E_i$ and that $\phi_v(E_{i+1})=E_{i+1}$. Thus $v\in V_{i+1}\setminus V_i$.
\end{proof}
Recall that a closed point on an irreducible component $E$ of $C$ is called \emph{special} if it is marked or singular.
\begin{corollary}\label{specialpoints}For each $1\leq i\leq m$, the special points on $E_i$  correspond to the elements of $(V_{i}/V_{i-1})\cup\{y_i\}$.
\end{corollary}
\begin{proof}We first show that $V_i$ acts transitively on the set of special points of $E_i$ not equal to $y_i$. Let $x\in E_i$ be such a point, and choose $v\in V$ corresponding to any marked point on the connected component of $\overline{C\setminus E_i}$ which intersects $E_i$ at $x$. The automorphism $\phi_v$ sends $\mathscr{E}$ to the unique chain from $\lambda_v$ to $\lambda_\infty$, and it follows that $\phi_v(x_i)=x$.

It remains to show that an element $v\in V_i$ fixes the distinguished special point $x_i$ if and only if $v\in V_{i-1}$. But any $v\in V_{i-1}$ fixes $x_i$ by Lemma \ref{triv}, and we showed in the proof of that lemma that any $v\in V_i\setminus V_{i-1}$ does not fix $x_i$.
\end{proof}
 \subsection{$V$-ferns for $\dim V=1$}
\begin{proposition}\label{dim1fern}If $\dim V=1$, then there is exactly one $V$-fern over $S$ up to unique isomorphism.
\end{proposition}
\begin{proof}
Let $(C,\lambda,\phi)$ be a $V$-fern over $S$. Let $s\in S$ and consider the fiber $C_s$. Since the elements of $\SE_s$ correspond to distinct non-zero subspaces of $V$ by Proposition \ref{FlagFib}, it follows that $|\SE_s|=1$; hence $C_s\cong\BP^1_{k(s)}$ and $C$ is a $\BP^1$-bundle over $S$. Since $C$ possesses three disjoint sections, it is in fact a trivial $\BP^1$-bundle over $S$. Let $(D,\mu,\psi)$ be another $V$-fern. Fix $v\in V\setminus\{0\}$. There is a unique isomorphism $\pi\colon C\stackrel\sim\to D$ with $\lambda_0\mapsto \mu_0$ and $\lambda_v\mapsto \mu_v$ and $\lambda_\infty\mapsto\mu_\infty$. Without loss of generality, we may assume $C$ and $D$ are both equal to $\BP^1_S$ with $0$- and $\infty$-sections given by $(0:1)$ and $(1:0)$ respectively, and that $\BF_q^\times$ acts by scalar multiplication. Then $\pi\colon \BP^1_S\stackrel\sim\to\BP^1_S$ fixes $0$ and $\infty$ and hence commutes with the action of $\BF_q^\times$. Since each $w\in V$ is equal to $\xi\cdot v$ for some $\xi\in \BF_q$, it follows that $\pi$ preserves all $\hat V$-marked sections; hence $\pi$ is an isomorphism of $V$-ferns.
\end{proof}

\subsection{$V$-ferns for $\dim V=2$}
Suppose $\dim V=2$, and let $(C,\lambda,\phi)$ be a $V$-fern over a field $k$. The associated flag $\CF$ is either trivial or of the form $\{0,V',V\}$ for some $V'\subset V$ of codimension 1. In the first case, we have $C\cong \BP^1$, and in the second $C$ is singular with $V'$-marked points lying on $E_{\{0\}}$ and the singular points of $E_{\{\infty\}}$ corresponding to the elements of $V/V'.$

\begin{figure}[H]
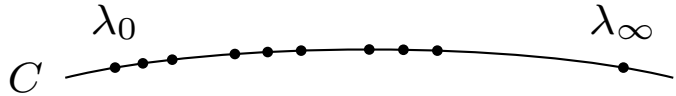

\centering
\caption{A smooth $V$-fern for $q=3$.}
\includestandalone[scale=2]{Figures/VfernDim2sm}
 
  \label{fig:dim2sm}
  \end{figure}
  
  \begin{figure}[H]
\centering
  \caption{A singular $V$-fern for $q=3$, where $v,w\in V$ are such that $v\not\equiv w\not\equiv 0\mod V'.$}
\begin{tikzpicture}
  \draw (0,0)  .. controls (1,.25) and (3,.25) .. node[at start,left] {\scriptsize$E_{\{\infty\}}$} (4,0) 
  node[circle,pos=0.9, fill=black,inner sep=0pt,minimum size=2pt, label=above:{\scriptsize$\lambda_\infty$}]{};

  \draw (.75,-0.25) .. controls (1,.25) and (1,0.75) .. (.75,1.75)
  node[at end,above] {\scriptsize $E_{\{0\}}$}
       	 node[circle,pos=0.9, fill=black,inner sep=0pt,minimum size=2pt, label=right:{\scriptsize$\lambda_0$}]{}
       	node[circle,pos=0.7, fill=black,inner sep=0pt,minimum size=2pt,]{}
       node[circle,pos=0.5, fill=black,inner sep=0pt,minimum size=2pt, ]{} ;
    \draw (1.75,-0.25) .. controls (2,.25) and (2,0.75) .. (1.75,1.75)
           		node[circle,pos=0.9, fill=black,inner sep=0pt,minimum size=2pt, label=right:{\scriptsize$\lambda_v$}]{}
       		node[circle,pos=0.7, fill=black,inner sep=0pt,minimum size=2pt,]{}
       		node[circle,pos=0.5, fill=black,inner sep=0pt,minimum size=2pt,]{};
     \draw (2.75,-0.25) .. controls (3,.25) and (3,0.75) .. (2.75,1.75)
                		node[circle,pos=0.9, fill=black,inner sep=0pt,minimum size=2pt, label=right:{\scriptsize$\lambda_w$}]{}
       		node[circle,pos=0.7, fill=black,inner sep=0pt,minimum size=2pt,]{}
       		node[circle,pos=0.5, fill=black,inner sep=0pt,minimum size=2pt,]{};

\end{tikzpicture}
  \label{fig:dim2sing}
  \end{figure}
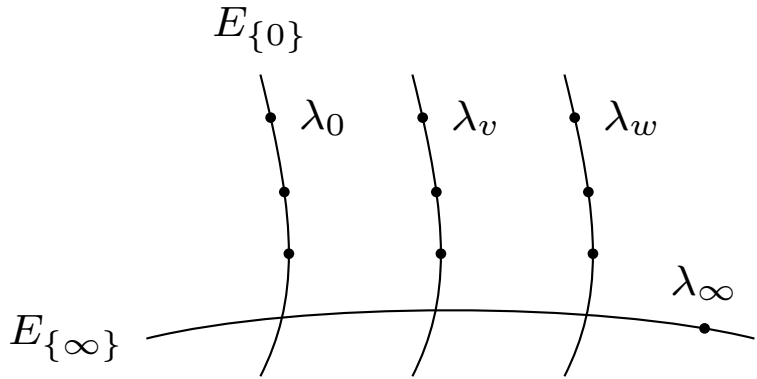
If $C$ is a $V$-fern over an arbitrary scheme $S$, we will see that the locus in $S$ where the fibers $C_s$ are smooth form an open subscheme of $S$.

\section{Constructions involving $V$-ferns}
\subsection{Contractions of $V$-ferns} Let $0\neq V'\subset V$ be an $\BF_q$-linear subspace, and let $G':=V'\rtimes\BF^\times_q$. Let $(C,\lambda,\phi)$ be a $V$-fern over a scheme $S$. Let $$\gamma\colon (C,\lambda)\to (C',\lambda')$$ be the contraction with respect to $\hat{V'}$, as defined in \ref{stablecontractions}.
\begin{lemma}\label{contrequivar}There exists a unique homomorphism
\begin{eqnarray*}\phi'\colon G'&\longto &\Aut_S(C'),
\\ (v',\xi)&\mapsto& \phi'_{v',\xi}
\end{eqnarray*}
such that for all $(v',\xi)\in G'$, the following diagram commutes:
$$\xymatrix{(C,\lambda)\ar[rr]_-\sim^-{\phi_{v',\xi}}\ar[d]_\gamma&& (C,\ar[d]_\gamma\lambda\cdot(v',\xi))\\
(C',\lambda')\ar[rr]_-\sim^-{\phi'_{v',\xi}} && (C',\lambda'\cdot (v',\xi)).}$$
\end{lemma}
\begin{proof} We first note that the $\hat V'$-marking on $C'$ is given by $\lambda'=c\circ\lambda|_{\hat V'}$. For each $(v',\xi)\in G'$ and $w\in\hat V$, we have
$$\bigl(\lambda'\cdot(v',\xi)\bigr)_{w}:=\lambda'_{\xi w+v'}=\gamma\circ{\lambda_{\xi w+v'}}=(\gamma\circ\phi_{v',\xi})\circ\lambda_{w}.$$
The morphism $\gamma\circ\phi_{v',\xi}\colon (C,\lambda)\to (C',\lambda'\cdot (v',\xi))$ is therefore also a stable contraction of $(C,\lambda)$ with respect to $\hat V'$. By the uniqueness of contractions, there is a unique $\phi'_{v',\xi}\in\Aut_S(C')$ inducing an isomorphism $(C',\lambda')\stackrel\sim\to(C',\lambda'\cdot(v',\xi))$. The resulting homomorphism $\phi'\colon G'\to\Aut(C'),\, (v',\xi)\mapsto \phi'_{v',\xi}$ satisfies the conditions in the lemma.
 \end{proof}

\begin{proposition}\label{FernContr}The tuple $(C,\lambda'|_{\hat V'},\phi')$ is a $V'$-fern.
\end{proposition}
\begin{proof}
It suffices to verify condition (2) of Definition \ref{FernDef}. Let $s\in S$ and let $\SE_s'\subset C'_s$ be the chain from $0$ to $\infty$. Let $E'\in\SE'_s$ and let $E$ be the unique irreducible component of $C_s$ mapping isomorphically onto $E'$ under $\gamma_s$. Lemma \ref{contrequivar} implies that $\gamma_s$ is $\BF_q^\times$-equivariant, and it follows that $E\in\SE_s$. There is thus an isomorphism $E\stackrel\sim\to\BP^1_{k(s)}$ satisfying condition (2) of Definition \ref{FernDef}. Pre-composing with the inverse of the isomorphism $E\stackrel\sim\to E'$ induced by $\gamma_s$ yields the desired isomorphism $E'\stackrel\sim\to \BP^1_{k(s)}$.
\end{proof}
To avoid cumbersome notation, we will often refer to $(C',\lambda',\phi')$ as a $V'$-fern in place of $(C',\lambda'|_{\hat V'},\phi')$. We usually view $(C',\lambda',\phi')$ as a $V'$-fern endowed with extra sections corresponding to $V\setminus V'$. 
\begin{definition}The $V'$-fern $(C',\lambda',\phi')$ along with the morphism $\gamma$ is called the \emph{contraction of $(C,\lambda,\phi)$} with respect to $\hat V'$. 
\end{definition}
\begin{remark}The contraction construction on stable $\hat V$-marked curves is functorial with respect to morphisms of stable $\hat V$-marked curves over a fixed base scheme as well as changing the base scheme. More specifically, for a morphism of stable $\hat V$-marked curves $(C,\lambda)\stackrel\sim\to(D,\mu)$ over a scheme $S$, there is a unique morphism of $\hat V$-marked curves $(C',\lambda')\stackrel\sim\to(D',\mu')$ over $S$ such that the following diagram commutes
$$\xymatrix{(C,\lambda)\ar[r]\ar[d]&(D,\mu)\ar[d]\\(C',\lambda')\ar[r]&(D',\mu').}$$ The commutativity with base change means the following: for a morphism $f\colon T\to S$, there is a unique morphism of $\hat V$-marked curves $\bigl((f^*C)',(f^*\lambda)'\bigr)\stackrel\sim\to\bigl(f^*(C'),f^*(\lambda')\bigr)$ over $T$ yielding a commutative diagram
$$\xymatrix{&(f^*C,f^*\lambda)\ar[dl]\ar[dr]&\\
\bigl((f^*C)',(f^*\lambda)'\bigr)\ar[rr]^\sim&&\bigl(f^*(C'),f^*(\lambda')\bigr).}$$
This implies that contraction corresponds to a natural transformation of moduli functors $\CM_{\hat V}\to\CM_{\hat V'}$. The contraction of $V$-ferns has the same functorial properties and thus yields a natural transformation
$\Fern_V\to\Fern_{V'}.$
\end{remark}

%
%
%

\subsection{Line bundles associated to $V$-ferns}\label{LineBundles}
Let $(C,\lambda, \phi)$ be a $V$-fern over $S$, and consider the contraction to the $\infty$-component $(C^\infty,\lambda^\infty)$ as in Definition \ref{i-comp}. The image of the $\infty$-section of $C^\infty$ is a relative effective Cartier divisor, whose complement we denote by $L$. The $\hat V$-marking on $C^\infty$ induces a $V$-marking on $L$. Slightly abusing notation, we denote the corresponding map by
$$\lambda: V\to L(S),\, v\mapsto \lambda_v.$$
 It follows from Corollary \ref{1-cor-uniquenessCi} that the left action of $G$ on $C$ induces a left action 
$$\phi^\infty\colon G\to \Aut_S(C^\infty).$$ 
This action fixes the infinity section, and restricts to and action on $L$. Again abusing notation, we denote the corresponding homomorphism by 
$$\phi\colon G\to\Aut_S(L).$$
\begin{proposition}\label{lbstr} There exists a natural structure of line bundle on $L$ such that
\begin{enumerate}[label=(\alph*)]
\item The zero section is $\lambda_0$.
\item For each $v\in V$ the morphism $\phi_v$ acts as translation by $\lambda_v$.
\item For each $\xi\in\BF_q^\times$, the automorphism $\phi_\xi$ acts as scalar multiplication by $\xi$.
\end{enumerate}
\end{proposition}
\begin{proof}
By construction, the curve $C^\infty$ is a $\BP^1$-bundle with disjoint sections $\lambda^\infty_0$ and $\lambda^\infty_\infty$. Suppose first that $S=\Spec R$ is affine and connected, that $C^\infty=\BP^1_R$ with $0$-section and $\infty$-section corresponding to $(0:1)$ and $(1:0)$ respectively. Then $L=\BA^1_R\subset C^\infty$ is the standard affine chart around $(0:1)$. We endow $L$ with the standard line bundle structure with the usual addition and scalar multiplication. Sections of $L$ correspond bijectively to elements of $R$, and we identify each section $\lambda_v$ with the associated element of $R$. By construction, we have $\lambda_0=0$ and (a) holds.

Since the $\infty$-section is fixed by $G$, the homomorphism $\phi^\infty\colon G\to \Aut_S(C^\infty)=\PGL_2(R)$ factors through $(\BG_m\ltimes \BG_a)(R)$, which we identify with the matrices of the form $\binom{a\ b}{ 0\ 1}$ with $a\in R^\times$ and $b\in R$. Since the action of $\BF_q^\times$ fixes the zero section, the restriction $\phi^\infty|_{\BF_q^\times}$ has image in $\BG_m(R)$. Moreover, each $\phi^\infty_\xi$ has order dividing $q-1$, so that $\phi^\infty_\xi\in \mu_{q-1}(R)$, the group of $(q-1)$-roots of unity in $R$. Our assumptions imply that $\mu_{q-1}(R)=\BF_q^\times$. Since $\phi|_{\BF_q^\times}$ is faithful, it must restrict to an automorphism of $\BF_q^\times$. We have the following cases.
\vspace{5mm}

\noindent\emph{Case 1. $\underline{(q\neq2)}$}: In this case, we have $[G,G]=V$. It follows that $\phi^\infty|_V\subset \Bigl[\binom{*\ *}{0\ 1},\binom{*\ *}{0\ 1}\Bigr]=\binom{1\ *}{0\ 1}$. Thus $V$ acts through $\BG_a(R)$. For every $v\in V$, we have $\phi^\infty_v\bigl((0:1)\bigr)=(\lambda_v:1)$, and it follows that $\phi^\infty_v=\binom{1\ \lambda_v}{0\ 1}$. Thus $\phi_v$ acts via translation by $\lambda_v$ on $L$ and (b) holds.

Let $s\in S$. Let $E:=E_{m_s}\subset C_s$ be the irreducible component containing the $\infty$-marked point, and let $x:=x_{m_s}$ and $y:=y_{m_s}$ be the distinguished special points on~$E$. The contraction morphism $C_s\to C^\infty_s$ restricts to a $G$-equivariant isomorphism $f\colon E\stackrel\sim\to C^\infty_s=\BP^1_k$ such that $f(x)=(0:1)$ and $f(y)=(1:0)$. It follows from the definition of $V$-ferns that $\BF_q^\times$ acts on $L_s$, and hence $L$, by scalar multiplication, so (c) holds.
\vspace{5mm}

\noindent\emph{Case 2. $\underline{(q=2)}$}: Pick any $v\in V$ such that $\lambda_v\neq0$. Shrinking $S$ if necessary, we may assume that $\lambda_v\in R^\times$. Write $\phi^\infty_v=\binom{a_v\ b_v}{0\ 1}$. Since $\phi^\infty_v\bigl((0:1)\bigr)=(\lambda_v:1)$, we deduce that $b_v=\lambda_v$. Moreover, the automorphism $\phi^\infty_v$ has order two, so
$$
\begin{pmatrix}a_v & \lambda_v\\0 &1\end{pmatrix}^2=\begin{pmatrix}a_v^2 & (a_v+1)\lambda_v\\0 &1\end{pmatrix}=\begin{pmatrix}1& 0\\0 &1\end{pmatrix}.
$$
 Since $\lambda_v\in R^\times$, it follows that $(a_v+1)=0$; hence $a_v=1$. For arbitrary $w\in V$, write $\phi^\infty_w=\binom{a_w\ b_w}{0\ 1}$. By the same reasoning as above, we have $b_w=\lambda_w$. Since $\phi^\infty_v$ and $\phi^\infty_w$ commute, we find that $\lambda_va_w+\lambda_w=\lambda_v+\lambda_w.$ Thus $a_w=1$ and (b) holds. Also, in this case (c) holds trivially. 
\bigskip

For general $S$, let $\pi^\infty\colon C^\infty\to S$ and $\pi\colon L\to S$ be the corresponding structure morphisms. We may choose an affine open covering $S=\cup_{\alpha}(U_\alpha=\Spec R_\alpha)$ with isomorphisms 
$$\psi^\infty_\alpha\colon (\pi^\infty)^{-1}(U_\alpha)\stackrel\sim\to \BP^1_{R_\alpha}$$
 sending the $0$- and $\infty$-sections to $(0:1)$ and $(1:0)$ respectively. These restrict to isomorphisms 
 $$\psi_\alpha\colon L_\alpha:=\pi^{-1}(U_\alpha)\stackrel\sim\to \BA^1_{R_\alpha}.$$ The above argument provides a line bundle structure on each $L_\alpha$ with the desired properties. For each $\alpha$ and $\beta$, let $U_{\alpha\beta}:=U_\alpha\cap U_\beta$. By construction, the automorphism $\psi_\alpha\circ\psi_\beta^{-1}$ of $\BA^1_{U_{\alpha\beta}}$ extends to an automorphism of $\BP^1_{U_{\alpha\beta}}$ and is therefore linear. The line bundle structures on the $L_\alpha$ thus glue to the desired line bundle structure on $L$.
\end{proof}
\begin{corollary}\label{fibnonzero}The map $\lambda\colon V\to L(S)$ is $\BF_q$-linear and fiberwise non-zero.
\end{corollary}
\begin{proof} Let $s\in S$, and choose $w\in V$ not contained in the second to last step of $\CF_s$. Then $\lambda_w(s)\neq \lambda_0(s)$ in $L_s$; hence $\lambda$ is fiberwise non-zero. Since $\phi_v\circ\phi_w=\phi_{v+w}$, it follows from Proposition \ref{lbstr}.2 that $\lambda_{v+w}=\lambda_{v}+\lambda_{w}$ for all $v,w\in V$. The action of $\BF_q^\times$ on $L(S)$ extends to $\BF_q$, and Proposition \ref{lbstr}.3 implies that $\lambda_{\xi v}=\xi\cdot \lambda_v$ for all $\xi\in \BF_q$ and $v\in V$. The $\BF_q$-linearity follows.
\end{proof}
\begin{proposition}[Compatibility with base change]\label{basechangelb}Let $f\colon T\to S$ be a morphism and let $M\to T$ be the line bundle associated to the $V$-fern $f^*C$ as in Proposition \ref{lbstr}. Then there is a unique isomorphism $M\stackrel\sim\to f^*L$ such that the following induced diagram commutes:
\begin{equation}\label{basechangelbdiag}
\begin{gathered}\xymatrix{V\ar[r]\ar[rd]& M(T)\ar[d]^\wr\\&f^*L(T).}\end{gathered}
\end{equation}
\end{proposition}
\begin{proof}By Proposition \ref{basechangeCi}, there is an isomorphism $(f^*C)^\infty\stackrel\sim\to f^*(C^\infty)$ compatible with the $\hat V$-markings. This induces an isomorphism of $T$-schemes $M\stackrel\sim\to f^*L$ such that (\ref{basechangelbdiag}) commutes. Since it is induced by an isomorphism of $\BP^1$-bundles and preserves the $0$-marked section, it is in fact an isomorphism of line bundles. The uniqueness follows from the fact that the $V$-marked sections generate $M$ and $f^*L$.
\end{proof}

\subsection{Grafting ferns}\label{grafting-section}
Let $0\neq V'\subsetneq V$ and denote the quotient space $V/V'$ by $\oV$. Let $(C',\lambda',\phi')$ and $(\oC,\overline{\lambda},\overline{\phi})$ be $V'$- and $\oV$-ferns respectively over a scheme $S$. We construct a $V$-fern $(C,\lambda,\phi)$ from $C'$ and $\oC$. On fibers $C$ will be given by attaching a copy of $C'$ to each $\oV$-marked point of $\oC$. We then use the $V'$- and $\oV$-fern structures on $C'$ and $\oC$ to define a $\hat V$-marking and $G$-action on $C$.

\subsubsection*{Clutching}
In \cite{Knud} a \emph{prestable curve} $\pi\colon X\to S$ is defined to be a flat and proper morphism such that the geometric fibres of $\pi$ are reduced curves with at most ordinary double points. Such curves are \textbf{not} assumed to be connected. Let $\pi\colon X\to S$ be a prestable curve with $\lambda_1,\lambda_2\colon S\to X$ non-crossing sections such that $\pi$ is smooth at the points $\lambda_i(s)$ for all $s\in S$. Knudsen proves the following, which is known as the \emph{clutching operation}:

\begin{proposition}[\cite{Knud},Theorem 3.4]\label{clutching} There is a commutative diagram $$
    \xymatrix{
              X\ar[r]^{p}\ar[d]^\pi & X'\ar[d]^{\pi'} \\ S\ar@{=}[r] &S}
$$
such that
\begin{enumerate}[label=(\alph*)]
\item $p\circ\lambda_1=p\circ\lambda_2$ and $p$ is universal with respect to this property, i.e., for all $S$-morphisms $q\colon X\to Y$ such that $q\circ\lambda_1=q\circ\lambda_2$, there exists a unique $S$-morphism $q'\colon X'\to Y$ such that $q=q'\circ p$.
\item $p$ is a finite morphism.
\item If $\os$ is a geometric point of $S$, the fibre $X'_\os$ is obtained from $X_\os$ by identifying $\lambda_1(\os)$ and $\lambda_2(\os)$ in such a way that the image point is an ordinary double point.
\item As a topological space $X'$ is the quotient of $X$ under the equivalence relation $\lambda_1(s)=\lambda_2(s)$ for all $s\in S$.
\item If $U'\subset X'$ is open and $U:=p^{-1}(U')$, then $$\Gamma(U',\CO_{X'})=\{f\in\Gamma(U,\CO_X)\mid \lambda_1^*(f)=\lambda_2^*(f)\}.$$
\item The morphism $X'\to S$ is flat; hence $X'$ is again a prestable curve by (c).
\end{enumerate}
\end{proposition}

\subsubsection*{Grafting}
Let $U\subset V$ be a subspace such that $V=U\oplus V'$.  We consider the prestable curve
$$\tilde{C}:= \oC\sqcup\Bigl(\bigsqcup_{u\in U} C'\Bigr).$$
For each $u\in U$, denote corresponding copy of $C'$ in $\tilde C$ by $(C')_u$, with infinity section denoted by $\infty_{u}$. We define a $(u+V')$-marking on $(C')_u$ by sending each $u+v'$ to the $v'$-marked section of $(C')_u$. Defining the $\infty$-section of $\tilde{C}$ to be the $\infty$-section of $\oC\subset \tilde C$, we obtain a $\hat V$-marking $$\tilde\lambda\colon \hat V\to \tilde C(S).$$ Now define a left $G$-action on $\tilde C$ as follows: Let $(v,\xi)\in G$ and write $v$ uniquely as $v=u+v'$ with $u\in U$ and $v'\in V'$. Then $(v,\xi)$ acts on $\tilde{C}$ via
\begin{enumerate} 
\item $\overline{\phi}_{\ov,\xi}$ on $\oC\subset \tilde{C}$, where $\ov$ is the image of $v$ in $\oV.$
\item mapping $(C')_r$ identically to $(C')_{\xi r+u}$ for each $r\in U$ and then acting via $\phi'_{v',\xi}.$
\end{enumerate}
This defines a left $G$-action on $\tilde C$ that is compatible with the $\hat V$-marking. Denote the corresponding homomorphism by 
$$\tilde \phi\colon G\to \Aut _S(\tilde C).$$

By iterating the clutching operation, for each $u\in U$ we identify $\infty_{u}$ and $\lambda_{\overline{u}}$, where $\overline{u}$ is the image of $u$ in $\oV$, and obtain a connected prestable curve $C$ with a morphism $p\colon \tilde C\to C$. Moreover, for each $(v,\xi)=(u+v',\xi)\in G$ and $r\in U$, we have 
$$(p\circ\tilde\phi_{v,\xi})\circ\tilde\lambda_{r}=p\circ\tilde\lambda_{\xi r+u}=p\circ\infty_{\xi r +u}=(p\circ\tilde\phi_{v,\xi})\circ \infty_r.$$
 By the universality of $p$ (Proposition \ref{clutching}.1), it follows that there exists a unique $\phi_{v,\xi}\in\Aut_S(C)$ making the following diagram commute: $$
    \xymatrix{
              \tilde C\ar[r]^{\tilde\phi_{v,\xi}}\ar[d]^p & \tilde C\ar[d]^p \\ C \ar[r]_{\phi_{v,\xi}} &C.}
$$ We thus obtain a faithful left $G$-action
$$\phi\colon G\to \Aut _S(C).$$
 In addition, the $\hat V$-marking $\tilde\lambda$ induces a $\hat V$-marking 
 $$\lambda\colon\hat V\to C(S).$$ 

\begin{proposition}\label{grafting} The triple $(C,\lambda,\phi)$ is a $V$-fern.
\end{proposition}
\begin{proof}
It is clear that $C$ is a stable $\hat V$-marked curve. It remains to verify condition (2) in Definiton \ref{FernDef}. Let $s\in S$, and let $\SE_s$ be the chain from $0$ to $\infty$ in $C_s$. Let $E\in\SE_s$. Since the $G$-action on $C$ is defined so that $p\colon \tilde C\to C$ is $G$-equivariant, it follows that the unique irreducible component $\tilde E\subset \tilde C_s$ mapping isomorphically onto $E$  is contained in either $\SE'_s$ or $\overline{\SE}_s$, the chains from $0$ to $\infty$ in $(C'_s)_0$ and $\oC$ repectively. There thus exists an isomorphism $\tilde E\stackrel\sim\to \BP^1_{k(s)}$ satisfying condition (2) in Definiton \ref{FernDef}. The composite 
$$\xymatrix{E\ar[r]_-\sim^{p_s^{-1}}&\tilde E\ar[r]_-\sim&\BP^1_{k(s)}}$$ provides the desired isomorphism for $E$.
\end{proof}

\begin{definition} We refer to the construction of $(C,\lambda,\phi)$ as \emph{grafting} and call $C$ the \emph{graft} of $C'$ and $\oC$.
\end{definition}

Note that the grafting construction depends on the choice of complementary subspace $U\subset V$ with $V=U\oplus V'.$ However, the isomorphism class of the resulting $V$-fern is independent of this choice. Indeed, let $T\subset V$ be another such subspace. With similar notation as above let $(C^U,\lambda^U,\phi^U)$ and $(C^T,\lambda^T,\phi^T)$ be the corresponding grafts of $C'$ and $\oC$. Then we have:

\begin{proposition} The $V$-ferns $C^U$ and $C^T$ are isomorphic.
\end{proposition} 
\begin{proof} For each $u\in U$ there exists a unique $t\in T$ such that $u\equiv t$ mod $V'$. Write $u=t+v'$ with $v'\in V'$. Consider the isomorphism $$f_u\colon (C')_u\subset \tilde C^U\stackrel\sim\to (C')_t\subset \tilde C^T$$ induced by the action of $v'$ on $C'$. For each $v\in V$, there exist unique $v_u',v_t'\in V'$ such that $v=u+v_u'=t+v'_t$. Since $u=t+v'$ we have $t+(v'+v'_u)=t+v'_t$ and hence $v'_t=v'+v'_u$. It follows that $f_u\circ\tilde\lambda^U_v=\tilde\lambda^T_v$ for all $v\in V$. We may thus combine the $f_u$ for varying $u$ with the identity on $\oC\subset \tilde C^U\stackrel\sim\to \oC\subset\tilde C^{T}$ to obtain an isomorphism $\tilde f\colon \tilde C^U\to \tilde C^T$ preserving $\hat V$-marked sections. Using the universality from Proposition \ref{clutching}.1 of the morphisms $p_U\colon \tilde C^U\to C^U$ and $p_T\colon \tilde C^T\to C^T$, this induces the desired isomorphism $f\colon C^U\stackrel\sim\to C^T$ of stable $\hat V$-marked curves.
\end{proof}
\begin{remark}The fact that the clutching operation is functorial (\cite[\S10.8]{Arb}) implies that grafting yields a natural transformation $\Fern_{V'}\times\Fern_{\oV}\to \Fern_V$.
\end{remark}
We can generalize the construction in the following way: Let $\CF=\{V_0,\ldots, V_m\}$ be a flag and, for each $0<i\leq m$, let $C_i$ be a $(V_i/V_{i-1})$-fern. By iteratively grafting, one obtains a $V$-fern $C$ from the $C_i$.

\section{$\CF$-ferns and flag coverings}
Let $(C,\lambda,\phi)$ be a $V$-fern over $S$. For each $s\in S$, we obtain a flag $\CF_s$ from the fiber $C_s$ using Proposition \ref{FlagFib}. 
\begin{definition}
We call $\CF_s$ the \emph{flag associated to $s$} (or to $C_s$, in case of ambiguity).
\end{definition}
\subsection{Contraction and the flags associated to fibers}
For any flag $\CF=\{V_0,\ldots,V_m\}$ of $V$, we obtain a flag of $V'$ given by $$\CF\cap V':=\{V_0\cap V',\ldots,V_m\cap V'\}.$$ 

\begin{proposition}\label{flgctr} Let $(C',\lambda',\phi')$ be the $V'$-fern obtained from $C$ by contraction. Let $s\in S$, and let $\CF:=\CF_s=\{V_0,\ldots, V_m=V\}$ and $\CF':=\CF'_s=\{V'_0,\ldots,V'_\ell=V'\}$ be the flags associated to $C_s$ and $C'_s$ respectively. Then $\CF'=\CF\cap V'$.
\end{proposition}
\begin{proof} We may assume that $S=\Spec(k)$, where $k$ is a field. Let $\SE:=\{E_1,\ldots E_m\}$ and $\SE':=\{E'_1,\ldots E'_\ell\} $ be the chains from $0$ to $\infty$ in $C$ and $C'$ respectively. Fix $1\leq j\leq \ell$. Let $\gamma\colon C\to C'$ denote the contraction map, and let $E$ be the unique irreducible component of $C$ such that $\gamma(E)=E_j'$. Since $\gamma$ preserves the $0$- and $\infty$-marked sections, we must have $\gamma(\SE)=\SE'$. Therefore $E=E_i$ for some $i$. Let $v\in V_i\cap V'.$ We have \begin{equation}\label{veq}E'_j=\gamma(E_i)=\gamma(\phi_v E_i)=\phi'_v \gamma(E_i)=\phi'_v E'_j,\end{equation} where we use the $V'$-equivariance of $\gamma$ given by Lemma \ref{contrequivar}. Thus $v$ stabilizes $E'_j$, and hence $v\in V'_j$. We conclude that $V_i\cap V'\subset V_j'$. Conversely, suppose $v\in V_j'$. Again using the $V'$-equivariance, we obtain $E_j'=\phi'_v \gamma(E_i)=\gamma(\phi_v E_i)$. Since $E_i$ is the unique irreducible component mapping onto $E_j'$, we have $\phi_v E_i= E_i$ and hence $v\in V_i\cap V'.$ It follows that $V'_j=V_i\cap V'$; hence $\CF'\subset \CF\cap V'$.
\vspace{2mm}

For the reverse inclusion, fix $1\leq i\leq m$. Either $V'\cap V_i=\{0\}\in \CF'$, or there exists a maximal $1\leq k\leq i$ such that $V'\cap (V_k\setminus V_{k-1})\neq \emptyset$. In the latter case, one has $V_i\cap V'=V_k\cap V'$. Consider the subtree $T_k$ of $C$ consisting of the connected component of $C\setminus (E_{k-1}\cup E_{k+1})$ containing $E_k$, where $E_0$ and $E_{m +1}$ are understood to be empty. The non-$\infty$-marked points on $T_k$ are precisely the $(V_k\setminus V_{k-1})$-marked points. Since $V'\cap (V_k\setminus V_{k-1})\neq \emptyset$, the contraction morphism $\gamma$ does not contract $E_k$ to a point. Hence $\gamma(E_k)=E'_j$ for some $j$. As before, we deduce that $V'_j=V_k\cap V'=V_i\cap V'.$ Thus $\CF\cap V'\subset \CF'$. This completes the proof.
\end{proof}

\subsection{Flag covering}
For an arbitrary flag $\CF$ of $V$, we define $$S_\CF:=\{s\in S\mid \CF_s\subset \CF\}.$$

\begin{lemma}\label{Sv} Let $v\in V$ and let $S_v\subset S$ be the locus of points $s\in S$ such that $v$ is not contained in the second to last step of $\CF_s$. Then $S_v$ is an open subscheme of $S$.
\end{lemma}
\begin{proof} Consider the contraction $C^v$ of $C$ with respect to $\{0,v,\infty\}$ and denote the $\hat V$-marked sections of $C^v$ by $(\lambda^v_w)_{w\in \hat V}$. We claim that 
$$S_v=\{s\in S\mid \lambda^v_w(s)\neq \lambda^v_\infty(s)\mbox{ for all $w\in V$}\}.$$ 
Indeed, let $s\in S$ and write $\CF_s=\{V_0,\ldots, V_m\}$ with corresponding chain $E_1,\ldots, E_m$ from $0$ to $\infty$. Let $1\leq i\leq m$ be minimal such that $v\in V_i$. Then $C_s\to C^v_s$ contracts all irreducible components of $C_s$ except for $E_{i}$. Thus $v\not\in V_{m-1}$ if and only if the contraction morphism $C_s\to C^v_s$ contracts all irreducible components of $C_s$ except $E_m$, which is equivalent to $\lambda^v_w(s)\neq \lambda^v_\infty(s)\mbox{ for all $w\in V$}$, and the claim follows. Since the locus where two sections are distinct is open and $V$ is finite, the desired result follows.
\end{proof}

\begin{proposition}\label{flagcov} For each flag $\CF$ of $V$, the set $S_\CF\subset S$ is an open subscheme, and $S=\bigcup_{\CF} S_\CF.$
\end{proposition}
\begin{proof}We proceed by induction on the length of $\CF$. Suppose first that $\CF=\{0,V\}$ is trivial. For each $v\in V$, let $S_v$ be as in Lemma \ref{Sv}. Now $S_\CF=\bigcap_{v\in V\setminus\{0\}} S_v$ is open since $V$ is finite.

Now suppose $\CF=\{V_0,V_1,\ldots, V_m\}$ with $m> 1$. Let $V':=V_{m-1}$ and denote by $C'$ the $V'$-fern obtained from $C$ by contraction. Consider the flag of $V'$ given by $\CF':=\CF\cap V'$. Then $S_{\CF'}\subset S$ is open by the induction hypothesis. By Proposition \ref{flgctr}, the open subscheme $S_{\CF'}$ is precisely the locus of points $s\in S$ such that $\CF_s\cap V'\subset \CF'$. If $\CF_s\subset \CF$, then $\CF_s\cap V'\subset\CF'$ and therefore $S_\CF\subset S_{\CF'}$.  The condition that $\CF_s\cap V'\subset\CF'$ is equivalent to $\CF_s$ is being a subflag of a flag of the form $$0=V_0\subsetneq V_1\subsetneq\cdots\subsetneq V_{m-1}=V'\subsetneq W_1\subsetneq\cdots\subsetneq W_k=V.$$
A point $s\in S_{\CF'}$ is thus in $S_\CF$ if and only if there does not exist a $W\in \CF_s$ such that $V'\subsetneq W\subsetneq V$. This is equivalent to every $v\in V\setminus V'$ not being contained in the second to last step of $\CF_s$. We thus have
$$S_\CF=\Big(\bigcap_{v\in V\setminus V'}S_v\Big)\cap S_{\CF'}.$$ Again applying Lemma \ref{Sv}, it follows that $S_\CF$ is open. The fact that the $S_\CF$ cover $S$ is clear. 
\end{proof}

\begin{definition} We call a $V$-fern $C\to S$ an \emph{$\CF$-fern} if $S=S_\CF$.
\end{definition}
Proposition \ref{flagcov} implies that the functor sending a scheme $S$ to the set of isomorphism classes of $\CF$-ferns over $S$ is an open subfunctor of $\Fern_V$. We denote this functor by $\Fern_\CF$.


%


\section{The universal $V$-fern}
In this section we \textbf{construct} a $V$-fern $(\CC_V,\lambda_V,\phi_V)$ over $B_V$. 
\subsection{The scheme $\CC_V$}
Let $\Sigma:=V\times(V\setminus \{0\}).$ We define
\begin{eqnarray*}
P^\Sigma:=\prod_{(v,w)\in\Sigma}\BP^1.
\end{eqnarray*}
Denote the projective coordinates on the copy of $\BP^1$ in the $(v,w)$-component of $P^\Sigma$ by $(X_{vw}:Y_{vw})$. Let $S$ be scheme. Recall that an element of $\OV(S)$ corresponds to a fiberwise injective linear map 
$$\lambda\colon V\to \CL(S),\, v\mapsto \lambda_v.$$ 
Moreover, an element of $\BA^1_{\BF_q}(S)$ corresponds to the choice of some $t\in\Gamma(S,\CO_S)$. Fix $v_0\in V\setminus\{0\}$ and define the following morphism:
\begin{eqnarray}\nonumber\alpha_V\colon\OV\times\BA^1 &\longto& \OV\times P^\Sigma
\\\nonumber(\lambda,t)&\mapsto& \bigl(\lambda,(\lambda_{v_0}t-\lambda_{v}:\lambda_{w})_{v,w}\bigr).
\end{eqnarray}
To see that $\alpha_V$ is well-defined, consider a pair $(\CM,\mu)$ in the same isomorphism class as $(\CL,\lambda)$, i.e., corresponding to the same element of $\OV(S)$. This means that there is an isomorphism $\theta\colon\CL\stackrel\sim\to \CM$ such that the following diagram commutes:
$$\xymatrix{V\ar[r]^{\lambda}\ar[rd]_\mu&\CL(S)\ar[d]^\theta_\wr\\&\CM(S).}$$
For $v,w\in\Sigma$ and $t\in \CO_S(S)$, we have
\begin{eqnarray*}
\theta(\lambda_{v_0}t-\lambda_v)&=&\mu_{v_0}t-\mu_v;\\
\theta(\lambda_w)&=&\mu_w.
\end{eqnarray*}
Thus the tuples $(\CL,\lambda_{v_0}t-\lambda_v,\lambda_w)$ and $(\CM,\mu_{v_0}t-\mu_v,\mu_w)$ determine the same morphism $S\to\BP^1$, and $\alpha_V$ is well-defined.

Consider the composite 
\begin{equation}\label{theMap}
\xymatrix@C-=0.8cm{f_V\colon\OV\times\BA^1\ar[r]^-{\alpha_V}&\OV\times P^\Sigma\ar@{^{(}->}[r]&B_V\times P^\Sigma.}\end{equation}
The projective scheme $\CC_V$ over $B_V$  is defined to be the scheme-theoretic closure of $\Image f_V$.
 \subsection{The scheme $\CC_\OF$} 
Fix a complete flag $\overline\CF:=\{V_0,\ldots, V_n\}$ of $V$ and an associated flag basis $\CB=\{b_1,\ldots,b_n\}$. To simplify notation, let $\UT:=(T_1,\ldots, T_{n-1})$. By Proposition \ref{Coords}, the choice of $\CB$ yields an open embedding $U_\OF\into\BA^{n-1}_{\BF_q}=\Spec \BF_q[\UT]$. We identify $U_\OF$ with its image in $\BA^{n-1}_{\BF_q}$. In this subsection, we construct a $V$-fern $(\CC_\OF,\lambda_\OF,\phi_\OF)$ over $U_\OF$. We will show (Propositions \ref{localCV} and \ref{VfernStr}) that $\CC_\OF$ is isomorphic to $\CC_V\cap(U_\OF\times P^\Sigma)$ and that the $V$-fern structures for varying $\OF$ combine to yield the desired structure of $V$-fern on $\CC_V$.
\vspace{2mm}

For each $1\leq k\leq n$ and each $v\in V_k$, write $v=\sum_{i=1}^k c_i b_i$ with $c_i\in\BF_q$, and define 
$$Q^k_v(\UT):=\sum_{i=1}^k c_i\Bigl(\prod_{j=i}^{k-1}T_j\Bigr)\in\BF_q[\UT].$$
\begin{remark}The $Q^k_\bullet$ are $\BF_q$-linear in $V_k$, i.e., for all $v,w\in V_k$ and $\xi\in\BF_q,$ we have $Q^k_{\xi v+w}=\xi Q^k_v+Q^k_w.$ 
Moreover, for $v\in V_k$ and $k\leq m\leq n$, we have 
\begin{equation}\label{lowerdim}Q^m_v(\UT)=Q^k_v(\UT)\cdot\prod_{j=k}^{m-1}T_j.
\end{equation}
\end{remark}
We define $\CC_\OF\subset U_\OF\times P^\Sigma$ to be the closed subscheme determined by the following equations:
\begin{equation}\label{defeqs}Q^{k}_w(\UT)X_{vw}Y_{v'w'}+Q^{k}_{v-v'}(\UT)Y_{vw}Y_{v'w'}=Q^k_{w'}(\UT)X_{v'w'}Y_{vw},\end{equation} 
where $1\leq k\leq n$ and $w,w'\in V_k\setminus\{0\}$ and $(v'-v)\in V_k.$ 
\vspace{2mm}

Let $\pi_\OF\colon \CC_\OF\to U_\OF$ denote the natural projection. For each subflag $\CF\subset\OF$, the scheme $U_\CF$ is an open subscheme of $U_\OF$, and hence inherits the affine coordinates $T$ from the latter. Let $\CC_\CF:=\pi_{\OF}^{-1}(U_\CF)$. It is precisely the closed subscheme of $U_\CF\times P^\Sigma$ defined by (\ref{defeqs}). We denote the restriction of $\pi_{\OF}$ to $\CC_\CF$ by $\pi_\CF$.

\begin{lemma}\label{cont}The image 
 of $f_V$ is contained in $\CC_\CF$ for each $\CF\subset\OF$.
\end{lemma}
\begin{proof}
Let $S$ be a scheme. The morphism $f_V$ maps a point $(\lambda,t)\in(\OV\times\BA^1)(S)$ to the point $$\bigl(\lambda,(\lambda_{v_0}t-\lambda_v:\lambda_w)_{v,w}\bigr)\in(\OV\times  P^\Sigma)(S).$$ Recall that $\lambda$ is an $\BF_q$-linear map from $V$ to the global sections of an invertible sheaf $\CL$ over $S$. Define $x_{vw}:=\lambda_{v_0}t-\lambda_v$ and $y_{vw}:=\lambda_w$, so that the morphism from $S$ to the $(v,w)$-component of $ P^\Sigma$ is determined by the tuple $(x_{vw}:y_{vw})$. We compute
\begin{eqnarray*}
\lambda_wx_{vw}y_{v'w'}+\lambda_{v-v'}y_{vw}y_{v'w'}&=\\
\lambda_w(\lambda_{v_0}t-\lambda_v)(\lambda_{w'})+\lambda_{v-v'}\lambda_{w}\lambda_{w'}&=\\
\lambda_w\lambda_{w'}(\lambda_{v_0}t-\lambda_v+\lambda_{v-v'})&=\\
\lambda_w\lambda_{w'}(\lambda_{v_0}t-\lambda_{v'})&=
\\\lambda_{w'}x_{v'w'}y_{vw}.&
\end{eqnarray*}
Thus for all pairs $(v,w),(v',w')\in \Sigma$, we have
\begin{equation}\label{eqsIm}\lambda_wx_{vw}y_{v'w'}+\lambda_{v-v'}y_{vw}y_{v'w'}=\lambda_{w'}x_{v'w'}y_{vw}
\end{equation}
in $\Gamma(S,\CL^{\otimes 3})$. Our choice of flag basis $\CB=\{b_1,\ldots,b_n\}$ induces projective coordinates $\UX=(X_1:\ldots:X_n)$ on $P_V$, hence an open embedding $j\colon\OV\into\BP^{n-1}_{\BF_q}.$
For each $v=\sum_{i=1}^n c_ib_i$ in $V$, define 
$$P_v(X):=\sum_{i=1}^n c_iX_i\in\BF_{q}[\UX].$$ 
Then $P_v(\UX)\in \Gamma(\BP_{\BF_q}^{n-1},\CO(1))$ and $j^*P_v(\UX)=v\in\Gamma(\OV,\CO(1))$. It follows that under the composite $$\xymatrix{S\ar[r]^\lambda&\OV\ar@{^{(}->}[r]^j&\BP^{n-1}_{\BF_q},}$$ the section $P_v(\UX)$ pulls back to $\lambda_v$ for all $v\in V$.
By equation (\ref{eqsIm}), we deduce that $f_V$ factors through the closed subscheme defined by
\begin{equation}\label{eqsProj}P_w(\UX)X_{vw}Y_{v'w'}+P_{v-v'}(\UX)Y_{vw}Y_{v'w'}=
P_{w'}(\UX)X_{v'w'}Y_{vw},
\end{equation} for all $(v,w),(v',w')\in \Sigma$. Using Lemma \ref{coordchange}, we observe that $j^*(P_v(\UX)/X_n)=Q^n_v(\UT)$ over $\OV$. Dividing (\ref{eqsProj}) by $X_n$ therefore yields
\begin{equation}\label{eqsImCoord}Q^{n}_w(\UT)X_{vw}Y_{v'w'}+Q^{n}_{v-v'}(\UT)Y_{vw}Y_{v'w'}=Q^n_{w'}(\UT)X_{v'w'}Y_{vw}
\end{equation}
This is almost (\ref{defeqs}), but differs in the superscripts of the $Q^n$. Fix $1\leq k\leq n$. Since $j(\OV)\subset \cap_iD^+(X_i)\subset \BP_{\BF_q}^{n-1}$, the product $\prod_{i=k}^{n-1}T_i=j^*\frac{X_k}{X_n}$ is invertible on $\OV$. If $w,w'$ and $(v'-v)$ are in $V_k$, we can thus divide both sides of (\ref{eqsImCoord}) by $\prod_{i=k}^{n-1}T_i$. Using equation (\ref{lowerdim}), this yields (\ref{defeqs}). The image of $f_V$ is thus contained in $\CC_\CF$, as desired.
\end{proof}
\subsection{$G$-action on $\CC_\OF$}\label{SubSectAction}
Recall that $\OV\cong\Spec RS_{V,0}$, where $RS_V$ is the localization of $\Sym(V)$ by all $v\in V\setminus \{0\}$. Let $A:=RS_{V,0}$. Our definition of 
$f_V\colon \OV\times \BA^1\to B_V\times P^\Sigma$
depended on a choice of $v_0\in V\setminus \{0\}$. Consider the faithful right $G$-action on the polynomial ring $A[X]$ where an element $(u,\xi)\in G$ acts via the $A$-algebra homomorphism
$$A[X]\to A[X],\, X\mapsto \xi X+\frac{u}{v_0}.$$
This induces a left $G$-action on $\OV\times \BA^1$ over $\OV$ given by $$(\lambda,t)\mapsto \Bigl(\lambda,\xi t+\frac{\lambda_u}{\lambda_{v_0}}\Bigr).$$
On the other hand, we endow $B_V\times P^\Sigma$ with the left $G$-action over $B_V$ where $(u,\xi)\in G$ acts via the automorphism
\begin{eqnarray*}
B_V\times P^\Sigma&\stackrel\sim\longto& B_V\times P^\Sigma\\
\nonumber(x_{vw}:y_{vw})_{v,w}&\mapsto & (x_{\tilde{v}\tilde{w}}:y_{\tilde{v}\tilde{w}})_{v,w},
\end{eqnarray*}
where
\begin{eqnarray}\label{transl}\tilde v&:=&\xi^{-1}(v-u);
\\\nonumber\tilde w&:=&\xi^{-1}w.
\end{eqnarray}
\begin{lemma}\label{AffAct}The morphism $f_V$ is $G$-equivariant.
\end{lemma}
\begin{proof}Recall that $f_V$ was defined to be the composite of the natural inclusion $\OV\times P^\Sigma\into B_V\times P^\Sigma$ with the morphism
\begin{eqnarray}\nonumber\alpha_V\colon\OV\times\BA^1 &\longto& \OV\times P^\Sigma
\\\nonumber(\lambda,t)&\mapsto& \Bigl(\lambda,\bigr(\lambda(v_0)t-\lambda(v):\lambda(w)\bigr)_{v,w}\Bigr).
\end{eqnarray} Under $\alpha_V$, we have
$$
\Bigl(\lambda,\xi t+\frac{\lambda_u}{\lambda_{v_0}}\Bigr)\mapsto \Bigl(\lambda,\Bigl(\lambda_{v_0}\Bigl(\xi t+\frac{\lambda_u}{\lambda_{v_0}}\Bigr)-\lambda_v:\lambda_w\Bigr)_{v,w}\Bigr),
$$
The lemma then follows from the following computation:
\begin{eqnarray*}
\Bigl(\lambda_{v_0}\Bigl(\xi t+\frac{\lambda_u}{\lambda_{v_0}}\Bigr)-\lambda_v:\lambda_w\Bigr)&=&\bigr(\lambda_{v_0}t+\xi^{-1}(\lambda_u-\lambda_v):\xi^{-1}\lambda_w\bigr)\\
&=&(\lambda_{v_0}t-\lambda_{\xi^{-1}(v-u)}:\lambda_{\xi^{-1}w})
\\&=&(\lambda_{v_0}t-\lambda_{\tilde v}:\lambda_{\tilde w}).
\end{eqnarray*}
\end{proof}
\begin{proposition}\label{Action} The left $G$-action on $B_V\times P^\Sigma$ restricts to a left $G$-action on $\CC_\OF$ over $U_\OF$. Moreover, the morphism $\OV\times\BA^1\to \CC_\OF$ induced by $f_V$ is $G$-equivariant.
\end{proposition} 
\begin{proof}
The action of an element $(u,\xi)\in G$ sends $\CC_\OF$ to the closed subscheme of $U_\OF\times P^\Sigma$ defined by the equations 
\begin{equation}\label{1}Q^{k}_w(\UT)X_{\tilde v\tilde w}Y_{\tilde v'\tilde w'}+Q^{k}_{v-v'}(\UT)Y_{\tilde v\tilde w}Y_{\tilde v'\tilde w'}=Q^k_{w'}(\UT)X_{\tilde v'\tilde w'}Y_{\tilde v\tilde w},
\end{equation}
for $1\leq k\leq n$ and $w, w'\in V_k\setminus\{0\}$ and $v,v'\in V$ such that $v'-v\in V_k$, and with $\tilde v$, $\tilde v'$, $\tilde w$ and $\tilde w'$ as in (\ref{transl}). We denote this scheme by $(u,\xi)\cdot C_\OF$. To show that the action restricts to $C_\OF$, we must show that $C_\OF=(u,\xi)\cdot C_\OF$, which is equivalent to showing that the equations (\ref{1}) hold on $C_\OF$.

By (\ref{defeqs}), the following equations hold on $C_\OF$ for varying $v, w, v'$ and $w':$ 
\begin{equation}\label{2}Q^{k}_{\tilde w}(\UT)X_{\tilde v\tilde w}Y_{\tilde v'\tilde w'}+Q^{k}_{\tilde v-\tilde v'}(\UT)Y_{\tilde v\tilde w}Y_{\tilde v'\tilde w'}=Q^k_{\tilde w'}(\UT)X_{\tilde v'\tilde w'}Y_{\tilde v\tilde w}.
\end{equation}
Since $$\tilde v-\tilde v'=\xi^{-1}(v-u)-\xi^{-1}(v'-u)=\xi^{-1}(v-v'),$$ and $Q_{\xi^{-1} x}^k=\xi^{-1} Q_x^k$ for all $x\in V_k$, we obtain (\ref{1}) from (\ref{2}) by multiplying by $\xi^{-1}$. This proves the first part of the proposition. The fact that the morphism $\OV\times\BA^1\into \CC_\OF$ is $G$-equivariant follows directly from Lemma \ref{AffAct}.
\end{proof}
We denote the left $G$-action in Proposition \ref{Action} by
$$\phi_\OF\colon G\to\Aut_{U_\OF}(\CC_\OF).$$

\subsection{$\hat V$-Marked Sections of $\CC_\OF\to U_\OF$}
Sections of the projection $\OV\times \BA^1\to\OV$ correspond bijectively to elements of $A:=RS_{V,0}$. The inclusion
$$V\into R,\,\,v\mapsto \frac{v}{v_0},$$ where $v_0\in V\setminus\{0\}$ is the vector we fixed in defining $f_V$, thus induces a $V$-marking
$$\lambda\colon V\to (\OV\times\BA^1)(\OV),\,\,v\mapsto \lambda_v.$$
Furthermore, let
$$\phi\colon G\to \Aut_\OV(\OV\times\BA^1)$$ be the left $G$-action on $\OV\times\BA^1$ as described in Subsection \ref{SubSectAction}.
\begin{lemma}\label{AffSects}The $V$-marking on $\OV\times\BA^1$ is compatible with the action of $G$, i.e., the equality $\phi_{w,\xi}\circ\lambda_v=\lambda_{\xi v+w}$ holds for all $(w,\xi)\in G$. \end{lemma}
\begin{proof}
Let $(w,\xi)\in G$. For each $v\in V$, the section $\lambda_v$ corresponds to the homomorphism $\lambda_v^\sharp\colon A[X]\to A,\,\, X\mapsto \frac{v}{v_0}$. We denote by $\phi^\sharp$ the left $G$-action on $A[X]$ corresponding to $\phi$. Then $$\lambda_v^\sharp\circ\phi_{w,\xi}^\sharp(X)=\xi \Bigl(\frac{v}{v_0}\Bigr)+\frac{w}{v_0}=\frac{\xi v+w}{v_0}=\lambda_{\xi v+w}^\sharp(X).$$ The desired result follows.
\end{proof}
\begin{lemma}\label{unit}Let $v\in V\setminus\{0\}$ and let $\CF=\{V_{i_0},\ldots, V_{i_m}\}$ be a subflag of $\OF$. Let $1\leq \ell\leq m$ be minimal such that $v\in V_{i_\ell}$. Then $Q^{i_\ell}_v(\UT)\in \Gamma(U_\CF,\CO_{U_\CF})^\times$.
\end{lemma}
\begin{proof}Let $s\in U_\CF$. Then by Proposition \ref{strata}, we know that $s\in\Omega_{\CF'}$ for some $\CF'\subset\CF$. Choose $0\leq j\leq \ell$ maximal such that $V_{i_j}\in\CF'$. Let $t=(t_1,\ldots,t_{n-1})$ denote the image of $T$ in $k(s)$. Then Lemma \ref{kpoints}.1 implies that
$$\prod_{i=i_{j'}}^{i_\ell-1}t_i=0,\mbox{ for all } j'\leq j.$$
It follows that $Q^{i_\ell}_w(\Ut)$ is a non-trivial $\BF_q$-linear combination of the elements
$$\Big\{\prod_{i=i_{j}+1}^{i_\ell-1}t_i,\prod_{i=i_{j}+2}^{i_\ell-1}t_i,\ldots,t_{i_\ell-1},1\Big\}.$$
By Lemma \ref{kpoints}.2, such a linear combination is non-zero, so $Q^k_w(\Ut)\neq 0$ in $k(s)$. It follows that $Q^k_w(\UT)\in (\CO_{U_\CF,s})^\times.$ Since $s$ was arbitrary, this proves the lemma.
\end{proof}

Let $\CF_0:=\{0,V\}$ be the trivial flag. Composing each $\lambda_v$ with $f_V$, we obtain a $V$-marking $\lambda_{\CF_0}\colon V\to \CC_{\CF_0}(U_{\CF_0})$. This marking is compatible with the action of $G$ by Lemma \ref{AffSects} and the fact that $f_V$ is $G$-equivariant. 
For a scheme $S$ and $\UT\in U_{\OF}(S)$, we observe that the point $$\bigl(\UT,(1:0)_{(v,w)\in\Sigma}\bigr)\in (U_\OF\times P^\Sigma)(S)$$ satisfies (\ref{defeqs}). We may thus extend $\lambda_{\CF_0}$ to a $\hat V$-marking by defining
\begin{equation}\label{6-eq-triv-flag-inf-section}\lambda_{\CF_0}(\infty)\colon U_{\CF_0}\to \CC_{\CF_0},\,\,\,\,\,\UT\mapsto \bigl(\UT,(1:0)_{(v,w)\in\Sigma}\bigr).\end{equation} Since $G$ acts on $\CC_\OF$ by permuting components, we see that $\lambda_{\CF_0}(\infty)$ is fixed by the $G$-action, so the resulting $\hat V$-marking on $C_{\CF_0}$ is compatible with the action of $G$.
\begin{proposition}\label{Vsects} The $\hat V$-marking $\lambda_{\CF_0}$ on $C_{\CF_0}$ extends uniquely to a $\hat{V}$-marking
$$\lambda_\OF\colon \hat V\to \CC_\OF(U_\OF).$$
Furthermore, this $\hat V$-marking is compatible with the action of $G$.
\end{proposition}
\begin{proof}It follows from the definition of $f_V$ that the $0$-marked section of $\OV\times\BA^1\to\OV$ is mapped to the morphism $\OV\to C_{\CF_0}\subset \OV\times P^\Sigma$ given by $\bigl( -\lambda_v:\lambda_w\bigr)_{v,w}$.
Define a section
$$\lambda_\OF(0)\colon U_\OF\to \CC_\OF\subset U_\OF\times P^{\Sigma}$$
as follows: For each $(v,w)\in\Sigma$, the $(v,w)$-component of $\lambda_\OF(0)$ is the morphism $U_\OF\to U_\OF\times\BP^1$ given by $\UT\mapsto \bigl(-Q^\ell_v(\UT):Q^\ell_{w}(\UT)\bigr),$ where $\ell\in\{1,\ldots,n\}$ is minimal such that $v, w\in V_\ell$. By Lemma \ref{unit}, the choice of $\ell$ insures that either $Q^\ell_v(\UT)$ or $Q^\ell_w(\UT)$ is an element of $\Gamma(U_{\OF},\CO_{U_\OF}^\times)$; hence $\lambda_\OF(0)$ is well-defined. The inclusion $\OV\into U_\CF\into \BA^{n-1}_{\BF_q}$ maps $\frac{\lambda_v}{\lambda_{b_n}}$ to $Q^n_v(\UT)$ for all $v\in V$. Note that $$Q^n_v(\UT)=Q^\ell_v(\UT)\cdot \prod_{i=\ell}^{n-1}T_i,$$ and that $\prod_{i=\ell}^{n-1}T_i$ is invertible on $\OV$. It follows that $\bigl(-Q^\ell_v(\UT):Q^\ell_{w}(\UT)\bigr)$ and  $\bigl(-\lambda_v:\lambda_w\bigr)$ define the same morphism $\OV\to \BP^1.$ Hence $\lambda_\OF(0)$ extends the zero section on $C_{\CF_0}$, as desired.

Let $\phi_\OF$ denote the left $G$-action on $\CC_\OF$ as described in Proposition \ref{Action}. For each $v\in V$, we define $$\lambda_\OF(v):=\phi_\OF(v)\circ \lambda_\OF(0).$$ This extends the $v$-marked section on $C_{\CF_0}$. We define the infinity section 
$\lambda_\OF(\infty)\colon U_\OF\to \CC_\OF$
 on each component via $\UT\mapsto (1:0)$. This clearly extends  the $\infty$-section (\ref{6-eq-triv-flag-inf-section}) on $\CC_{\CF_0}$, and we thus obtain a $\hat V$-marking $\lambda_\OF$ extending $\lambda_{\CF_0}$. The uniqueness follows from the fact that $\CC_{\CF_0}$ is dense in $\CC_\OF$, and the latter is separated over $U_\CF$ and reduced (see Corollary \ref{CFProps}). The compatibility of the $\hat V$-marking with the action of $G$ then follows from the uniqueness of the extended sections and the compatibility of the $V$-marking on $\CC_{\CF_0}$.
\end{proof}

\subsection{Elimination of redundant equations}\label{redund}
Let $\CF:=\{V_{i_0},\ldots,V_{i_m}\}$ be a subflag of $\OF$. For each $v=\sum_{i=1}^n c_ib_i\in V$ and $1\leq k\leq n$, we define 
\begin{eqnarray*}v^{\leq k}&:=&\sum_{i=1}^k c_ib_i\in V_k,
\\v^{>k}&:=&\sum_{i=k+1}^n c_i b_i=v-v^{\leq k}.
\end{eqnarray*} 
Consider the set $$\Sigma_\CF:=\{(v,k)\in V\times \{1,\ldots,m\}\mid v^{\leq i_k}=0\},$$
and define
 $$P^{\Sigma_\CF}:=\prod_{(v,k)\in\Sigma_\CF}\BP^1.$$ 
For each $(v,k)\in\Sigma_\CF$, we reindex the $(v,b_{i_k})$-component of $ P^\Sigma$ by $(v,k)$. We then have the natural projection $p_\CF\colon P^\Sigma\to P^{\Sigma_\CF}$.
Write $(X_{vk}:Y_{vk})$ for the projective coordinates on the $(v,k)$-component of $ P^{\Sigma_\CF}$. These correspond to the projective coordinates on the $(v,b_{i_k})$-component of $ P^\Sigma$, which we previously denoted by $(X_{vb_{i_k}}:Y_{vb_{i_k}})$. The next lemma simplifies the situation by reducing the number of equations defining $\CC_\CF$. 
\begin{proposition} \label{RelComps}The composite
 $$\xymatrix{\CC_\CF\ar@{^`->}[r]& U_\CF\times P^\Sigma\ar[r]^{p_\OF}& U_\CF\times  P^{\Sigma_\CF}}$$ 
 is an isomorphism onto the closed subscheme given by the equations
\begin{equation}\label{eqsFlag}Q^{i_\ell}_{b_{i_k}}(\UT)X_{vk}Y_{v'k'}+Q^{i_\ell}_{v-v'}(\UT)Y_{vk}Y_{v'k'}=Q^{i_\ell}_{b_{i_{k'}}}(\UT)X_{v'k'}Y_{vk},
\end{equation} where $1\leq (k', k)\leq\ell\leq m$ and $v-v'\in V_{i_\ell}$.
\end{proposition}
\begin{proof} For simplicity in the indices, we only provide the proof in the case where $\CF=\overline\CF$. The general case is similar. Consider $(v,w)\in \Sigma$. There exists a minimal $k$ such that $w\in V_k$. Let
$v':=v^{>k}$ and consider the equation in (\ref{defeqs}) corresponding to $(v',b_k)$ and $(v,w)$:
\begin{equation}\label{canceleq}X_{v'k}Y_{vw}+Q^k_{v'-v}(\UT)Y_{v'k}Y_{vw}=Q^k_w(\UT)X_{vw}Y_{v'k}.
\end{equation}
By Lemma \ref{unit}, we have $Q^k_w(\UT)\in \Gamma(U_\CF,\CO_{U_\CF})^\times$. We may thus rewrite (\ref{canceleq}) as
\begin{equation}\label{canceleqII}Y_{vw}\bigl(X_{v'k}+Q^k_{v'-v}(\UT)Y_{v'k}\bigr)\cdot\bigl(Q^k_w(\UT)\bigr)^{-1}=X_{vw}Y_{v'k}.
\end{equation}
Let $S$ be a $U_\CF$-scheme. A morphism $S\to \CC_\OF$ over $U_\CF$ corresponds to an isomorphism a set of tuples $\{\bigl(\CL_{rs},(x_{rs}:y_{rs})\bigr)\}_{(r,s)\in\Sigma}$, where each $\CL_{rs}$ is  an invertible sheaf  on $S$ with generating global sections $x_{rs}$ and $y_{rs}$. The $(x_{rs}:y_{rs})$ satisfy (\ref{defeqs}) and, in particular, equation (\ref{canceleqII}). Define $$\tilde x_{v'k}:=\bigl(x_{v'k}+Q^k_{v'-v}(\UT)y_{v'k}\bigr)\cdot\bigl(Q^k_w(\UT)\bigr)^{-1}\in \Gamma(S,\CL_{v'k}).$$ Then $\tilde x_{v'k}$ and $y_{v'k}$ also generate $\CL_{v'k}$. Moreover, equation (\ref{canceleqII}) implies that for all $p\in S$, we have $(x_{vw})_p\in \Fm_p\CL_{vw,p}$ if and only if $(\tilde x_{v'k})_p\in\Fm_p\CL_{v'k,p}$, and similarly for $(y_{vw})_p$ and $(y_{v'k})_p.$ Let $S_x\subset S$ be the open subscheme where $x_{vw}$ (resp. $\tilde x_{v'k}$) generates $\CL_{vw}$ (resp. $\CL_{v'k}$). Define $S_y$ analogously and let $S_{xy}:=S_x\cap S_y$. By (\ref{canceleqII}), the following diagram commutes:
$$\xymatrix{\CO_{S_x}\cdot x_{vw}|_{S_{xy}}\ar[r]^\sim\ar[d]^{\cdot y_{vw}x_{vw}^{-1}}&  \CO_{S_x}\cdot \tilde x_{v'k}|_{S_{xy}}\ar[d]^{\cdot y_{v'k}\tilde x_{v'k}^{-1}}
\\\CO_{S_y}\cdot y_{vw}|_{S_{xy}}\ar[r]^\sim&  \CO_{S_y}\cdot y_{v'k}|_{S_{xy}}.}
$$
We thus obtain an isomorphism $\CL_{vw}\stackrel\sim\to\CL_{v'k}$ sending $(x_{vw}:y_{vw})$ to $(\tilde x_{v'k}:y_{v'k})$. It follows that the $(v,w)$-component of the morphism $S\to \CC_\CF$ is determined by the $(v',k)$-component. Hence $\CC_\CF$ is an isomorphism onto its image in $U_\CF\times P^{\Sigma_\CF}$. The image is determined by the subset of equations in (\ref{defeqs}) involving the elements of $\Sigma_\CF$, which yields (\ref{eqsFlag}).
\end{proof}
\begin{lemma}\label{immersion}The morphism $f_V$ induces an open immersion $\OV\times \BA^1\into C_{\CF_0}$.
\end{lemma}
\begin{proof}For the trivial flag $\CF_0$, we have $\Sigma_{\CF_0}=\{(0,1)\}$ and Proposition \ref{RelComps} yields an isomorphism $\rho_0\colon C_{\CF_0}\stackrel\sim\to\OV\times\BP^1$. The composite $\rho_0\circ f_V$ corresponds to the morphism $\OV\times \BA^1\to \OV\times\BP^1$ given by $$(\lambda, t)\mapsto \Bigl(\lambda,\bigl(\lambda_{v_0}t:\lambda_{b_n}\bigr)\Bigr).$$ Here $\lambda$ is a fiberwise injective map $V\to \CL(S)$ for a scheme $S$ and invertible sheaf $\CL$ on $S$. Since $\lambda$ is fiberwise injective, the invertible sheaf $\CL$ is trivial, and we may assume that $\CL=\CO_S$ and $\frac{\lambda_{b_n}}{\lambda_{v_0}}=1$. Thus $\rho_0\circ f_V$ is the open immersion $\OV\times\BA^1\into \OV\times\BP^1$ given by $t\mapsto (t:1)$. It follows that $f_V$ induces an open immersion $\OV\times\BA^1\into C_{\CF_0}$, as desired.
\end{proof}

\subsection{Properties of $\CC_\OF$}
We begin by defining particular open subschemes of $\CC_\OF$ whose translates under the $V$-action form an open covering $\CU$ of $\CC_\OF$ (Lemma \ref{opencov}). The cover $\CU$ will be useful in establishing flatness over $U_\OF$, reducedness and several other properties of $\CC_\OF$ (Corollary \ref{CFProps}).
 Let $\CF=\{V_{i_0},\ldots,V_{i_{m}}\}\subset \OF$ be an arbitrary subflag. Recall that $U_\CF$ is an open subscheme of $U_{\OF}$, and $\pi_\OF$ restricts to the projection $\pi_\CF\colon \CC_\CF\to U_\CF.$ We identify $\CC_\CF$ with its image in $U_\CF\times P^{\Sigma_\CF}$ via the isomorphism in Proposition \ref{RelComps}. Fix $1\leq k\leq m$. For each $v\in V$, let $k\leq k^v\leq m$ be minimal such that $v\in V_{i_{k^v}}$. Let $W_\CF^k\subset \CC_\CF\subset \CC_\OF$ be the open subscheme defined by
 \begin{eqnarray}\label{smoothnbhd}Q^{i_{k^v}}_{b_{i_k}}(\UT)X_{0k}&\neq& Q^{i_{k^v}}_{v}(\UT)Y_{0k},\mbox{ for all $v\in V$,}\\ \nonumber
 Y_{0k}&\neq& 0,\mbox{ if $k<m$.}
 \end{eqnarray}
Similarly, let $Z_\CF^k\subset \CC_\CF\subset \CC_\OF$ be the open subscheme defined by
 \begin{eqnarray}\label{nodenbhd}
Q^{i_{k^v}}_{b_{i_k}}(\UT)X_{0k}&\neq& Q^{i_{k^v}}_v(\UT)Y_{0k},\mbox{ for all $v\in V\setminus V_{i_{k-1}}$},\\\nonumber
X_{0k-1}&\neq& Q_{v}^{i_{k-1}}(\UT)Y_{0k-1},\mbox{if $k>1$ and for all $v\in V_{i_{k-1}}$},
\\ \nonumber Y_{0k}&\neq& 0.
\end{eqnarray}
\begin{remark}To motivate the definitions of $W^k_\CF$ and $Z^k_\CF$, we provide the following description on the fiber $C_s$ over a point $s\in \Omega_\CF\subset U_\OF$: We will show that $C_s$  is a nodal curve of genus 0 with irreducible components indexed by the elements of $\Sigma_\CF$ (Lemma \ref{irredcomps}). If $k>1$, the open subscheme $W_\CF^k$ is defined so that the base change to $\Spec k(s)$ is precisely the smooth part of the $(0,k)$-component of $C_s$. Similarly, the base change of $Z_\CF^k$ to $\Spec k(s)$ consists of the $(0,k)$- and $(0,k-1)$-components minus their singular points, with the exception of the unique nodal point at which they intersect. The base change of $Z^\CF_1$ is the smooth part of the $(0,1)$-component minus the $(V_{i_1}\setminus\{0\})$-marked points.
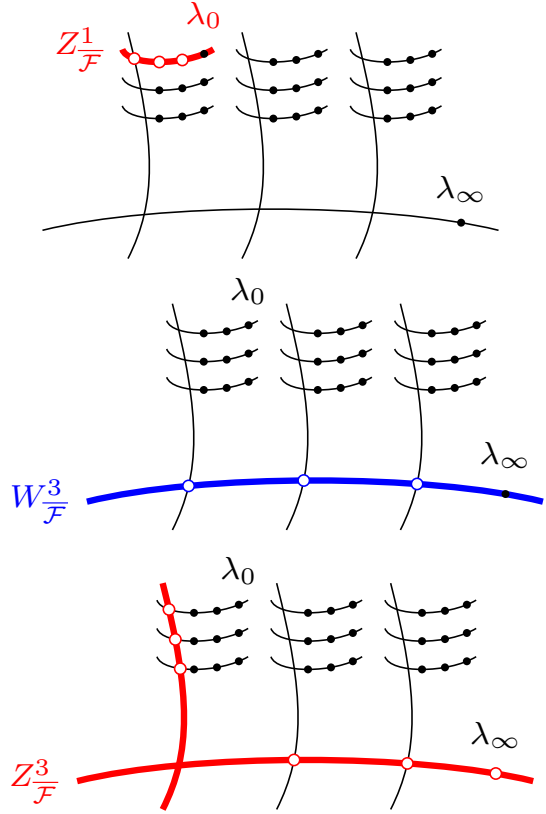
\begin{figure}[H]
\centering
  \caption{Examples for $\dim V=3$, $q=3$ over a point $s\in\Omega_{\OF}$. The open subschemes $W^k_\OF$ and $Z^k_\OF$ are represented by thick lines. Unfilled points are not included.}
\begin{tikzpicture}
  \draw[name path=E03] (0,0)  .. controls (1,.25) and (3,.25) .. 
  (4,0) 
  node[color=black,circle,pos=0.9, fill=black,inner sep=0pt,minimum size=2pt, label={above:{\scriptsize$\lambda_\infty$}}]{};

  \draw[name path=E02] (.75,-0.25) .. controls (1,.25) and (1,0.75) .. (.75,1.75)
  ;
    \draw[name path=E12] (1.75,-0.25) .. controls (2,.25) and (2,0.75) .. (1.75,1.75);
     \draw[name path=E22] (2.75,-0.25) .. controls (3,.25) and (3,0.75) .. (2.75,1.75);

     \draw[ultra thick, color=red, name path=E01] (.7,1.6) .. controls (.75,1.45) and (1.25,1.45) .. (1.5,1.6)
     node[at start,left] {\scriptsize $Z^1_{\overline{\mathcal{F}}}$}   	
     	 node[circle,pos=0.9, fill=black,inner sep=0pt,minimum size=2pt, label=above:{\scriptsize$\lambda_0$}]{}
       	node[thin, circle, draw,pos=0.7, fill=white,inner sep=1pt,minimum size=2pt,]{}
       node[thin, draw, circle,pos=0.5, fill=white,inner sep=1pt,minimum size=2pt, ]{};        
    \draw[name path=E11] (.7,1.35) .. controls (.75,1.20) and (1.25,1.20) .. (1.5,1.35)
       		node[circle,pos=0.9, fill=black,inner sep=0pt,minimum size=2pt,]{}
       		node[circle,pos=0.7, fill=black,inner sep=0pt,minimum size=2pt,]{}
       		node[circle,pos=0.5, fill=black,inner sep=0pt,minimum size=2pt,]{};
     \draw[name path=E21] (.7,1.1) .. controls (.75,.95) and (1.25,.95) .. (1.5,1.1)
     		node[circle,pos=0.9, fill=black,inner sep=0pt,minimum size=2pt,]{}
       		node[circle,pos=0.7, fill=black,inner sep=0pt,minimum size=2pt,]{}
       		node[circle,pos=0.5, fill=black,inner sep=0pt,minimum size=2pt,]{};;

         \draw (1.7,1.6) .. controls (1.75,1.45) and (2.25,1.45) .. (2.5,1.6)
         		node[circle,pos=0.9, fill=black,inner sep=0pt,minimum size=2pt,]{}
       		node[circle,pos=0.7, fill=black,inner sep=0pt,minimum size=2pt,]{}
       		node[circle,pos=0.5, fill=black,inner sep=0pt,minimum size=2pt,]{};
        \draw (1.7,1.35) .. controls (1.75,1.20) and (2.25,1.20) .. (2.5,1.35)
        		node[circle,pos=0.9, fill=black,inner sep=0pt,minimum size=2pt,]{}
       		node[circle,pos=0.7, fill=black,inner sep=0pt,minimum size=2pt,]{}
       		node[circle,pos=0.5, fill=black,inner sep=0pt,minimum size=2pt,]{};
        \draw (1.7,1.1) .. controls (1.75,.95) and (2.25,.95) .. (2.5,1.1)
        		node[circle,pos=0.9, fill=black,inner sep=0pt,minimum size=2pt,]{}
       		node[circle,pos=0.7, fill=black,inner sep=0pt,minimum size=2pt,]{}
       		node[circle,pos=0.5, fill=black,inner sep=0pt,minimum size=2pt,]{};
        
         \draw (2.7,1.6) .. controls (2.75,1.45) and (3.25,1.45) .. (3.5,1.6)
         		node[circle,pos=0.9, fill=black,inner sep=0pt,minimum size=2pt,]{}
       		node[circle,pos=0.7, fill=black,inner sep=0pt,minimum size=2pt,]{}
       		node[circle,pos=0.5, fill=black,inner sep=0pt,minimum size=2pt,]{};
        \draw (2.7,1.35) .. controls (2.75,1.20) and (3.25,1.20) .. (3.5,1.35)
        		node[circle,pos=0.9, fill=black,inner sep=0pt,minimum size=2pt,]{}
       		node[circle,pos=0.7, fill=black,inner sep=0pt,minimum size=2pt,]{}
       		node[circle,pos=0.5, fill=black,inner sep=0pt,minimum size=2pt,]{};
        \draw (2.7,1.1) .. controls (2.75,.95) and (3.25,.95) .. (3.5,1.1)
        		node[circle,pos=0.9, fill=black,inner sep=0pt,minimum size=2pt,]{}
       		node[circle,pos=0.7, fill=black,inner sep=0pt,minimum size=2pt,]{}
       		node[circle,pos=0.5, fill=black,inner sep=0pt,minimum size=2pt,]{};

	     \draw[name intersections={of=E02 and E01}] 
     (intersection-1) node[circle, draw, color=red, fill=white,inner sep=1pt,minimum size=2pt]{};

\end{tikzpicture}

\begin{tikzpicture}
  \draw[ultra thick, color=blue, name path=E03] (0,0)  .. controls (1,.25) and (3,.25) .. 
node[at start,left] {\scriptsize $W^3_{\overline{\mathcal{F}}}$} 
  (4,0) 
  node[color=black,circle,pos=0.9, fill=black,inner sep=0pt,minimum size=2pt, label={[black]above:{\scriptsize$\lambda_\infty$}}]{};

  \draw[name path=E02] (.75,-0.25) .. controls (1,.25) and (1,0.75) .. (.75,1.75)
  ;
    \draw[name path=E12] (1.75,-0.25) .. controls (2,.25) and (2,0.75) .. (1.75,1.75);
     \draw[name path=E22] (2.75,-0.25) .. controls (3,.25) and (3,0.75) .. (2.75,1.75);
     
     \draw[name intersections={of=E03 and E02}] 
     (intersection-1) node[circle, draw, color=blue, fill=white,inner sep=1pt,minimum size=2pt]{};
      \draw[name intersections={of=E03 and E12}] 
     (intersection-1) node[circle, draw, color=blue, fill=white,inner sep=1pt,minimum size=2pt]{};
      \draw[name intersections={of=E03 and E22}] 
     (intersection-1) node[circle, draw, color=blue, fill=white,inner sep=1pt,minimum size=2pt]{};

     \draw (.7,1.6) .. controls (.75,1.45) and (1.25,1.45) .. (1.5,1.6)
     	 node[circle,pos=0.9, fill=black,inner sep=0pt,minimum size=2pt, label=above:{\scriptsize$\lambda_0$}]{}
       	node[circle,pos=0.7, fill=black,inner sep=0pt,minimum size=2pt,]{}
       node[circle,pos=0.5, fill=black,inner sep=0pt,minimum size=2pt, ]{};        
    \draw (.7,1.35) .. controls (.75,1.20) and (1.25,1.20) .. (1.5,1.35)
       		node[circle,pos=0.9, fill=black,inner sep=0pt,minimum size=2pt,]{}
       		node[circle,pos=0.7, fill=black,inner sep=0pt,minimum size=2pt,]{}
       		node[circle,pos=0.5, fill=black,inner sep=0pt,minimum size=2pt,]{};
     \draw (.7,1.1) .. controls (.75,.95) and (1.25,.95) .. (1.5,1.1)
     		node[circle,pos=0.9, fill=black,inner sep=0pt,minimum size=2pt,]{}
       		node[circle,pos=0.7, fill=black,inner sep=0pt,minimum size=2pt,]{}
       		node[circle,pos=0.5, fill=black,inner sep=0pt,minimum size=2pt,]{};;

         \draw (1.7,1.6) .. controls (1.75,1.45) and (2.25,1.45) .. (2.5,1.6)
         		node[circle,pos=0.9, fill=black,inner sep=0pt,minimum size=2pt,]{}
       		node[circle,pos=0.7, fill=black,inner sep=0pt,minimum size=2pt,]{}
       		node[circle,pos=0.5, fill=black,inner sep=0pt,minimum size=2pt,]{};
        \draw (1.7,1.35) .. controls (1.75,1.20) and (2.25,1.20) .. (2.5,1.35)
        		node[circle,pos=0.9, fill=black,inner sep=0pt,minimum size=2pt,]{}
       		node[circle,pos=0.7, fill=black,inner sep=0pt,minimum size=2pt,]{}
       		node[circle,pos=0.5, fill=black,inner sep=0pt,minimum size=2pt,]{};
        \draw (1.7,1.1) .. controls (1.75,.95) and (2.25,.95) .. (2.5,1.1)
        		node[circle,pos=0.9, fill=black,inner sep=0pt,minimum size=2pt,]{}
       		node[circle,pos=0.7, fill=black,inner sep=0pt,minimum size=2pt,]{}
       		node[circle,pos=0.5, fill=black,inner sep=0pt,minimum size=2pt,]{};
        
         \draw (2.7,1.6) .. controls (2.75,1.45) and (3.25,1.45) .. (3.5,1.6)
         		node[circle,pos=0.9, fill=black,inner sep=0pt,minimum size=2pt,]{}
       		node[circle,pos=0.7, fill=black,inner sep=0pt,minimum size=2pt,]{}
       		node[circle,pos=0.5, fill=black,inner sep=0pt,minimum size=2pt,]{};
        \draw (2.7,1.35) .. controls (2.75,1.20) and (3.25,1.20) .. (3.5,1.35)
        		node[circle,pos=0.9, fill=black,inner sep=0pt,minimum size=2pt,]{}
       		node[circle,pos=0.7, fill=black,inner sep=0pt,minimum size=2pt,]{}
       		node[circle,pos=0.5, fill=black,inner sep=0pt,minimum size=2pt,]{};
        \draw (2.7,1.1) .. controls (2.75,.95) and (3.25,.95) .. (3.5,1.1)
        		node[circle,pos=0.9, fill=black,inner sep=0pt,minimum size=2pt,]{}
       		node[circle,pos=0.7, fill=black,inner sep=0pt,minimum size=2pt,]{}
       		node[circle,pos=0.5, fill=black,inner sep=0pt,minimum size=2pt,]{};
	
\end{tikzpicture}

\begin{tikzpicture}
  \draw[ultra thick, color=red, name path=E03] (0,0)  .. controls (1,.25) and (3,.25) .. 
node[at start,left] {\scriptsize $Z^3_{\overline{\mathcal{F}}}$} 
  (4,0) 
  node[thin,draw,color=red,circle,pos=0.9, fill=white,inner sep=1pt,minimum size=2pt, label={[black]above:{\scriptsize$\lambda_\infty$}}]{};

  \draw[name path=E02, ultra thick, color=red] (.75,-0.25) .. controls (1,.25) and (1,0.75) .. (.75,1.75)
  ;
    \draw[name path=E12] (1.75,-0.25) .. controls (2,.25) and (2,0.75) .. (1.75,1.75);
     \draw[name path=E22] (2.75,-0.25) .. controls (3,.25) and (3,0.75) .. (2.75,1.75);

      \draw[name intersections={of=E03 and E12}] 
     (intersection-1) node[circle, draw, color=red, fill=white,inner sep=1pt,minimum size=2pt]{};
      \draw[name intersections={of=E03 and E22}] 
     (intersection-1) node[circle, draw, color=red, fill=white,inner sep=1pt,minimum size=2pt]{};

     \draw[name path=E01] (.7,1.6) .. controls (.75,1.45) and (1.25,1.45) .. (1.5,1.6)
     	 node[circle,pos=0.9, fill=black,inner sep=0pt,minimum size=2pt, label=above:{\scriptsize$\lambda_0$}]{}
       	node[circle,pos=0.7, fill=black,inner sep=0pt,minimum size=2pt,]{}
       node[circle,pos=0.5, fill=black,inner sep=0pt,minimum size=2pt, ]{};        
    \draw[name path=E11] (.7,1.35) .. controls (.75,1.20) and (1.25,1.20) .. (1.5,1.35)
       		node[circle,pos=0.9, fill=black,inner sep=0pt,minimum size=2pt,]{}
       		node[circle,pos=0.7, fill=black,inner sep=0pt,minimum size=2pt,]{}
       		node[circle,pos=0.5, fill=black,inner sep=0pt,minimum size=2pt,]{};
     \draw[name path=E21] (.7,1.1) .. controls (.75,.95) and (1.25,.95) .. (1.5,1.1)
     		node[circle,pos=0.9, fill=black,inner sep=0pt,minimum size=2pt,]{}
       		node[circle,pos=0.7, fill=black,inner sep=0pt,minimum size=2pt,]{}
       		node[circle,pos=0.5, fill=black,inner sep=0pt,minimum size=2pt,]{};;

         \draw (1.7,1.6) .. controls (1.75,1.45) and (2.25,1.45) .. (2.5,1.6)
         		node[circle,pos=0.9, fill=black,inner sep=0pt,minimum size=2pt,]{}
       		node[circle,pos=0.7, fill=black,inner sep=0pt,minimum size=2pt,]{}
       		node[circle,pos=0.5, fill=black,inner sep=0pt,minimum size=2pt,]{};
        \draw (1.7,1.35) .. controls (1.75,1.20) and (2.25,1.20) .. (2.5,1.35)
        		node[circle,pos=0.9, fill=black,inner sep=0pt,minimum size=2pt,]{}
       		node[circle,pos=0.7, fill=black,inner sep=0pt,minimum size=2pt,]{}
       		node[circle,pos=0.5, fill=black,inner sep=0pt,minimum size=2pt,]{};
        \draw (1.7,1.1) .. controls (1.75,.95) and (2.25,.95) .. (2.5,1.1)
        		node[circle,pos=0.9, fill=black,inner sep=0pt,minimum size=2pt,]{}
       		node[circle,pos=0.7, fill=black,inner sep=0pt,minimum size=2pt,]{}
       		node[circle,pos=0.5, fill=black,inner sep=0pt,minimum size=2pt,]{};
        
         \draw (2.7,1.6) .. controls (2.75,1.45) and (3.25,1.45) .. (3.5,1.6)
         		node[circle,pos=0.9, fill=black,inner sep=0pt,minimum size=2pt,]{}
       		node[circle,pos=0.7, fill=black,inner sep=0pt,minimum size=2pt,]{}
       		node[circle,pos=0.5, fill=black,inner sep=0pt,minimum size=2pt,]{};
        \draw (2.7,1.35) .. controls (2.75,1.20) and (3.25,1.20) .. (3.5,1.35)
        		node[circle,pos=0.9, fill=black,inner sep=0pt,minimum size=2pt,]{}
       		node[circle,pos=0.7, fill=black,inner sep=0pt,minimum size=2pt,]{}
       		node[circle,pos=0.5, fill=black,inner sep=0pt,minimum size=2pt,]{};
        \draw (2.7,1.1) .. controls (2.75,.95) and (3.25,.95) .. (3.5,1.1)
        		node[circle,pos=0.9, fill=black,inner sep=0pt,minimum size=2pt,]{}
       		node[circle,pos=0.7, fill=black,inner sep=0pt,minimum size=2pt,]{}
       		node[circle,pos=0.5, fill=black,inner sep=0pt,minimum size=2pt,]{};

	     \draw[name intersections={of=E02 and E01}] 
     (intersection-1) node[circle, draw, color=red, fill=white,inner sep=1pt,minimum size=2pt]{};
     \draw[name intersections={of=E02 and E11}] 
     (intersection-1) node[circle, draw, color=red, fill=white,inner sep=1pt,minimum size=2pt]{};
     \draw[name intersections={of=E02 and E21}] 
     (intersection-1) node[circle, draw, color=red, fill=white,inner sep=1pt,minimum size=2pt]{};
\end{tikzpicture}
  \label{fig:nbhds}
  \end{figure}
\end{remark}
Note that $k$ is fixed in (\ref{smoothnbhd}) and (\ref{nodenbhd}) while $v\in V$ varies. Considering the copy of $\BP^1$ in the $(0,k)$-component of $P^{\Sigma_\OF}$, we can thus make sense of taking the open subscheme of $\BP^1$ defined by (\ref{smoothnbhd}). Similary, looking at the copy of $\BP^1\times\BP^1$ in the $(0,k)$- and $(0,k-1)$-components of $P^{\Sigma_\OF}$, we can consider the open subscheme of $\BP^1\times\BP^1$ defined by (\ref{nodenbhd}).
\begin{lemma}\label{opennbhds} Fix $1\leq k\leq m$.
\begin{enumerate}[label=(\alph*)] 
\item The projection $\CC_\CF\to U_\CF\times\BP^1$ onto the $(0,k)$-component restricts to an isomorphism from $W_\CF^k$ to the open subscheme defined by (\ref{smoothnbhd}).
\item The projection $\CC_\CF\to U_\CF\times\BP^1$ onto the $(0,1)$-component restricts to an isomorphism from $Z_\CF^1$ to the open subscheme defined by (\ref{nodenbhd}).
\item If $k\geq 2,$ the projection $\CC_\CF\to U_\CF\times(\BP^1\times\BP^1)$ onto the $(0,k-1)$- and $(0,k)$-components restricts to an isomorphism from $Z_\CF^k$ to the locally closed subscheme defined by (\ref{nodenbhd}) and
$$X_{0k}Y_{0k-1}=Q^{i_k}_{b_{i_{k-1}}}(\UT)X_{0k-1}Y_{0k}.$$
\end{enumerate}
\end{lemma}
\begin{proof}
For (a), consider $(u,j)\in\Sigma_\CF\setminus\{(0,k)\}$. Recall that $k^u\geq k$ was defined to be minimal such that $u\in V_{i_{k^u}}$. Let $\ell:=\max(j,k^{u})$. Equation (\ref{eqsFlag}) yields
\begin{eqnarray}\label{releq}Q^{i_\ell}_{b_{i_k}}(\UT)X_{0k}Y_{uj}-Q^{i_\ell}_{u}(\UT)Y_{0k}Y_{uj}
=\\\nonumber\underbrace{\Bigl(Q^{i_\ell}_{b_{i_k}}(\UT)X_{0k}-Q^{i_\ell}_{u}(\UT)Y_{0k}\Bigr)}_{(*)}Y_{uj}=\underbrace{Q^{i_\ell}_{b_{i_{j}}}(\UT)Y_{0k}}_{(**)}X_{uj}.\end{eqnarray} We have the following two cases:
\vspace{2mm}

\noindent\emph{Case 1. $(\underline{\ell=k_u})$}: By construction, when $\ell=k^{u}$ then $(*)$ is invertible on $W_\CF^k$. Let $S$ be a $U_\CF$-scheme. A morphism $S\to W_\CF^k\subset \CC_\CF$ over $U_\CF$ is equivalent to an isomorphism class of a set of tuples $\{\bigl(\CL_{wi},(x_{wi}:y_{wi})\bigr)\}_{(w,i)\in\Sigma_\CF}$, where each $\CL_{wi}$ is an invertible sheaf over $S$ with global generating sections $x_{wi}$ and $y_{wi}$ satisfying  (\ref{eqsFlag}) and (\ref{smoothnbhd}). It follows that $(x_{0k}:y_{0k})$ and $(x_{uj}:y_{uj})$ satisfy (\ref{releq}). Since $(*)$ is invertible, this implies that $y_{uj}=F\cdot x_{uj}$ for some $F\in\CO_S(S)$ depending only on $(x_{0k}:y_{0k}).$ Since $x_{uj}$ and $y_{uj}$ generate $\CL_{uj}$, it follows that $x_{uj}$ vanishes nowhere. Hence $\CL_{uj}\cong \CO_S$. Without loss of generality, we may assume $x_{uj}=1$ and $y_{uj}=F.$ 
\vspace{5mm}

\noindent\emph{Case 2. $(\underline{\ell=j})$}: In this case, either $k=m=j$, which implies that $u=0$, or $k<m$. Since $(u,j)\neq (0,k)$ by assumption, the former does not occur. We observe that $Q^{i_j}_{b_{i_{j}}}(\UT)=1$ and that $Y_{0k}$ is invertible on $W_\CF^k$ by construction. It follows that $(**)$ is invertible on $W_\CF^k$. An analogous argument as in case 1 shows that we may take $\CL_{uj}=\CO_S$ and $y_{uj}=1$ and $x_{uj}=G$, where $G\in\CO_S(S)$ depends only on $(x_{0k},y_{0k})$. 
\vspace{5mm}

It follows that $(\CL_{uj},x_{uj},y_{uj})$ is determined by $(\CL_{0k},x_{0k},y_{0k})$ for all $(u,j)\neq (0,k)$. The composite 
$$C_{\CF}\into U_{\CF}\times P^{\Sigma_\CF}\to U_{\CF}\times\BP^1$$ therefore restricts to an isomorphism from $W_\CF^k$ to its image. The image is precisely the open subscheme defined by (\ref{smoothnbhd}), as desired. The proof of (b) is analogous.

\vspace{2mm}
For (c), consider $(u,j)\in\Sigma_\CF\setminus\{(0,k-1),(0,k)\}$. 
Let $S$ be a $U_\CF$-scheme. If $u\in V\setminus V_{i_{k-1}}$, then the same argument as in (a) shows that the $(u,j)$-component of a morphism $S\to Z_\CF^k$ over $U_\CF$ is determined by the $(0,k)$-component. Now suppose $u\in V_{i_{k-1}}.$ If $j\geq k$, then, since $u^{\leq i_j}=0$ by assumption, it follows that $u=0$. From (\ref{eqsFlag}) we obtain the following:
$$Q^{i_j}_{b_{i_k}}(\UT)X_{0k}Y_{0j}=X_{0j}Y_{0k}.$$ Since $Y_{0k}$ is invertible over $Z_\CF^k$, we apply a similar argument is in the proof of (a) to deduce that the $(0,j)$-component of a morphism $S\to Z_\CF^k$ over $U_\CF$ is determined by the $(0,k)$-component. 

If $j<k$, then (\ref{eqsFlag}) yields
$$\underbrace{(X_{0k-1}-Q^{i_{k-1}}_u(\UT)Y_{0k-1})}_{(*)}Y_{uj}=Q^{i_{k-1}}_{b_{i_j}}(\UT)X_{uj}Y_{0k-1}.$$
By construction $(*)$ is invertible on $Z_\CF^k$, and we once again apply a similar argument as in the proof of (a) to conclude that the $(u,j)$-component of a morphism $S\to Z_\CF^k$ over $U_\CF$ is determined by the $(0,k-1)$-component. All together, we have shown that any morphism $S\to Z_\CF^k$ over $U_\CF$ is determined by the $(0,k)$- and $(0,k-1)$-components. The image is defined by (\ref{nodenbhd}) in addition to the equation in (\ref{eqsFlag}) relating the $(0,k)$- and $(0,k-1)$-components. This proves (c).
\end{proof}

Let $s\in U_{\OF}$. The canonical morphism $\Spec k(s)\to U_{\OF}$ corresponds to a point $t\in k(s)^{n-1}$. Denote the fiber over $s$ of $\pi_\OF\colon \CC_{\OF}\to U_{\OF}$ by $C_s$. By Proposition \ref{strata}, there is a unique subflag $\CF=\{V_{i_0},\ldots,V_{i_{m}}\}$ of $\OF$ such that $s\in\Omega_{\CF}\subset U_{\OF}$. 

For each $1\leq k\leq m$, let $$W_s^k:=W_\CF^k\times\Spec k(s)\subset C_s,$$ 
and
 $$Z_s^k:=Z_\CF^k\times\Spec k(s)\subset C_s.$$
The $G$-action (Proposition \ref{Action}) on $\CC_{\OF}$ induces an action
$$\phi_s\colon G\to \Aut_{k(s)}(C_s).$$
\begin{lemma}\label{pointcov}
Let $p\in C_s$ be a closed point. Then there exist $v\in V$ and $1\leq k\leq m$ such that $p\in \phi_s(v)(W_s^k)$ or $p\in \phi_s(v)(Z_s^k)$.
\end{lemma}
\begin{proof}The morphism $$\Spec k(p)\to C_s\into P^{\Sigma_\CF}_s$$ corresponds to a point $$(x_{w\ell}:y_{w\ell})_{(w,\ell)\in\Sigma_\CF}\in  P^{\Sigma_{\CF}}_{s}\bigl(k(p)\bigr),$$
where the $(x_{w\ell},y_{w\ell})\in k(p)^2\setminus \{(0,0)\}$ and $\bigl(t,(x_{w\ell}:y_{w\ell})_{w\ell}\bigr)$ satisfies (\ref{eqsFlag}). If possible, choose $k$ minimal such that there exists a $(u,k)\in\Sigma_\CF$ with $y_{uk}\neq 0$. Otherwise, let $(u,k):=(0,n)$. Let $\tilde p:=\phi_s(-u)(p),$ with corresponding $(\tilde x_{w\ell}:\tilde y_{w\ell})_{w\ell}\in P^{\Sigma_{\CF}}_{s}\bigl(k(p)\bigr)$. The description of the $G$-action in \ref{SubSectAction} shows that $(\tilde x_{0k}:\tilde y_{0k})=(x_{uk},y_{uk})$.  It suffices to prove the lemma for $\tilde p$ in place of $p$, so we may suppose without loss of generality that $y_{0k}\neq 0$ and $y_{0\ell}=0$ for all $\ell<k$. We have the following two cases:
\vspace{2mm}

\noindent\emph{Case 1.} \underline{$\bigl(x_{0k}\neq Q^{i_k}_{v'}(\Ut)y_{0k}$ for all $v'\in V_{i_k}\bigr)$}: Using the defining equations (\ref{smoothnbhd}) of $W_\CF^k$, we see that $p\in W_s^k$ if and only if
 \begin{eqnarray}\label{smoothnbhdfib}Q^{i_{k^v}}_{b_{i_k}}(\Ut) x_{0k}&\neq& Q^{i_{k^v}}_{v}(\Ut) y_{0k},\mbox{ for all $v\in V$,}\\ \nonumber
 y_{0k}&\neq& 0,\mbox{ if $k<m$.}
 \end{eqnarray}
Let $v\in V$. Recall that we defined $k^v\geq k$ to be minimal such that $v\in V_{i_{k^v}}$. If $k_v=k$, then (\ref{smoothnbhdfib}) says that $ x_{0k}\neq Q^{i_k}_{v}(\Ut) y_{0k}$, which is true by assumption. If $k_v>k$, then Lemma \ref{kpoints}.1 implies that $Q^{i_{k^v}}_{b_{i_k}}(\Ut)=0$, so (\ref{smoothnbhdfib}) reads $Q^{i_{k^v}}_{v}(\Ut) y_{0k}\neq 0$. By Lemma \ref{kpoints}.2, we have $Q^{i_{k^v}}_{v}(\Ut)\neq 0$. Since $y_{0k}\neq 0$ by assumption, it follows that (\ref{smoothnbhdfib}) holds. Since $v$ was arbitrary, we conclude that $p\in W^k_s$.
\vspace{5mm}

\noindent\emph{Case 2.} \underline{$\bigl(x_{0k}= Q^{i_k}_{v'}(\Ut) y_{0k}$ for some $v'\in V_{i_k}\bigr)$}: Let $p':=\phi_s(-v')(p)$, with corresponding $(x'_{w\ell}:y'_{w\ell})_{(w,\ell)\in\Sigma_\CF}\in P^{\Sigma_\CF}_s\bigl(k(p)\bigr)$. From the defining equations (\ref{nodenbhd}) of $Z_\CF^k$, we deduce that $p'\in Z_s^k$ if and only if  
\begin{eqnarray}\label{nodenbhdfib}
Q^{i_{k^v}}_{b_{i_k}}(\Ut)x'_{0k}&\neq& Q^{i_{k^v}}_v(\Ut)y'_{0k},\mbox{ for all $v\in V\setminus V_{i_{k-1}}$},\\\nonumber
x'_{0k-1}&\neq& Q_{v}^{i_{k-1}}(\Ut)y'_{0k-1},\mbox{if $k>1$, and for all $v\in V_{i_{k-1}}$},
\\ \nonumber y'_{0k}&\neq& 0.
\end{eqnarray}
Considering $C_s$ as a closed subscheme of $P^\Sigma_s$, the morphism $\Spec k(p)\to C_s$ with image $p$ corresponds to some $(x_{vw}:y_{vw})_{(v,w)\in\Sigma}\in P^\Sigma_s\bigr(k(p)\bigl)$ such that $(x_{wb_{i_\ell}}:y_{wb_{i_\ell}})=(x_{w\ell}:y_{w\ell})$ for all $(w,\ell)\in\Sigma_\CF$. By (\ref{defeqs}), we have 
$$\underbrace{(x_{0k}-Q^{i_k}_{v'}(\Ut)y_{0k})}_{(*)}y_{v'b_{i_k}}=x_{v'b_{i_k}}y_{0k}.$$
By assumption $(*)$ is zero. We thus have $x_{v'b_{i_k}}y_{0k}=0$.  It follows that $(x_{0k}:y_{0k})=(0:1)$ and that $$(x'_{0k}:y'_{0k})=(x_{v'b_{i_k}}:y_{v'b_{i_k}})=(0:1).$$ Our choice of $k$ implies that $y'_{0k-1}=0$ if $k>1$. Moreover, Lemma \ref{kpoints} implies that $Q_v^{i_{k^v}}(\Ut)\neq 0$. We then see directly that (\ref{nodenbhdfib}) holds. Thus $p'\in Z_s^k$.
\vspace{5mm}

It follows that every closed point $p\in C_s$ is contained in a neighborhood of the form perscribed in the lemma, as desired.
\end{proof}
Recall that we defined a left $G$-action $\phi_\OF$ on $\CC_{\OF}$. For each subflag $\CF\subset\OF$, let $m_\CF$ denote the length of $\CF$.
\begin{lemma}\label{opencov} The set
$$\CU:=\{ \phi_\OF(v)\bigl(W_\CF^k\bigr), \phi_\OF(v)\bigl(Z_\CF^k\bigr)\mid v\in V\mbox,\,\CF\subset\OF,\mbox{ and }1\leq k\leq m_\CF\}$$ is an open cover of $\CC_{\OF}$.
\end{lemma}
\begin{proof} For each $s\in U_{\OF}$, let $$\CU_s:=\{U\times_{U_{\OF}}{k(s)}\mid U\in \CU\}.$$ Lemma \ref{pointcov} implies that every closed point in the fiber $C_s$ is contained in an element of $\CU_s$. Since $C_s$ is of finite type over $k(s)$, the closed points in $C_s$ are dense (\cite{GW}, Proposition 3.35). Hence $\CU_s$ is an open cover of $C_s$ for each $s\in S$. This implies that $\CU$ is an open cover of $\CC_{\OF}$.
\end{proof}
\begin{proposition}\label{localisos} The projection $\CC_{\OF}\to U_{\OF}$ is Zariski locally isomorphic to one of the following:
\begin{enumerate}[label=(\alph*)]
\item The projection $\BA_{\BF_q}^n\to\BA_{\BF_q}^{n-1}$ onto the first $n-1$ components.
\item The composite $$Z\into\BA_{\BF_q}^{n+1}=\Spec\BF_q[\UT,X,Y]\to\BA^{n-1}_{\BF_q}=\Spec \BF_q[\UT],$$ where $Z$ is the closed subscheme defined by the equation
\begin{equation}\label{b}XY=P(\UT),\end{equation}
for some $P(\UT)\in\BF_q[\UT]$.
\end{enumerate}
\end{proposition}
\begin{proof} Let $\CF=\{V_{i_0},\ldots, V_{i_m}\}\subset \OF$ be a subflag and fix $1\leq k\leq m$. Since the projection $U_\CF\times \BP^1\to U_\CF$ is locally isomorphic to the projection $\BA^n_{\BF_q}\to\BA^{n-1}_{\BF_q}$, the same holds for $W_\CF^k\to U_\CF$ and $Z_\CF^1\to U_\CF$ by Lemma \ref{opennbhds}.(a),(b).

By (\ref{eqsFlag}), the equation 
$$X_{0k}Y_{0k-1}=Q^{i_k}_{b_{i_k}}(\UT)X_{0k-1}Y_{0k}$$ holds on $\CC_\OF$. On $Z_\CF^k$ we may divide by $Y_{0k}$ and $X_{0k-1}.$ With $X:=X_{0k}/Y_{0k}$ and $Y:=Y_{0k-1}/X_{0k-1}$, this yields 
\begin{equation}\label{nodeeq}XY=Q^{i_k}_{b_{i_{k-1}}}(\UT).
\end{equation}
 It follows from Lemma \ref{opennbhds}(c) that $Z_\CF^k$ is locally isomorphic to the subscheme of $$U_\CF\times\bigl(D(X_{0k-1})\cap D(Y_{0k})\bigr)\subset U_\CF\times(\BP^1\times\BP^1)$$ defined by (\ref{nodeeq}). We apply Lemma \ref{opencov} to conclude the proof.
\end{proof}

\begin{corollary}\label{CFProps}The scheme $\CC_{\OF}$ has the following properties:
\begin{enumerate}[label=(\alph*)]
\item $\dim_{\CC_\OF}=n$;
\item $\CC_\OF$ is reduced;
\item $\CC_\OF$ is irreducible;
\item $\CC_\OF$ is flat over $U_\OF$;
\end{enumerate}
\end{corollary}
\begin{proof}By Proposition \ref{localisos}, the scheme $\CC_\OF$ admits an open covering by $n$-dimensional, reduced subschemes. This implies (a) and (b). Moreover, each $U\in \CU$ has a non-empty intersection with $C_{\CF_0}\cong \OV\times\BP^1$. Since the latter is irreducible, it follows that any two elements of $\CU$ have a non-empty intersection. A scheme is irreducible if and only if it admits a covering by irreducible open subschemes whose pairwise intersections are non-empty (\cite{Stacks}, Tag 01OM). Thus the irreducibility of each $U\in\CU$ implies (c). Finally, property (d) follows from the fact that both of the morphisms in Proposition \ref{localisos} are flat.
\end{proof}

\subsection{The fibers of $\CC_\OF\to U_\OF$}
Let $s\in U_\OF$. In our chosen coordinates on $U_\OF$, the point $s$ corresponds to some $\Ut\in k(s)^{n-1}$.
\begin{proposition}\label{geomred}
The fiber $C_s$ is a geometrically reduced curve with at worst nodal singularities.
\end{proposition}
\begin{proof}By Proposition \ref{localisos}, the fiber $C_s$ is locally isomorphic to $\BA^1_{k(s)}$ or the curve $Z$ in $\BA^2_{k(s)}$ defined by $XY=Q^k_{b_j}(\Ut)$, for some $1\leq j<k\leq n$. Since both $\BA^1_{k(s)}$ and $Z$ are geometrically reduced curves, the same holds for $C_s$. Moreover, the curve $Z$ is singular if and only if $Q^k_{b_j}(\Ut)=0$, in which case $Z$ is clearly nodal.
\end{proof}
For a unique $\CF=\{V_{i_0},\ldots, V_{i_m}\}\subset \OF$, we have $s\in\Omega_\CF\subset U_\OF$. We identify $C_s$ with its image in $P^{\Sigma_\CF}_s$ via Proposition \ref{RelComps}.
\begin{lemma}\label{fibeqs}The fiber $C_s$ is the closed subscheme of $P^{\Sigma_{\CF}}_{s}$ defined by the equations
\begin{equation}\label{I}\big(X_{vk}-Q^{i_k}_{v'^{\leq i_k}}(\Ut)Y_{vk}\big)Y_{v'k'}=0,
\end{equation}
for $1\leq k'<k\leq m$ and $v\equiv v'\mod V_{i_k}$, and
\begin{equation}\label{II}Y_{vk}Y_{v'k'}=0,\end{equation}
for $1\leq k'\leq k\leq m$ and $v\not\equiv v'\mod V_{i_k}$.
\end{lemma}
\begin{proof}
Proposition \ref{RelComps} implies that the fiber $C_s$ is the closed subscheme of $P^{\Sigma_\CF}_s$ cut out by the equations obtained from (\ref{eqsFlag}) by setting $T=t$:
\begin{equation}\label{eqsFlagPoint}Q^{i_\ell}_{b_{i_k}}(\Ut)X_{vk}Y_{v'k'}+Q^{i_\ell}_{v-v'}(\Ut)Y_{vk}Y_{v'k'}=Q^{i_\ell}_{b_{i_{k'}}}(\Ut)X_{vk}Y_{v'k'},
\end{equation} where $1\leq k'\leq k\leq\ell\leq m$ and $v-v'\in V_{i_\ell}$.
Consider distinct $(v,k),(v',k')\in\Sigma_\CF$ with $k'\leq k$. Write $v=\sum_{i=1}^n c_ib_i$ and $v'=\sum_{i=1}^n c_i'b_i$. We have the following possibilities:
\vspace{3mm}

\noindent\emph{Case 1.} \underline{$v\equiv v'$ mod $V_{i_k}$}: If $k'=k$, then (\ref{eqsFlagPoint}) is trivial. Otherwise, since $t_{i_{k'}}=0$ by Lemma \ref{kpoints}.1, we have $$Q^{i_k}_{b_{i_{k'}}}(\Ut)=\prod_{j=i_{k'}}^{i_k-1}t_j=0.$$ Then (\ref{eqsFlagPoint}) yields 
$$X_{vk}Y_{v'k'}+Q^{i_k}_{v-v'}(\Ut)Y_{vk}Y_{v'k'}=Q^{i_k}_{b_{i_{k'}}}(\Ut)X_{v'k'}Y_{vk}=0.$$
 Since $c_i=0$ for all $i\leq i_k$, we also have 
 $$Q^{i_k}_{v-v'}(\Ut)=\sum_{i=1}^{i_k}(c_i- c_i')\Big(\prod_{j=i}^{i_k-1}t_j\Big)=-\sum_{i=1}^{i_k}c_i'\Big(\prod_{j=i}^{i_k-1}t_j\Big)=-Q^{i_k}_{v'^{\leq i_k}}(\Ut).$$ 
 Combining these gives (\ref{I}).

\vspace{5mm}
\noindent\emph{Case 2.} \underline{$v\not\equiv v'$ mod $V_{i_k}$}: In this case there is a minimal $\ell>k$ such that $v\equiv v'$ mod $V_{i_\ell}$. Again by Lemma \ref{kpoints}.1, we have 
$$Q^{i_\ell}_{b_{i_k}}(\Ut)=Q^{i_\ell}_{b_{i_{k'}}}(\Ut)=0.$$
Equation (\ref{eqsFlagPoint}) then yields $Q^{i_\ell}_{v-v'}(\Ut)Y_{vk}Y_{v'k'}=0.$ By Lemma \ref{kpoints}.2, we have $Q^{i_\ell}_{v-v'}(\Ut)\neq 0$; hence we may divide by it to obtain (\ref{II}).
\end{proof}
For each $(w,\ell)\in\Sigma_\CF$, let $E_{w\ell}\subset P^{\Sigma_\CF}_s$ be the closed subscheme defined by:
\begin{eqnarray}\label{comps}
X_{vk}-Q^{i_{k}}_{w^{\leq i_k}}(\Ut)Y_{vk}&=&0,\mbox{\,\,\,\,\,if $k>\ell$ and $v\equiv w$ mod $V_{i_{k}}$},\\
\nonumber Y_{vk}&=&0,\mbox{\,\,\,\,\,for all other $(v,k)\neq (w,\ell)$}.
\end{eqnarray}

\begin{lemma}\label{compP1}The projection $P^{\Sigma_\CF}_s\to\BP^1_{k(s)}$ onto the $(w,\ell)$-component induces an isomorphism $E_{w\ell}\stackrel\sim\to\BP^1_{k(s)}$.
\end{lemma}
\begin{proof}The equations in (\ref{comps}) determine the $(v,k)$-component for all $(v,k)\neq (w,\ell)$ and place no constraints on the $(w,\ell)$-component.
\end{proof}

\begin{lemma}\label{irredcomps}The equality $$C_s=\bigcup_{(w,\ell)\in\Sigma_\CF} E_{w\ell}\subset P^{\Sigma_\CF}_s$$ holds.
\end{lemma}
\begin{proof}
We first observe that $\bigcup_{(w,\ell)\in\Sigma_\CF} E_{w\ell}$ is reduced. Since $C_s$ is also reduced by Corollary \ref{geomred}, it suffices to show equality on the level of closed points. Let $(w,\ell)\in\Sigma_\CF$ and let $P\in E_{w\ell}$ be a closed point, which corresponds to some $$p:=(x_{vk}:y_{vk})_{v,k}\in E_{w\ell}\bigl(k(P)\bigr)$$ satisfying (\ref{comps}). Fix $(v,k),(v',k')\in\Sigma_\CF$ distinct with $k'\leq k$. We have the following two cases:
\vspace{2mm}

\noindent\emph{Case 1.} \underline{$\bigl(v'\equiv v\mod V_{i_k}\bigr)$}: Since $(v,k)\neq(v',k')$, we must have $k'<k$. Suppose $\ell<k'$. If $v'\equiv w\mod V_{i_{k'}}$, then $v\equiv w\mod V_{i_k}$ and by (\ref{comps}) we have $$x_{vk}-Q^{i_k}_{w^{\leq i_k}}(\Ut)y_{vk}=0.$$   The equality $Q^{i_k}_{w^{\leq i_k}}(\Ut)=Q^{i_k}_{v'^{\leq i_k}}(\Ut)$ holds by Lemma \ref{kpoints}. It follows that 
$$\bigl(x_{vk}-Q^{i_k}_{v'^{\leq i_k}}(\Ut)y_{vk}\bigr)y_{v'k'}=0;$$ hence $p$ satisfies (\ref{I}). If $v'\not\equiv w\mod V_{i_{k'}}$ or $k'\leq\ell$, then (\ref{comps}) says that $y_{v'k'}=0$. Again we see that $p$ is a solution to (\ref{I}).
\vspace{3mm}

\noindent\emph{Case 2.} \underline{$\bigl(v'\not\equiv v\mod V_{i_k}\bigr)$}: Suppose $k'\leq\ell$. By (\ref{comps}) we have $y_{v'k'}=0.$ Hence $y_{vk}y_{v'k'}=0$ and $p$ is a solution to (\ref{II}). Now suppose $k'>\ell$. If $v'\equiv w\mod V_{i_{k'}}$, then $v\not\equiv w\mod V_{i_k}$ since otherwise $v'\equiv v\mod V_{i_k}$. Thus $y_{vk}=0$ by (\ref{comps}). If $v'\not\equiv w\mod V_{i_{k'}}$, then (\ref{comps}) says that $y_{v'k'}=0$. In either case, we have $y_{vk}y_{v'k'}=0$ and $p$ is a solution to (\ref{II}).
\vspace{3mm}

We conclude that $p$ satisfies (\ref{I}) and (\ref{II}) for all possible choices of $(v,k)$ and $(v',k')$. It follows from Lemma \ref{fibeqs} that $P\in C_s$. Thus $\bigcup_{(w,\ell)\in\Sigma_\CF}E_{w,\ell}\subset C_s$.
\vspace{4mm}

For the reverse inclusion, suppose $P\in C_s$ and take the corresponding $$p:=(x_{vk}:y_{vk})_{v,k}\in C_{\os}\bigl(k(P)\bigr).$$
 If $y_{vk}=0$ for all $(v,k)$, then $p$ is a solution to (\ref{comps}) for $(w,\ell)=(0,m)$ and thus $p\in (E_{0,m})\bigl(k(P)\bigr)$. Otherwise, let $1\leq\ell\leq m$ be minimal such that there exists a $w$ with $y_{w\ell}\neq 0$. By Lemma \ref{fibeqs}, we know that $p$ is a solution to (\ref{I}) and (\ref{II}). We observe that dividing the equations in (\ref{I}) and (\ref{II}) that involve $(w,\ell)$ by $Y_{w\ell}$ yields precisely the equations (\ref{comps}) defining $E_{w\ell}$. Hence $p\in E_{w\ell}\bigl(k(P)\bigr)$. It follows that $C_s\subset \bigcup_{(w,\ell)\in\Sigma_\CF} E_{w\ell}$. This proves the lemma.
\end{proof}
\begin{remark}The $E_{w\ell}$ for varying $(w,\ell)\in \Sigma_\CF$ are distinct. Indeed, suppose $(w,\ell)$ and $(w',\ell')$ are distinct elements of $\Sigma_\CF$. The $(w,\ell)$-component of any point $(x_{vk}:y_{vk})_{v,k}$ of $E_{w\ell}$ is free while the $(w,\ell)$-component of a point in $E_{w'\ell'}$ is fixed by (\ref{comps}). Hence $E_{w\ell}$ and $E_{w'\ell'}$ cannot be equal. The above lemma thus implies that the irreducible components of $C_s$ are indexed by the $\Sigma_\CF$.
\end{remark}
\begin{lemma}\label{int} For distinct pairs $(v,k)$ and $(v',k')$ with $k'\leq k$, we have the following:
$$E_{vk}\cap E_{v'k'}=\begin{cases}\mbox{a nodal point, if $k'=k-1$ and $v'\equiv v$ mod $V_k$,}\\
\emptyset,\mbox{ otherwise.}\end{cases}$$
\end{lemma}
\begin{proof}
If $v'\not\equiv v$ mod $V_{i_k}$, there exists a minimal $\ell>k$ such that $v'\equiv v\mod V_{i_\ell}$. Write $v=\sum_{i=1}^nc_ib_i$. Recall that $v^{>i_\ell}:=\sum_{i=i_{\ell}+1}^n c_ib_i.$ On the $(v^{>\ell},\ell)$-components we have $(E_{vk})_{v^{>\ell},\ell}=(Q^{i_\ell}_{v^{\leq i_\ell}}(\Ut):1)$ and $(E_{v'k'})_{v^{>\ell},\ell}=(Q^{i_\ell}_{v'^{\leq i_\ell}}(\Ut):1)$. However, the choice of $\ell$ and Lemma \ref{kpoints} together imply that 
$$Q^{i_\ell}_{v^{\leq i_\ell}}(\Ut)=\sum_{i=i_{\ell-1}+1}^{i_\ell} c_i\Big(\prod_{j=i}^{i_\ell-1}t_j\Big)\neq \sum_{i=i_{\ell-1}+1}^{i_\ell} c'_i\Big(\prod_{j=i}^{i_\ell-1}t_j\Big)=Q^{i_\ell}_{v'^{\leq i_\ell}}(\Ut),$$ 
Thus $(E_{vk})_{v^{>\ell},\ell}\neq (E_{v'k'})_{v^{>\ell},\ell}$, whence $E_{vk}\cap E_{v'k'}=\emptyset.$

\vspace{2mm}
Suppose  $k'\neq k-1$ and $v'\equiv v$ mod $V_{i_k}$. Let $v'':=(v')^{>i_{k'+1}}$. Then by definition $(E_{vk})_{v'',k'+1}=(1:0)$ and $(E_{v'k'})_{v'',k'+1}=(Q_{v'}^{i_{k'+1}}(\Ut):1)$. Again we conclude that $E_{vk}\cap E_{v'k'}=\emptyset.$

\vspace{2mm}
If $k'=k-1$ and $v'\equiv v$ mod $V_{i_k}$, it follows from Lemma \ref{kpoints} that $Q^{i_{\ell}}_{v^{\leq i_\ell}}(\Ut)=Q^{i_\ell}_{v'^{\leq i_\ell}}(\Ut)$ for all $\ell>k$. We then see directly from (\ref{comps}) that $E_{vk}\cap E_{v',k-1}$ is a single closed point $x=(x_{w\ell}:y_{w\ell})_{w,\ell}\in C_s(k(s))$ with
$$(x_{w\ell}:y_{w\ell})=\begin{cases}(Q^{i_\ell}_{v^{\leq i_\ell}}(\Ut):1),\mbox{ if $\ell\geq k$ and $w\equiv v$ mod $V_{i_\ell}$,}\\
(1:0),\mbox{ otherwise.}\end{cases}$$ 
The fact that $x$ is nodal follows from Corollary \ref{geomred}. 
\end{proof}
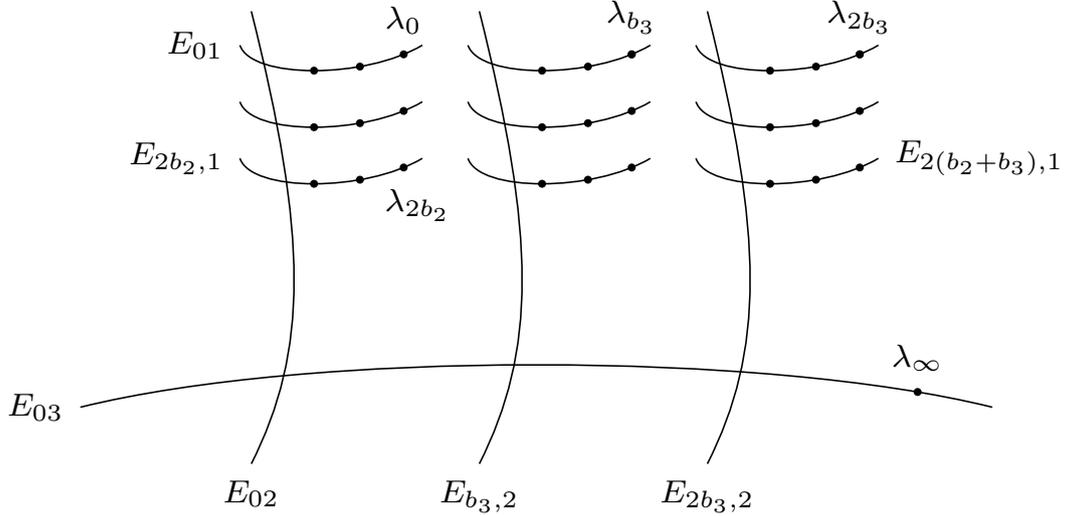
\begin{figure}[H]
\centering
  \caption{Example for $s\in\Omega_\OF$, $V=\Span\{b_1,b_2,b_3\}$, and $q=3$.}
\begin{tikzpicture}[scale=2]
  \draw (0,0)  .. controls (1,.25) and (3,.25) .. node[at start,left] {\scriptsize$E_{03}$} (4,0) 
  node[circle,pos=0.9, fill=black,inner sep=0pt,minimum size=2pt, label=above:{\scriptsize$\lambda_\infty$}]{};

  \draw (.75,-0.25) .. controls (1,.25) and (1,0.75) .. (.75,1.75)
  node[at start,below] {\scriptsize$E_{02}$} ;
    \draw (1.75,-0.25) .. controls (2,.25) and (2,0.75) .. (1.75,1.75)
     node[at start,below] {\scriptsize$E_{b_3,2}$};
     \draw (2.75,-0.25) .. controls (3,.25) and (3,0.75) .. (2.75,1.75)
      node[at start,below] {\scriptsize$E_{2b_3,2}$};

     \draw (.7,1.6) .. controls (.75,1.45) and (1.25,1.45) .. (1.5,1.6)
	node[at start,left] {\scriptsize$E_{01}$}     	
     	 node[circle,pos=0.9, fill=black,inner sep=0pt,minimum size=2pt, label=above:{\scriptsize$\lambda_0$}]{}
       	node[circle,pos=0.7, fill=black,inner sep=0pt,minimum size=2pt,]{}
       node[circle,pos=0.5, fill=black,inner sep=0pt,minimum size=2pt, ]{};        
    \draw (.7,1.35) .. controls (.75,1.20) and (1.25,1.20) .. (1.5,1.35)
       		node[circle,pos=0.9, fill=black,inner sep=0pt,minimum size=2pt,]{}
       		node[circle,pos=0.7, fill=black,inner sep=0pt,minimum size=2pt,]{}
       		node[circle,pos=0.5, fill=black,inner sep=0pt,minimum size=2pt,]{};
     \draw (.7,1.1) .. controls (.75,.95) and (1.25,.95) .. (1.5,1.1)
     node[at start,left] {\scriptsize$E_{2b_2,1}$}
     		node[circle,pos=0.9, fill=black,inner sep=0pt,minimum size=2pt, label=below:{\scriptsize$\,\,\,\,\,\lambda_{2b_2}$}]{}
       		node[circle,pos=0.7, fill=black,inner sep=0pt,minimum size=2pt,]{}
       		node[circle,pos=0.5, fill=black,inner sep=0pt,minimum size=2pt,]{};

         \draw (1.7,1.6) .. controls (1.75,1.45) and (2.25,1.45) .. (2.5,1.6)
         		node[circle,pos=0.9, fill=black,inner sep=0pt,minimum size=2pt,]{}
       		node[circle,pos=0.7, fill=black,inner sep=0pt,minimum size=2pt,]{}
       		node[circle,pos=0.5, fill=black,inner sep=0pt,minimum size=2pt,]{}
       		node[circle,pos=0.9, fill=black,inner sep=0pt,minimum size=2pt, label=above:{\scriptsize$\lambda_{b_3}$}]{};
        \draw (1.7,1.35) .. controls (1.75,1.20) and (2.25,1.20) .. (2.5,1.35)
        		node[circle,pos=0.9, fill=black,inner sep=0pt,minimum size=2pt,]{}
       		node[circle,pos=0.7, fill=black,inner sep=0pt,minimum size=2pt,]{}
       		node[circle,pos=0.5, fill=black,inner sep=0pt,minimum size=2pt,]{};
        \draw (1.7,1.1) .. controls (1.75,.95) and (2.25,.95) .. (2.5,1.1)
        		node[circle,pos=0.9, fill=black,inner sep=0pt,minimum size=2pt,]{}
       		node[circle,pos=0.7, fill=black,inner sep=0pt,minimum size=2pt,]{}
       		node[circle,pos=0.5, fill=black,inner sep=0pt,minimum size=2pt,]{};
        
         \draw (2.7,1.6) .. controls (2.75,1.45) and (3.25,1.45) .. (3.5,1.6)
         		node[circle,pos=0.9, fill=black,inner sep=0pt,minimum size=2pt,]{}
       		node[circle,pos=0.7, fill=black,inner sep=0pt,minimum size=2pt,]{}
       		node[circle,pos=0.5, fill=black,inner sep=0pt,minimum size=2pt,]{}
       		node[circle,pos=0.9, fill=black,inner sep=0pt,minimum size=2pt, label=above:{\scriptsize$\lambda_{2b_3}$}]{};
        \draw (2.7,1.35) .. controls (2.75,1.20) and (3.25,1.20) .. (3.5,1.35)
        		node[circle,pos=0.9, fill=black,inner sep=0pt,minimum size=2pt,]{}
       		node[circle,pos=0.7, fill=black,inner sep=0pt,minimum size=2pt,]{}
       		node[circle,pos=0.5, fill=black,inner sep=0pt,minimum size=2pt,]{};
        \draw (2.7,1.1) .. controls (2.75,.95) and (3.25,.95) .. (3.5,1.1)
        		node[circle,pos=0.9, fill=black,inner sep=0pt,minimum size=2pt,]{}
       		node[circle,pos=0.7, fill=black,inner sep=0pt,minimum size=2pt,]{}
       		node[circle,pos=0.5, fill=black,inner sep=0pt,minimum size=2pt,]{}
       			node[at end,right] {\scriptsize$E_{2(b_2+b_3),1}$}  ;
	
\end{tikzpicture}
  \label{fig:nbhds}
  \end{figure}

 \begin{proposition}\label{fibs} The geometric fibers of $\CC_\OF\to U_\OF$ are connected, reduced curves of genus 0 with at worst nodal singularities.
 \end{proposition}
 \begin{proof}
Combine Corollary \ref{geomred} and Lemmas \ref{irredcomps} and \ref{int}.
\end{proof}

\subsection{$V$-fern structure on $\CC_\OF$}
Consider the left $G$-action $\phi_\OF$ on $\CC_\OF$ described in \ref{SubSectAction} and the $\hat V$-marking $\lambda_\OF$ from Proposition \ref{Vsects}.
\begin{proposition}\label{LocFern} The triple $(\CC_\OF, \lambda_\OF,\phi_\OF)$ is a $V$-fern over $U_\OF$.
\end{proposition}
\begin{proof}
The morphism $\pi_\OF\colon \CC_\OF\to U_\OF$ is projective and hence proper. In light of Propositions \ref{Action}, \ref{Vsects}, \ref{fibs} and Corollary  \ref{CFProps}, it only remains to verify that $(\CC_\OF,\lambda_\OF)$ is stable and that Condition (2) of Definition \ref{FernDef} holds.
\vspace{2mm}

Let $s\in U_\OF$. For simplicity, we assume that $s\in\Omega_\OF$. The general case is handled similarly. Write we first write $C$ as the union of its irreducible components: 
$$C_s=\bigcup_{(w,\ell)\in \Sigma_\OF}E_{w\ell}.$$
By Lemma \ref{int}, for each $(v,k)\in\Sigma_{\OF}$ such that $k>1$, the rational component $E_{vk}\subset C_{\os}$ intersects each $E_{v'k-1}$ such that $v'\equiv v\mod V_k$. There are $q$ such $v'$. For $k<n$, the same lemma shows that $E_{vk}$ also intersects $E_{v^{>k+1}k+1}$. It follows that $E_{vk}\subset C_\os$ contains $q+1$ singular points for all $(v,k)\in\Sigma_\OF$ such that $1<k< n$. For $E_{0n}$, we have the aforementioned $q$ singular points and the $\infty$-marked point. Since $q\geq 2$, there are therefore at least three special points on $E_{vk}$. 

Fix a pair $(v,1)\in \Sigma_\OF$. One then checks directly from the definitions of $E_{v1}$ and $\lambda_\OF$ that the sections $\lambda_\OF(v+\xi b_1)$ with $\xi\in\BF_q$ determine $q$ distinct smooth marked points on $E_{v1}$. Taken with the singular point where $E_{v,1}$ intersects $E_{v^{>2}2}$ or the $\infty$-marked point when $n=1$, we conclude that there are at least $3$ special points on $E_{v1}$. It follows that $(\CC_{\OF},\lambda_\OF)$ is stable.

\vspace{2mm}
To check condition (2) of Definition \ref{FernDef}, we first note that the chain from $0$ to $\infty$ in $C_s$ is given by $\{E_{0k},\ldots,E_{0n}\}$. For each $1\leq k\leq n$, the projection $P^{\Sigma_\OF}_s\to \BP^1_{k(s)}$ onto the $(0,k)$-component induces an isomorphism $\rho_{k}\colon E_{0k}\stackrel\sim\to\BP^1_{k(s)}.$ Let $\xi\in\BF_q^\times.$ The $(0,k)$-component of $\phi_s(\xi)(E_{0k})$ is equal to $(x_{0,\xi^{-1}b_k}:y_{0,\xi^{-1}b_k})$. By (\ref{defeqs}), we have
$$x_{0k}y_{0,\xi^{-1}b_k}=Q^k_{\xi^{-1}b_k}(\Ut)x_{0,\xi^{-1}b_k}y_{0k}.$$
Since $Q^k_{\xi^{-1}b_k}(\Ut)=\xi^{-1}$, this yields
$$\frac{x_{0,\xi^{-1}b_k}}{y_{0,\xi^{-1}b_k}}=\xi\Bigl(\frac{x_{0k}}{y_{0k}}\Bigr).$$
It follows that $\BF_q^\times$ acts on the affine chart $D(Y_{0k})\subset\BP^1_{k(s)}$ via scalar multiplication under the action induced by $\rho_k$. The isomorphism $\rho_k$ sends the distinguished special points of $E_{0k}$ to the correct targets, and thus satisfies Definition \ref{FernDef}(2). We conclude that $(\CC_\OF,\lambda_\OF,\phi_\OF)$ is a $V$-fern, as desired.
\end{proof}

\begin{proposition}\label{fibflag}Let $s\in\Omega_\CF\subset U_\OF$. The flag associated to the fiber $C_s$ is equal to $\CF$. In particular, the $V$-fern $\CC_\OF$ is an $\OF$-fern over $U_\OF$.
\end{proposition}
\begin{proof}Write $\CF=\{V_{i_0},\ldots, V_{i_m}\}.$ Recall that the flag associated to a $V$-fern over a field was defined to be the set of stabilizers of the irreducible components in the chain connecting the $0$- and $\infty$-marked points. Writing $C_s=\bigcup_{(w,\ell)\in\Sigma_\CF} E_{w\ell}$, it follows from the construction of the $\hat V$-marking in Proposition \ref{Vsects} that this chain consists of $\{E_{01},\ldots,E_{0m}\}.$ We must therefore show that $\Stab_{V}(E_{0\ell})=V_{i_\ell}$ for each $1\leq \ell\leq m$.

Fix $1\leq \ell\leq m$. Then by (\ref{comps}), the irreducible component $E_{0\ell}$ is the closed subscheme of $C_s\subset P^{\Sigma_\CF}_s$ defined by
\begin{eqnarray}\label{0lcomp}
X_{0k}&=&0,\mbox{\,\,\,\,\,for all $k>\ell$},\\
\nonumber Y_{vk}&=&0,\mbox{\,\,\,\,\,for all $(v,k)\in\Sigma_\CF$ such that $v\neq 0$ or $k<\ell$}.
\end{eqnarray}
In order to simplify the description of the $V$-action, we consider $C_s$ as a closed subscheme of $P_s^\Sigma$. An element $u\in V$ then acts on $C_s$ via the automorphism
$$(x_{vw}:y_{vw})_{(v,w)\in\Sigma}\mapsto(x_{(v-u)w}:y_{(v-u)w})_{(v,w)\in\Sigma}.$$ For each $u\in V$ and $(v,i_k)\in\Sigma_\CF$, let 
$$(X^u_{vk}:Y^u_{vk}):=(X_{v-u,b_{i_k}}:Y_{v-u,b_{i_k}}).$$  Then $u\in V$ stabilizes $E_{0\ell}$ if and only if (\ref{0lcomp}) implies that
\begin{eqnarray}\label{u0lcomp}
X^u_{0k}&=&0,\mbox{\,\,\,\,\,for all $k>\ell$},\\
\nonumber Y^u_{vk}&=&0,\mbox{\,\,\,\,\,for all $(v,k)\in\Sigma_\CF$ such that $v\neq 0$ or $k<\ell$}.
\end{eqnarray}

Suppose $u\in V_{i_\ell}$ and that (\ref{0lcomp}) holds. The point $s\in\Omega_{\CF}$ corresponds to a tuple $\Ut:=(t_1,\ldots t_{n-1})\in k(s)^{n-1}$. For $k>\ell$ the equation
$$X_{0k}Y_{0k}^u-Q_u^{i_k}(\Ut)Y_{0k}Y^u_{0k}=X^u_{0k}Y_{0k}$$
holds on $C_s$ by (\ref{eqsFlag}). It follows from Lemma \ref{kpoints} that $Q^{i_k}_u(\Ut)=0$; hence
$$X_{0k}Y_{0k}^u=X^u_{0k}Y_{0k}.$$
Combining this with (\ref{0lcomp}) implies that $X^u_{0k}=0$. Fix $(v,k)\in\Sigma_\CF$ with $v\neq 0$. By (\ref{eqsFlag}) we have
\begin{equation}\label{c}X_{vk}^{u^{>i_k}}Y_{vk}^u-Q^{i_k}_{u^{\leq i_k}}(\Ut)Y_{vk}^{u^{>i_k}}Y_{vk}^{u}=X_{vk}^{u}Y_{vk}^{u^{>i_k}}\end{equation}
on $C_s$. Since (\ref{0lcomp}) implies that $Y_{vk}^{u^{>i_k}}=0$, equation (\ref{c}) yields $Y^u_{vk}=0$.
Finally, consider $(0,k)\in \Sigma_\CF$ with $k<\ell$. By similar reasoning, we find that $Y_{0k}^u=0$.
All together, we have shown that (\ref{u0lcomp}) holds; hence $u\in\Stab_V(E_{0\ell})$. It follows that $V_{i_\ell}\subset \Stab_V(E_{0\ell})$.
\vspace{2mm}

For the reverse inclusion, suppose $u\in V\setminus V_{i_\ell}$. 
Choose $k>\ell$ minimal such that $u\in V_{i_k}$. On $C_s$ we again have:
$$X_{0k}Y_{0k}^u-Q_u^{i_k}(\Ut)Y_{0k}Y^u_{0k}=X^u_{0k}Y_{0k}.$$
On $E_{0\ell}$, we have $X_{0k}=0$ and $Y_{0k}\neq 0$. If $X^u_{0k}=0$, then we deduce that $Q_u^{i_k}(\Ut)Y_{0k}Y^u_{0k}=0$. Since $Q_u^{i_k}(\Ut)\neq 0$ by Lemma \ref{kpoints}, it follows that $Y^u_{0k}=0$, a contradiction. Hence (\ref{0lcomp}) does not imply (\ref{u0lcomp}), so $u\not\in\Stab_V(E_{0\ell})$. We therefore have $V_{i_\ell}=\Stab(E_{0\ell})$, as desired.
\end{proof}


\subsection{Contraction of $\CC_\OF$ with respect to $\hat V_{n-1}$}
Let $V':=V_{n-1}$ and $\OF':=\{V_0,\ldots V_{n-1}\}$. The natural projection $\prod_{0\neq W\subset V}P_W\to \prod_{0\neq W\subset V'}P_W$ induces a morphism 
$$p\colon U_{\OF}\to U_{\OF'}.$$ 
Recall that we fixed a flag basis $\CB=\{b_1,\ldots,b_n\}$ associated to $\OF$ in order to describe the coordinates on $U_{\OF}$. The basis $\CB$ restricts to a flag basis of $V'$ associated to $\OF'$. We show that pullback of the universal family $\CC_{\OF'}\to U_{\OF'}$ along $p$ is precisely the contraction of $\CC_\OF$ with respect to $\hat V'$. This will be key for the induction step in our proof of the main theorem. For clarity, we recall/fix the following notation:
$$
\begin{array}{rl}
\CC_\OF& \mbox{ the $\OF$-fern over $U_\OF$ constructed above;}\\
\CC_{\OF'}& \mbox{ the analogously constructed $\OF'$-fern over $U_{\OF'}$;}\\
(\CC_{\OF})'&\mbox{ the $\OF'$-fern over $U_\OF$ obtained by contracting $\CC_\OF$ with respect to $\hat V'$;}\\
p^*\CC_{\OF'}&\mbox{ the pullback $\CC_{\OF'}\times_{U_{\OF'}}U_\OF$ along the morphism $p$.}
\end{array}
$$
Recall that for a subflag $\CF:=\{V_{i_0},\ldots, V_{i_m}\}\subset\OF$, we defined
$$\Sigma_\CF:=\{(v,k)\in V\times \{1,\ldots, m\}\mid v^{\leq i_k}=0\}.$$
 We view $\Sigma_\CF$ as a subset $\Sigma_\OF$ via $(v,k)\mapsto (v,i_k)$ and view $\Sigma_{\OF'}$ as a subset of $\Sigma_\OF$ in a similar manner.
\begin{lemma}\label{ContraMap}The natural projection $P^{\Sigma_\OF}\to P^{\Sigma_{\OF'}}$ induces a morphism of $\hat V'$-marked curves $$\gamma\colon \CC_\OF\to p^*\CC_{\OF'}.$$
\end{lemma}
\begin{proof} Regard that $\CC_\OF$ is a closed subscheme of $U_\OF\times P^{\Sigma_\OF}$. Since the equations defining $p^*\CC_{\OF'}$ in $U_{\OF}\times P^{\Sigma_{\OF'}}$ pull back to a subset of the equations  (\ref{eqsFlag}) defining $\CC_\OF$, the composite $\CC_\OF\into U_\OF\times P^{\Sigma_\OF}\to U_\OF\times P^{\Sigma_{\OF'}}$ indeed factors through a morphism $\gamma\colon \CC_\OF\to p^*\CC_{\OF'}$. We must verify that $\gamma$ preserves the $\hat V'$-markings on $\CC_\OF$ and $p^*\CC_{\OF'}$. Fix $v\in V'$, and let $\lambda_{v}$ and $\lambda_v'$ denote the corresponding marked sections of $\CC_\OF$ and $p^*\CC_{\OF'}$ respectively. We wish to show that $\gamma\circ\lambda_v=\lambda_v'$. Since $p^*\CC_{\OF'}$ is separated and $U_\OF$ is reduced and $\OV\subset U_\OF$ is dense, it suffices to show that $\gamma\circ\lambda_v|_\OV=\lambda_v'|_\OV$ (\cite{GW}, Corollary 9.9). Let $\Sigma':=V'\times (V'\setminus\{0\})$, and fix $v_0\in V'\setminus\{0\}$. Define the morphisms 
\begin{eqnarray*}
f_V\colon \OV\times \BA^1\to B_V\times P^\Sigma;\\
f_{V'}\colon \Omega_{V'}\times\BA^1 \to B_V\times P^{\Sigma'}
\end{eqnarray*}
as in (\ref{theMap}). The natural inclusion $RS_{V',0}[T]\into RS_{V,0}[T]$ corresponds to a morphism 
$$\pi\colon\OV\times\BA^1\to \Omega_{V'}\times\BA^1.$$
Recall that $\CC_\OF$ is the scheme theoretic closure of $f_V$ in $U_\OF\times P^\Sigma$ and similarly for $\CC_{\OF'}$. Moreover, the product $\Pi$ of the morphism $B_V\to \CB_{V'}$ induced by the forgetful functor $(\CE_W)_{0\neq W\subset V}\mapsto (\CE_W)_{0\neq W\subset V'}$ and the natural projection $P^\Sigma\to P^{\Sigma'}$ yields the following commutative diagram:
$$
\xymatrix{\OV\times\BA^1\ar@{=}[r]\ar@/^1pc/[rr]^\pi\ar[d]\ar@/_1pc/[dd]_{f_V} &\OV\times\BA^1\ar[d]\ar[r]&\Omega_{V'}\times\BA^1\ar[d]\ar@/^1pc/[dd]^{f_{V'}}
\\\CC_\OF\ar[d]\ar[r]^\gamma & p^*\CC_{\OF'}\ar[r]\ar[d] & \CC_{\OF'}\ar[d] \\
B_V\times P^\Sigma\ar[r]\ar@/_1pc/[rr]_{\Pi} &B_V\times P^{\Sigma'}\ar[r]& \CB_{V'}\times P^{\Sigma'}.}
$$
The $v$-marked sections of $\OV\times\BA^1$ and $\Omega_{V'}\times\BA^1$ correspond to $\frac{v}{v_0}\in RS_{V',0}\subset RS_{V,0}$ and are hence preserved by $\pi$. By construction, the open immersions $\OV\times\BA^1\into \CC_{\OF}$ and $\Omega_{V'}\times\BA^1\into \CC_{\OF'}$ also preserve the $v$-marked sections. We thus have $\gamma\circ\lambda_v|_\OV=\lambda_v'|_\OV$. Since $v$ was arbitrary, we deduce that $\gamma$ preserves the $V'$-marking. We defined the infinity section $\lambda_\infty\colon U_\OF\to \CC_\OF\subset P^\Sigma$ to be the morphism $$\UT\mapsto (\UT,(1:0)_{(v,w)\in\Sigma})$$ and similarly for $\infty$-marked section of $\CC_{\OF'}\to U_{\OF'}$. This is clearly preserved by $\gamma$, which thus preserves the $\hat V'$-marked sections, as desired.
\end{proof}
\begin{proposition}\label{cont'}The contraction $(\CC_\OF)'$ of $\CC_\OF$ is isomorphic to the pullback $p^*\CC_{\OF'}$.
\end{proposition}
\begin{proof} By the uniqueness of contractions, it suffices to show that the morphism $\gamma\colon \CC_\OF\to p^*\CC_{\OF'}$ from Lemma \ref{ContraMap} is a contraction with respect to $\hat V'$. Since $\gamma$ preserves the $\hat V$-marking and $p^*\CC_{\OF'}$ is a stable $\hat V'$-marked curve, it only remains to show that $\gamma$ satisfies the contraction condition on fibers (Definition \ref{cntrct}.2). Consider $s\in U_\OF$ with corresponding fibers $C_s:=(\CC_\OF)_s$ and $C'_s:=(p^*\CC_{\OF'})_s$. Let $s'\in U_{\OF'}$ be the image of $s$ under the projection $p\colon U_\OF\to U_{\OF'}$. Let $\CF\subset\OF$ and $\CF'\subset\OF'$ respectively denote the unique subflags such that $s\in\Omega_\CF$ and $s'\in\Omega_{\CF'}$. Let $K:=k(s)$. In our choice of coordinates, the point $s$ corresponds to a tuple $\Ut:=(t_1,\ldots, t_{n-1})\in K^{n-1}$, while $s'$ corresponds to $(t_1,\ldots ,t_{n-2})\in K^{n-2}$. Let $\{i_0,\ldots, i_m\}$ be the indices for which $t_i=0$. By Lemma \ref{kpoints}, we have $\CF=\{V_{i_0},\ldots,V_{i_m}\}$ and $\CF'=\{V_{i_0},\ldots, V_{i_{m-1}}, V'\}$. Note that $V_{i_{m-1}}$ may be equal to $V'$. The natural inclusions $\Sigma_\CF\subset\Sigma_\OF$ and $\Sigma_{\CF'}\subset\Sigma_{\OF'}\subset\Sigma_\OF$ yield projections
$$\xymatrix{&P^{\Sigma_\OF}_K\ar[dl]\ar[dr]&\\
P^{\Sigma_\CF}_K\ar@{-->}[drr] & & P^{\Sigma_{\OF'}}_K\ar[d]\\
& & P^{\Sigma_{\CF'}}_K,}
$$
where the morphism represented by the dotted arrow occurs only when $V'=V_{i_{m-1}}$ (so that $\Sigma_{\CF'}\subset\Sigma_\CF$). In this case, the above diagram commutes. 
By Lemma \ref{irredcomps}, we have the following decompositions into irreducible components:
\begin{eqnarray*}
C_s=\bigcup_{(w,\ell)\in\Sigma_{\CF}}E_{w\ell},\,\,\,\,C_s'=\bigcup_{(w,\ell)\in\Sigma_{\CF'}}E'_{w\ell}.
\end{eqnarray*} 
Fix $(w,i_\ell)\in \Sigma_\CF\subset\Sigma_\OF.$ We have the following two cases:
\vspace{5mm}

\noindent\emph{Case 1. $\underline{\bigl((w,i_\ell)\in \Sigma_\CF\cap\Sigma_{\CF'}\bigr)}$}: The diagram
$$\xymatrix{C_s\ar@{^`->}[r]\ar[d]_{c_s}&P^{\Sigma_{\OF}}_K\ar[d]\ar[r]^{\pi_{wi_\ell}}&\BP^1_K
\\C'_s\ar@{^`->}[r]&P^{\Sigma_{\OF'}}_K\ar[ur]_{\pi'_{wi_\ell}}&,}$$
where $\pi_{wi_\ell}$ and $\pi'_{wi_\ell}$ denote the projections to the $(w,i_\ell)$-component, commutes. By Lemma \ref{compP1}, these projections induce isomorphisms $E_{w\ell}\stackrel\sim\to\BP^1_K$ and $E_{w\ell}'\stackrel\sim\to\BP^1_K$. It follows that $\gamma_s$ restricts to an isomorphism $E_{w\ell}\stackrel\sim\to E'_{w\ell}$.
\vspace{5mm}

\noindent\emph{Case 2. $\underline{\bigl((w,i_\ell)\not\in \Sigma_\CF\cap\Sigma_{\CF'}\bigr)}$}:
We have the following subcases:
\vspace{5mm}

\emph{Case 2a. $\underline{\bigl(V_{i_m-1}= V'\bigr)}$}: Consider the closed point 
$$p'_\infty:=\bigl((1:0)_{vk}\bigr)_{(v,i_k)\in\Sigma_\CF'}\in C'_s\subset P^{\Sigma_{\CF'}}.$$
We claim that $\gamma_s(E_{w\ell})=p'_\infty$. This is equivalent to showing that the equation $Y_{vk}=0$ holds on $E_{w\ell}$ for all $(v,i_k)\in\Sigma_{\CF'}$ Recall from (\ref{comps}) that $E_{w\ell}\subset C_s\subset P^{\Sigma_\CF}$ is the closed subscheme defined by 
\begin{eqnarray*}
X_{vk}-Q^{i_{k}}_{w}(\Ut)Y_{vk}&=&0,\mbox{\,\,\,\,\,for all $(v,k)\in\Sigma_\CF$ such that $k>\ell$ and $v\equiv w$ mod $V_{i_{k}}$},\\
\nonumber Y_{vk}&=&0,\mbox{\,\,\,\,\,for all other $(v,k)\in\Sigma_\CF\setminus \{(w,i_\ell)\}$}.
\end{eqnarray*}
Fix $(v,i_k)\in\Sigma_{\CF'}\subset\Sigma_{\CF}$. The assumption that $(w,i_\ell)\not\in \Sigma_\CF\cap\Sigma_{\CF'}$ implies that the conditions $k>\ell$ and $v\equiv w \mod V_{i_k}$ do not hold simultaneously. Hence $Y_{vk}=0$ on $E_{w\ell}$, as desired.
\vspace{5mm}

\emph{Case 2b. $\underline{\bigl(V_{i_{m-1}}\subsetneq V'\bigr)}$}: Suppose first that $(w,i_\ell)\neq (0,n)$. By Lemma \ref{kpoints}, the assumption that $V_{i_{m-1}}\subsetneq V'$ implies that $t_{n-1}\neq 0$. We may thus consider the closed point
$$p'_{w}:=\bigl((x_{vk}:y_{vk})\bigr)_{(v,k)\in\Sigma_{\CF'}}\in P^{\Sigma_{\CF'}}$$
with
$$(x_{vk}:y_{vk})=\begin{cases}(1:0),\mbox{ if $(v,k)\neq (0,n-1),$}
 \\ \bigl((t_{n-1})^{-1}Q^n_w(\Ut):1\bigr),\mbox{ if $(v,k)=(0,n-1).$}
\end{cases}$$
It follows directly from the defining equations (\ref{comps}) of $E'_{0n-1}$ that $p'_w\in E'_{0,n-1}$. We claim that $\gamma_s(E_{w\ell})=p'_w.$ This is equivalent to showing that the equations 
\begin{eqnarray}\label{a}
Y_{vk}&=&0,\mbox{ for all $(v,k)\in\Sigma_{\CF'}\setminus\{(0,n-1)\}$};\\
\label{b}\frac{X_{0n-1}}{Y_{0n-1}}&=&(t_{n-1})^{-1}Q_w^n(\Ut)
\end{eqnarray}
hold on $E_{w\ell}$. That (\ref{a}) holds follows directly from (\ref{comps}) and the assumption that $(w,i_\ell)\not\in \Sigma_\CF\cap\Sigma_{\CF'}.$ Observe that on $C_s$ we have
\begin{equation}\label{Ewl I}X_{0n}Y_{0n-1}=Q^n_{b_{n-1}}(\Ut)X_{0n-1}Y_{0n}=t_{n-1}X_{0n-1}Y_{0n}\end{equation}
by (\ref{eqsFlag}), and on $E_{w\ell}$ we have
\begin{equation}\label{Ewl II}X_{0n}-Q^n_{w}(\Ut)Y_{0n}=0.\end{equation}
Together (\ref{Ewl I}) and (\ref{Ewl II}) imply (\ref{b}). Thus $\gamma_s(E_{w\ell})=p'_\ell$, as desired.
\vspace{3mm}

Finally, suppose $(w,i_\ell)=(0,n)$. We claim that $\gamma_s$ induces an isomorphism $E_{0n}\stackrel\sim\to E'_{0n-1}$. We have a commutative diagram
$$\xymatrix{C_s\ar@{^`->}[r]\ar[d]_{c_s}&P^{\Sigma_{\OF}}_K\ar[d]\ar[r]^{\pi_{0n-1}}&\BP^1_K
\\C'_s\ar@{^`->}[r]&P^{\Sigma_{\OF'}}_K\ar[ur]_{\pi'_{0n-1}}&}.$$
Proposition \ref{RelComps} implies that the $(0,n)$-component of $C_s$ is determined by the $(0,n-1)$-component and vice versa. Since $\pi_{0n}$, the projection onto the $(0,n)$-component, induces an isomorphism $E_{0n}\stackrel\sim\to\BP^1_K$, it follows that the same is true for $\pi_{0n-1}$. We may thus apply the same reasoning as in case 1 to conclude that $\gamma_s$ induces an isomorphism $E_{0n}\stackrel\sim\to E'_{0n-1}$, as claimed.
\vspace{5mm}

All together, we have shown that each irreducible component of $C_s$ is either contracted to a point under $\gamma_s$, or mapped isomorphically onto an irreducible component of $C'_s$. The above argument also shows that $\gamma_s$ is surjective. This proves that $\gamma$ is a contraction.
\end{proof}

\subsection{$V$-fern structure on $\CC_V$}
Let $\CF$ be an arbitrary flag of $V$. A choice of flag basis corresponding to $\CF$ yields an open immersion $U_\CF\into \BA^{n-1}_{\BF_q}$, and we identify $U_\CF$ with its image in $\BA^{n-1}_{\BF_q}$. Let $\CC_\CF$ be defined as in the previous subsection.
\begin{proposition}\label{localCV}We have $\CC_V\cap (U_\CF\times P^\Sigma)=\CC_\CF$ as closed subschemes of $U_\CF\times P^\Sigma$. 
\end{proposition}
\begin{proof}The scheme $\CC_V$ was defined to be the scheme-theoretic closure of $\Image f_V\subset B_V\times P^\Sigma$. Since $\OV\times \BA^1$ is reduced, the scheme $\CC_V$ is just the topological closure of the image of $f_V$ with the induced reduced scheme structure (\cite{Stacks}, Tag 056B). Since $f_V$ is quasicompact, taking the scheme-theoretic closure of the image commutes with restriction to open subsets of $B_V\times P^\Sigma$ (\cite{Stacks}, Tag 01R8). All together, this means that $\CC_\CF^\sharp:=\CC_V\cap (U_\CF\times P^\Sigma)$ is the topological closure of $\Image f_V$ in $U_\CF\times P^\Sigma$ endowed with the induced reduced subscheme structure. Since $\CC_\CF$ is reduced by Corollary \ref{CFProps}, it thus suffices to show that $C^\sharp_\CF=\CC_\CF$ as topological spaces. By Lemma \ref{cont}, we have $\Image f_V\subset \CC_\CF$, and hence $\CC_\CF^\sharp=\overline{\Image f_V}\subset \CC_\CF.$ By Lemma \ref{immersion}, we have $\Image f_V\cong \OV\times\BA^1$, which is $n$-dimensional and irreducible; hence so is $\CC_\CF^\sharp$. The same holds for $\CC_\CF$ by Corollary \ref{CFProps}. The inclusion $\CC_\CF^\sharp\subset \CC_\CF$ must therefore be an equality, proving the lemma. 
\end{proof}
Recall that $B_V\subset\bigcup_{\OF}U_{\OF}$, where $\OF$ runs over the set of complete flags of $V$ (Proposition \ref{covering}). For each such $\OF$, consider the $V$-fern $(\CC_\OF,\lambda_\OF,\phi_\OF)$. We have $\CC_V=\bigcup_\OF \CC_\OF$ by Proposition \ref{localCV}. The uniqueness of the $\lambda_\OF$ and $\phi_\OF$ implies that they glue respectively to a $\hat V$-marking $\lambda_V$ of $\CC_V\to B_V$ and a left $G$-action $\phi_V\colon G\to\Aut_{B_V}(\CC_V)$. The following proposition follows directly from Proposition \ref{LocFern}.
\begin{proposition}\label{VfernStr}The tuple $(\CC_V,\lambda_V,\phi_V)$ is a $V$-fern over $B_V$. 
\end{proposition}



\section{The Moduli Space of $V$-ferns}
In this section we prove the following more precise version of the Main Theorem:
\begin{theorem}[Main Theorem]\label{rep} The pair $(\CC_V,B_V)$ represents $\Fern_V$.
\end{theorem}
We prove Theorem \ref{rep} at the end of this section.
\begin{remark}To say that the pair $(\CC_V,B_V)$ represents $\Fern_V$ means that $C_V$ is a $V$-fern over $B_V$ and for each $S\in\SchFq$ and each $V$-fern $C$ over $S$, there exists a unique morphism $f_C\colon S\to B_V$ such that $C\cong f_C^*(\CC_V)$. This is equivalent to giving an explicit isomorphism of functors 
$$\begin{array}{ccc}\Mor_{\BF_q}(-,B_V)&\stackrel\sim\longto& \Fern_V,\\
(f\colon S\to B_V)&\mapsto&[f^*\CC_V],
\end{array}$$
where $[f^*\CC_V]$ denotes the isomorphism class of the $V$-fern $f^*\CC_V$ over $S$.
\end{remark}

\subsection{Reduction to the case of $\OF$-ferns}
Recall that a $V$-fern $C$ over $S$ is called an $\CF$-fern for a flag $\CF$ if $S=S_\CF$, i.e., for all $s\in S$, the flag $\CF_s$ associated to the fiber $C_s$ is contained in $\CF$. Consider the following analog to Theorem \ref{rep}:
\begin{theorem}\label{flagrep}Let $\CF$ be a flag of $V$. The scheme $U_\CF$ together with the $V$-fern $C_\CF$ represent $\Fern_\CF$.
\end{theorem}
\begin{corollary}\label{7-cor-smoothferns}
The scheme $\OV$ represents the functor that associates to an $\BF_q$-scheme $S$ the set of isomorphism classes of smooth $V$-ferns over $S$.
\end{corollary}
\begin{proposition}\label{redtoflag}Theorem \ref{rep} follows from Theorem \ref{flagrep}. 
\end{proposition}
\begin{proof}Suppose Theorem \ref{flagrep} holds. By Propositions \ref{covering} and \ref{localCV}, the family of morphisms $$(C_\CF\to U_\CF)_\CF,$$ where $\CF$ runs over all flags of $V$, glue to the structure morphism $\CC_V\to B_V$. Let $S$ be a scheme and let $\pi\colon C\to S$ be a $V$-fern. By Proposition \ref{flagcov}, we have $S=\bigcup_{\CF}S_\CF.$ By assumption, we obtain a family of unique morphisms $$(f_\CF\colon S_\CF\to U_\CF)_\CF$$ inducing unique isomorphisms $$(F_\CF\colon \pi^{-1}(S_\CF)\stackrel\sim\to C_{\CF}\times_{U_\CF}S_{\CF})_\CF.$$ The uniqueness implies that for any two flags $\CF$ and $\CG$, we have $$f_\CF|_{S_{\CF\cap\CG}}=f_\CG|_{S_{\CF\cap\CG}}=f_{\CF\cap\CG}$$ and similarly for $F_\CF$ and $F_\CG$. By gluing, we thus obtain a unique morphism $f\colon S\to B_V$ inducing a unique isomorphism $F\colon C\stackrel\sim\to \CC_V\times_{B_V} S.$ This proves the proposition.
\end{proof}

\begin{proposition}\label{completetogeneral}Let $\CF$ be a flag, and let $\OF$ be a complete flag containing it. If Theorem \ref{flagrep} holds for $\OF$, then it holds for $\CF$.
\end{proposition}
\begin{proof}Suppose Theorem \ref{flagrep} holds for $\OF$ and let $C$ be an $\CF$-fern over $S$. Then $C$ is also an $\OF$-fern, and it suffices to show that the unique morphism $f\colon S\to U_\OF$ corresponding to $C$ factors through $U_\CF\subset U_\OF$. Let $s\in S$ and let $p:=f(s)$. Then the isomorphism $C\stackrel\sim\to \CC_\OF\times_{U_\OF} S$ induces an isomorphism $C_s\stackrel\sim\to (\CC_\OF)_p$. This implies that the flag associated to $(\CC_\OF)_p$ is contained in $\CF$. By Proposition \ref{fibflag}, we deduce that $p\in U_\CF$, as desired.
\end{proof}

For the remainder of this section, fix a complete flag $\OF=\{V_0,\ldots,V_n\}.$
\subsection{The representing morphism associated to an $\OF$-fern}\label{morphfern}
Let $S$ be a scheme and let $(C,\lambda,\phi)$ be an $\OF$-fern over $S$. For each $\BF_q$-subspace $0\neq V'\subset V$, let $(C',\lambda',\phi')$ denote the $V'$-fern obtained from contracting $C$ (Proposition \ref{FernContr}). By Proposition \ref{lbstr} and Corollary \ref{fibnonzero}, there is a natural line bundle $L'$ associated to $C'$ endowed with a fiberwise nonzero $\BF_q$-linear map $\lambda'\colon V'\to L'(S)$. The pair $(L',\lambda')$ thus determines a morphism of schemes $S\to P_{V'}$. By varying $V',$ we obtain a morphism 
\begin{equation}\label{assocMorph}f_C\colon S\to \prod_{0\neq V'\subset V}P_{V'}.\end{equation}

\begin{proposition}\label{isoclass} The morphism $f_C$ depends only on the isomorphism class of $(C,\lambda,\phi)$.
\end{proposition}
\begin{proof} Let $(D,\mu,\psi)$ be a $V$-fern isomorphic to $(C,\lambda,\phi)$. Let $0\neq V'\subset V$, and let $D'$ denote the contraction of $D$ with respect to $\hat V'$. The composite $C\stackrel\sim\to D\to D'$ makes $D'$ a contraction of $C$ with respect to $\hat V'$. We thus obtain a unique isomorphism $C'\stackrel\sim\to D'$ of $V'$-ferns. Similarly, by the uniqueness of the contraction to the $\infty$-component (Corollary \ref{1-cor-uniquenessCi}), we obtain an isomorphism of $\hat V'$-marked curves $\alpha\colon (C')^{\infty}\stackrel\sim\to (D')^\infty.$ Recall that the pair $(L',\lambda')$ was obtained from $(C')^\infty$ by taking the complement of the $\infty$-divisor and restricting the $\hat V'$-marking to $V'$. Denoting the corresponding line bundle associated to $D$ by $(M',\mu')$, the isomorphism $\alpha$ this induces an isomorphism of pairs $$(L',\lambda')\stackrel\sim\to(M',\mu').$$
Hence $C$ and $D$ induce the same morphism $S\to P_{V'}.$ The proposition follows.
\end{proof}

\begin{lemma}\label{infcompF}Fix $0\neq V'\subset V$ and let $1\leq i\leq n$ be the smallest integer such that $V'\subset V_i$. Let $v'\in V'\setminus (V_{i-1}\cap V')$. The contraction of $C'$ to the $\infty$-component $(C')^\infty$ is isomorphic to the contraction of $C$ with respect to $\{0, v',\infty\}.$
\end{lemma}
\begin{proof}On each fiber, the marked points of $(C')^\infty$ are in one-to-one correspondence with $V'/(V_{j}\cap V')\cup\{\infty\}$ for some $j\leq i-1$. The choice of $v'$ thus guarantees that the $0\,$- and $v'$- and $\infty$-marked sections are disjoint in $(C')^\infty$. By Proposition \ref{loccontr}, we deduce that $(C')^\infty$ is the contraction of $C'$ (and hence $C$) with respect to $\{0,v',\infty\}$. 
\end{proof}
Consider $0\neq V''\subsetneq V'\subset V$, with corresponding pairs $(L'',\lambda'')$ and $(L',\lambda')$ as in the definition of $f_C$ above. 
\begin{lemma}\label{7-lem-line-bundle-hom}There exists a morphism $\psi\colon L''\to L'$ (as schemes over $S$)  such that $\lambda'|_{V''}=\psi\circ\lambda''.$
\end{lemma}
 \begin{proof} 

%
%
Let $1\leq i\leq j\leq n$ be minimal such that $V''\subset V_i$ and $V'\subset V_{j}$. Let $v''\in V''\setminus (V''\cap V_{i-1})$ and $v'\in V'\setminus (V'\cap V_{j-1})$. Define $\alpha:=\{0,v',v'',\infty\}$ and $\alpha':=\{0,v',\infty\}$ and $\alpha'':=\{0,v'',\infty\}$, with corresponding contractions $C^\alpha$ and $C^{\alpha'}$ and $C^{\alpha''}$. Let $C'$ and $C''$ denote the contractions of $C$ with respect to $\hat V'$ and $\hat V''$ respectively. By Lemma \ref{infcompF}, for the contractions to the $\infty$-component we have $(C'')^\infty\cong C^{\alpha''}$ and $(C')^\infty\cong C^{\alpha'}$. By construction, the line bundle $L'$ is the complement of the $\infty$-divisor in $C^{\alpha'}$ and similarly for $L''$ in $C^{\alpha''}$. Moreover, the contraction morphisms $C\to C^{\alpha'}$ and $C\to C^{\alpha''}$ factor uniquely through $C^\alpha$. We depict this in the following commutative diagram: 
$$\xymatrix{&C\ar[ldd]\ar[rdd]\ar[d]&\\&C^\alpha\ar[ld]\ar[rd]&\\*+[l]{L'\subset C^{\alpha'}}&&*+[r]{C^{\alpha''}\supset L''.}}$$
Let $U''\subset C^\alpha$ denote the inverse image of $L''$. By Proposition \ref{contriso}, the contraction $C^\alpha\to C^{\alpha''}$ induces an isomorphism $U''\stackrel\sim\to L''$. The composite 
$L''\stackrel\sim\to U''\into C^{\alpha}\to C^{\alpha'}$ yields a morphism $\psi\colon L''\to L'$ which preserves $V''$-marked sections.\end{proof}
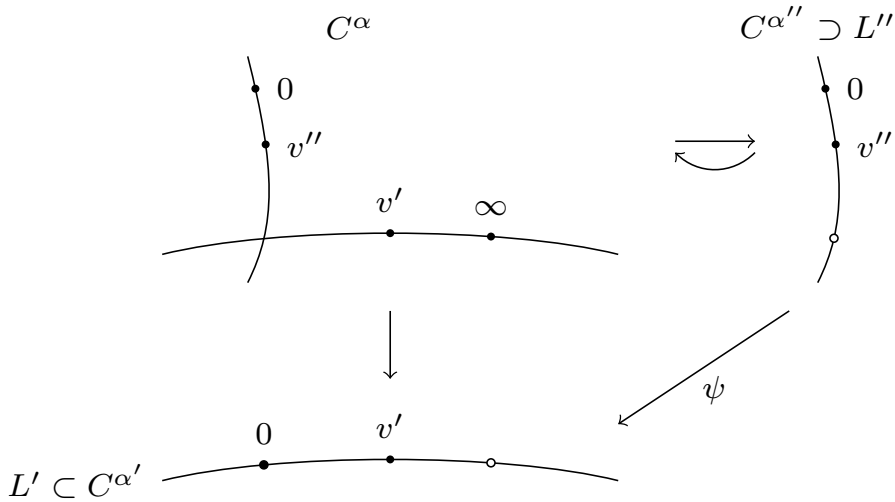
\begin{figure}[H]
\centering
 \caption{The morphism $\psi\colon L''\to L'$ on a fiber where $C^\alpha$ is singular. The unfilled points indicate the $\infty$-marked points of $C^{\alpha'}$ and $C^{\alpha''},$ which are not included in $L'$ and $L''$. In this case, the morphism $\psi$ maps $L''$ to the $0$-marked point of $L'$.}
\begin{tikzpicture}
  \draw[name path=E](0,0)  .. controls (1,.25) and (3,.25) .. 
  (4,0) 
   node[color=black,circle,pos=0.5, fill=black,inner sep=0pt,minimum size=2pt, label={above:{\scriptsize${v'}$}}]{}
     node[color=black,circle,pos=0.7, fill=black,inner sep=0pt,minimum size=2pt, label={above:{\scriptsize$\infty$}}]{};

     \draw[opacity=0](2,2)--(2,3)
     node[opacity=1,at start, left]{\scriptsize$C^\alpha$};

  \draw[name path=D] (.75,-0.25) .. controls (1,.25) and (1,0.75) .. (.75,1.75)
node[color=black,circle,pos=0.9, fill=black,inner sep=0pt,minimum size=2pt, label={right:{\scriptsize$0$}}]{}
node[color=black,circle,pos=0.7, fill=black,inner sep=0pt,minimum size=2pt, label={right:{\scriptsize${v''}$}}]{}
  ;

     
\draw[->,
line join=round,
decorate, decoration={
    segment length=4,
    amplitude=.9,post=lineto,
    post length=2pt}]
 (2,-.5)--(2,-1.10)
 ;
 
 \begin{scope}[shift={(2.5,0)}]
 \draw[->,
line join=round,
decorate, decoration={
    segment length=4,
    amplitude=.9,post=lineto,
    post length=2pt}]
 (3,-.5)--(1.5,-1.50)
 node[pos=0.7, label={right:{\scriptsize$\psi$}}]{}    ;
 \end{scope}

 \draw[->,
line join=round,
decorate, decoration={
    segment length=4,
    amplitude=.9,post=lineto,
    post length=2pt}]
 (4.5,1)--(5.2,1);
 
  \draw[->,
line join=round,
decorate, decoration={
    segment length=4,
    amplitude=.9,post=lineto,
    post length=2pt}]
 (5.2,.9).. controls (5,.7) and (4.7,.7) ..(4.5,.9);
 
     
    \draw[name path=0'', opacity=0] (5,0)  .. controls (1+5,.25) and (3+5,.25) .. 
  (7,0) 
   node[color=black,circle,pos=0.5, fill=black,inner sep=0pt,minimum size=2pt, label={above:{\scriptsize${v'}$}}]{}
     node[color=black,circle,pos=0.7, fill=black,inner sep=0pt,minimum size=2pt, label={above:{\scriptsize$\infty$}}]{};
     
       \draw[name path=L''](.75+5,-.25) .. controls (1+5,.25) and (1+5,0.75) .. (5.75,1.75)
       node[at end,above] {\scriptsize $C^{\alpha''}\supset L''$} 
node[color=black,circle,pos=0.9, fill=black,inner sep=0pt,minimum size=2pt, label={right:{\scriptsize$0$}}]{}
node[color=black,circle,pos=0.7, fill=black,inner sep=0pt,minimum size=2pt, label={right:{\scriptsize${v''}$}}]{}
  ;

    \draw[name path=0'](0,-2)  .. controls (1,.25-2) and (3,.25-2) .. 
  node[at start,left] {\scriptsize$L'\subset C^{\alpha'}$} 
  (4,-2) 
   node[color=black,circle,pos=0.5, fill=black,inner sep=0pt,minimum size=2pt, label={above:{\scriptsize${v'}$}}]{}
     node[color=black,draw,circle,pos=0.7, fill=white,inner sep=0pt,minimum size=2pt, label={above:{\scriptsize$\,$}}]{};

       \draw[name path=L', opacity=0]( (.75,-2.25) .. controls (1,.25-2) and (1,0.75-2) .. (.75,1.75-2)
node[color=black,circle,pos=0.9, fill=black,inner sep=0pt,minimum size=2pt, label={right:{\scriptsize$1$}}]{}
node[color=black,circle,pos=0.7, fill=black,inner sep=0pt,minimum size=2pt, label={right:{\scriptsize$2$}}]{}
  ;

 
 \draw[name intersections={of=0' and L'}] 
     (intersection-1) node[circle, draw, color=black, fill=black,inner sep=0pt,minimum size=2pt,label={above:{\scriptsize$0$}}]{};
     
      \draw[name intersections={of=0'' and L''}] 
     (intersection-1) node[circle, draw, color=black, fill=white,inner sep=0pt,minimum size=2pt,label={right:{\scriptsize$\,$}}]{};
     
\end{tikzpicture}
  \label{fig:LineBundles}
  \end{figure}
 \begin{remark}For any $s\in S$ for which the fiber $C^\alpha_s$ is smooth,  there is a Zariski open neighborhood $U$ of $s$ over which $C^\alpha$ is smooth. The contraction morphisms $C^\alpha\to C^{\alpha'}$ and $C^\alpha\to C^{\alpha''}$ are isomorphisms over $U$ and induce an isomorphism $C^{\alpha'}\stackrel\sim\to C^{\alpha''}$ over $U$. The morphism $\psi$ is then obtained by restriction. This implies that $\psi_s\colon L''_s\to L'_s$ is an isomorphism when $C^\alpha_s$ is smooth. Otherwise $\psi_s$ maps $L''_s$ to the $0$-marked point of $L'_s$ as in Figure \ref{fig:LineBundles}.
 \end{remark}
The construction of the morphism $\psi\colon L''\to L'$ does not rely on the $V$-fern structure of $C$. Indeed, we can preform the same construction for an arbitrary stable $\alpha$-marked curve over $S$, where $\alpha=\{0,v,v'',\infty\}$. Let $(X,\lambda)$ be such a curve. The contractions of $X$ with respect to $\alpha'=\{0,v',\infty\}$ and $\alpha''=\{0,v'',\infty\}$ are both isomorphic to $\BP^1_S$. Denote the complement of the $\infty$-divisor in the contractions by $L_X'$ and $L_X''$ respectively. We endow each with a line bundle structure by taking any isomorphism to the trivial line bundle $\BG_{a,S}$ which extends to an isomorphism $X^{\alpha'}\stackrel\sim\to \BP^1_S$ or $X^{\alpha''}\stackrel\sim\to \BP^1_S$ and sends the $0$-marked section to $0$. We then define the line bundle structures on $L_X'$ and $L''_X$ to be the unique ones induced by these isomorphisms. Note that the resulting line bundle structure does not depend on the choice of isomorphism since any two such choices differ by a linear automorphism of $\BG_{a,S}$. As in the proof of Lemma \ref{7-lem-line-bundle-hom},  we obtain a morphism $\psi_X\colon L_X''\to~L_X'$ of schemes over $S$ via Proposition \ref{contriso}.
\begin{lemma}\label{7-lem-line-bundle-hom-general}\sloppy With the aforementioned line bundle structures, the morphism ${\psi_X\colon L_X''\to L_X}$ is a homomorphism of line bundles.
\end{lemma}
\fussy
\begin{proof}We must show that the following diagrams, which correspond respectively to additivity and compatibility with scalar multiplication, commute:
\begin{equation}\label{add}\begin{gathered}\xymatrix{L''_X\times_S L''_X\ar[d]_{\psi_X\times\psi_X}\ar[r]^-{+''}& L''_X\ar[d]_{\psi_X}\\
L'_X\times_S L'_X\ar[r]^-{+'}& L'_X,}\end{gathered}
\end{equation}
\begin{equation}\label{scalmult}\begin{gathered}\xymatrix{\BG_a\times L''_X\ar[d]_{\id\times\psi_X}\ar[r]& L''_X\ar[d]_{\psi_X}\\
\BG_a\times L'_X\ar[r]& L'_X.}\end{gathered}
\end{equation}
We first reduce to the case where $S$ is integral. For this we note that the constructions of the schemes $L'_X$ and $L''_X$ and the morphism $\psi_X$ are compatible with base change by Proposition \ref{basechangeCi}. Let $\CC_\alpha\to \CM_\alpha$ be the universal family over the moduli space of stable $\alpha$-marked curves. There exists a unique morphism $f\colon S\to \CM_\alpha$ such that $X\cong f^*\CC_\alpha$. Preforming the exact same construction on $\CC_\alpha$, we obtain line bundles $(L_\alpha)'$ and $(L_\alpha)''$ such that $L'_X\cong (f^\alpha)^*(L_\alpha)'$ and $L''_X\cong (f^\alpha)^*(L_\alpha)''$ as line bundles. Moreover, the morphism $\psi_X$ is the pullback of the analogously defined morphism $\psi_\alpha\colon (L_\alpha)''\to (L_\alpha)'.$ Since $\CM_\alpha\cong\BP^1_\BZ$ is integral, it thus suffices to prove the lemma when $S$ is integral.
\vspace{1mm}

Suppose that $S$ is integral and there exists an $s\in S$ such that the fiber $X_s$ is smooth. Then $X$ is smooth over a non-empty open subset $U\subset S$. The morphism $\psi_X$ restricts to an isomorphism over $U$ and extends to an isomorphism $X^{\alpha''}|_U\stackrel\sim\to X|_U$. In terms of the isomorphism of (trivial) line bundles $\tilde\psi_X\colon\BG_{a,U}\stackrel\sim\to\BG_{a,U}$ induced by $\psi_X$, this means that $\tilde\psi_X$ extends to an isomorphism of $\BP^1_U$ and is hence linear. It follows that (\ref{add}) and (\ref{scalmult}) are commutative over $U$. Since $L'_X$ and $L''_X$ are reduced and separated and $+'\circ(\psi\times\psi)$ and $\psi\circ +''$ are equal over $U$, it follows that $+'\circ(\psi_X\times\psi_X)=\psi_X\circ +''$ over $S$; hence (\ref{add}) commutes. A similar argument shows that (\ref{scalmult}) commutes. If $X_s$ is singular for every $s\in S$, then $\psi_X$ must be the zero map, and it follows immediately that (\ref{add}) and (\ref{scalmult}) commute.
\end{proof}
We now return to the setting of $V$-ferns and the morphism $\psi\colon L''\to L'$ of Lemma \ref{7-lem-line-bundle-hom}.
\begin{proposition}\label{7-prop-hom-line-bundles} The morphism $\psi$ is a homomorphism of line bundles.
\end{proposition}
\begin{proof}This follows directly from the Lemma \ref{7-lem-line-bundle-hom-general} and the construction of the line bundle structures on $L'$ and $L''$ from the proof of Proposition \ref{lbstr}, which is the same as for the line bundle structures obtained as in the discussion preceeding Lemma \ref{7-lem-line-bundle-hom-general}.
\end{proof}

\begin{proposition}The morphism $f_C$ factors though $B_V\into \prod_{0\neq V'\subset V}P_{V'}$.
\end{proposition}
\begin{proof}
 By the Lemma \ref{7-lem-line-bundle-hom} and Proposition \ref{7-prop-hom-line-bundles}, we have the following commutative diagram: 
$$\xymatrix{V'\ar[r]^{\lambda'} & L'(S)\\
V''\ar[r]^{\lambda''}\ar@{^{(}->}[u]  &L''(S)\ar[u]_\psi.}$$
Varying $V''$ and $V'$, it then follows from Lemma \ref{AltRep} that $f_C$ factors through $B_V$.
\end{proof}

\begin{proposition}\label{flagrest} The morphism $f_C\colon S\to B_V$ factors through the open immersion $U_\OF\into B_V$.
\end{proposition}
\begin{proof} We just need to verify that the tuple $\CE_\bullet$ associated to $f_C$ satisfies the open condition (\ref{UFRep}) defining $U_\OF$. Fix $0\neq V''\subset V'\subset V$ such that there does \textbf{not} exist $1\leq i\leq n$ with $V''\subset V_i$ and $V'\not\subset V_i$. Denote the corresponding line bundles by $L''$ and $L'$. Let $1\leq j\leq n$ be minimal such that $V''\subset V_j$. By assumption $j$ is also minimal such that $V'\subset V_j$. Choose $v\in V''\setminus (V''\cap V_{j-1})$. It follows from Lemma \ref{infcompF} that $(C')^\infty$ and $(C'')^\infty$ are both equal to the contraction of $C$ with respect to ${\{0,v'',\infty\}}$. By the constructions of $(L'',\lambda_{V''})$ and $(L',\lambda_{V'})$, we conclude that $L'=L''$ and $\lambda_{V''}=\lambda_{V'}|_{V''}$. Recall that 
$$\CE_{V'}=\ker(\lambda_{V'}\otimes\CO_S\colon V'\otimes\CO_S\to \CL'),$$
 where $\CL'$ denotes the sheaf of sections of $L'$. We wish to show that
$$
V'\otimes \CO_S=\CE_{V'}+(V''\otimes\CO_S).
$$
It suffices to show equality on stalks, which follows directly from the fact that $\lambda|_{V''}=\lambda_{V'}|_{V''}$ and that if $\phi\colon M\to L$ is a surjective morphism of $A$-modules and $N\subset M$ is a submodule such that $\phi|_N$ is still surjective, then $M=\ker(\phi)+N$.
\end{proof}

\subsection{Key lemma for induction}
Let $V':=V_{n-1}$ and define $\OF':=\OF\cap V'$. Our inductive proof of Theorem \ref{rep} will rely on the following lemma:
\begin{lemma}[Key Lemma]\label{key} Let $(C,\lambda,\phi)$ and $(D,\mu,\psi)$ be $\OF$-ferns such that the $V'$-contractions $C'$ and $D'$ are isomorphic $\OF'$-ferns. Suppose further that $C^\infty\cong D^\infty$ as $\hat V$-marked curves. Then $C$ and $D$ are isomorphic.
\end{lemma}
\begin{proof}
If $\dim V=n=1$, then $C=C^\infty$ and $D=D^\infty$ and there is nothing to prove. Suppose $n\geq 2$. Fix $v\in V\setminus V'$. By Lemma \ref{infcompF}, the schemes $C^\infty$ and $D^\infty$ are just the contractions of $C$ and $D$ respectively with respect to $\{0,v,\infty\}$. By assumption, there is an isomorphism of $V'$-ferns 
\begin{equation}\label{rho'}\rho'\colon C'\stackrel\sim\to D'\end{equation}
 In particular, the morphism (\ref{rho'}) preserves the $\hat V'$-marking. We break the remainder of the proof into steps:

\vspace{5mm}
\noindent\emph{\textbf{Claim 1}. The isomorphism (\ref{rho'}) also preserves the $v$-marked section.}

\begin{proof}[Proof of Claim 1]Fix $v'\in V'\setminus V_{n-2}$ and consider the contractions $\tilde C$ and $\tilde D$ of $C$ and $D$ with respect to $\{0,v',v,\infty\}.$ Since the isomorphism $\rho^\infty\colon C^\infty\stackrel\sim\to D^\infty$ preserves the $\{0,v',v,\infty\}$-marked sections, it follows that $\tilde C$ and $\tilde D$ are both stabilizations (Section \ref{stab}) of the same stable $\{0,v,\infty\}$-marked curve endowed with an additional $v'$-marked section. By the uniqueness of stabilizations, we obtain an isomorphism $\tilde\rho\colon\tilde C\stackrel\sim\to\tilde D$ of stable $\{0,v',v,\infty\}$-marked curves. 


On the other hand, by Lemma \ref{infcompF}, the contraction to the infinity component $(C')^\infty$ is obtained by contracting $C$ with respect to $\{0,v',\infty\}$. There is thus a contraction morphism $\tilde C\to (C')^\infty$. By the uniqueness of contractions, the isomorphism $\rho'$ induces an isomorphism $(\rho')^\infty\colon (C')^\infty\stackrel\sim\to (D')^\infty$ such that the following diagram commutes:
$$
\xymatrix{\tilde C\ar[d]\ar[r]^{\tilde\rho}_\sim &\tilde D\ar[d]\\ (C')^\infty\ar[r]^{(\rho')^\infty}_\sim &(D')^\infty.}
$$
Since $\tilde\rho$ preserves the $v$-marked section, it follows that $(\rho')^\infty$ does as well.

We claim that the contraction $C'\to (C')^\infty$ is an isomorphism in a neighborhood of the $v$-marked section. Indeed, for every $s\in S$ the $v$-marked point of the fiber $C'_s$ is either equal to the $\infty$-marked point or a smooth point on the $\infty$-component that is disjoint from the $V'$-marked sections. The claim then follows from Proposition \ref{contriso}.

Finally, we consider the commutative diagram:
\begin{equation}\label{vsection}\begin{gathered}
\xymatrix{C'\ar[d]\ar[r]^{\rho'}_\sim &D'\ar[d]\\ (C')^\infty\ar[r]^{(\rho')^\infty}_\sim &(D')^\infty.}\end{gathered}
\end{equation}
By the above argument, the vertical arrows in (\ref{vsection}) become isomorphisms when restricted to a neighborhood of the $v$-marked section. Since $(\rho')^\infty$ preserves the $v$-marked section, we conclude that the same holds for $\rho'$.
\end{proof}

Let $C^v$ and $D^v$ denote the contractions of $C$ and $D$ with respect to $\hat V'\cup \{v\}$. We provide examples of the contractions $C'$ and $C^v$ on fibers with varying associated flags below:
\begin{figure}[H]
\centering
 \caption{Example with  $\dim V=3$ and $q=3$ of the contractions $C^v$ and $C'$ for a fiber with corresponding flag equal to $\OF=\{V_0,V_1,V_2, V_3\}$.}
\begin{tikzpicture}

\draw (.5,2) node{\scriptsize$C$};

  \draw (0,0)  .. controls (1,.25) and (3,.25) .. node[at start,left] {} (4,0) 
  node[circle,pos=0.9, fill=black,inner sep=0pt,minimum size=2pt, label=above:{\scriptsize$\lambda_\infty$}]{};

  \draw (.75,-0.25) .. controls (1,.25) and (1,0.75) .. (.75,1.75)
  node[at start,below] {} ;
    \draw (1.75,-0.25) .. controls (2,.25) and (2,0.75) .. (1.75,1.75);
     \draw (2.75,-0.25) .. controls (3,.25) and (3,0.75) .. (2.75,1.75);

     \draw (.7,1.6) .. controls (.75,1.45) and (1.25,1.45) .. (1.5,1.6)
	node[at start,left] {}     	
     	 node[circle,pos=0.9, fill=black,inner sep=0pt,minimum size=2pt, label=above:{\scriptsize$\lambda_0$}]{}
       	node[circle,pos=0.7, fill=black,inner sep=0pt,minimum size=2pt,]{}
       node[circle,pos=0.5, fill=black,inner sep=0pt,minimum size=2pt, ]{};        
    \draw (.7,1.35) .. controls (.75,1.20) and (1.25,1.20) .. (1.5,1.35)
       		node[circle,pos=0.9, fill=black,inner sep=0pt,minimum size=2pt,]{}
       		node[circle,pos=0.7, fill=black,inner sep=0pt,minimum size=2pt,]{}
       		node[circle,pos=0.5, fill=black,inner sep=0pt,minimum size=2pt,]{};
     \draw (.7,1.1) .. controls (.75,.95) and (1.25,.95) .. (1.5,1.1)
     		node[circle,pos=0.9, fill=black,inner sep=0pt,minimum size=2pt,]{}
       		node[circle,pos=0.7, fill=black,inner sep=0pt,minimum size=2pt,]{}
       		node[circle,pos=0.5, fill=black,inner sep=0pt,minimum size=2pt,]{};;

         \draw (1.7,1.6) .. controls (1.75,1.45) and (2.25,1.45) .. (2.5,1.6)
         		node[circle,pos=0.9, fill=black,inner sep=0pt,minimum size=2pt,label=above:{\scriptsize$\lambda_v$}]{}
       		node[circle,pos=0.7, fill=black,inner sep=0pt,minimum size=2pt,]{}
       		node[circle,pos=0.5, fill=black,inner sep=0pt,minimum size=2pt,]{};
        \draw (1.7,1.35) .. controls (1.75,1.20) and (2.25,1.20) .. (2.5,1.35)
        		node[circle,pos=0.9, fill=black,inner sep=0pt,minimum size=2pt,]{}
       		node[circle,pos=0.7, fill=black,inner sep=0pt,minimum size=2pt,]{}
       		node[circle,pos=0.5, fill=black,inner sep=0pt,minimum size=2pt,]{};
        \draw (1.7,1.1) .. controls (1.75,.95) and (2.25,.95) .. (2.5,1.1)
        		node[circle,pos=0.9, fill=black,inner sep=0pt,minimum size=2pt,]{}
       		node[circle,pos=0.7, fill=black,inner sep=0pt,minimum size=2pt,]{}
       		node[circle,pos=0.5, fill=black,inner sep=0pt,minimum size=2pt,]{};
        
         \draw (2.7,1.6) .. controls (2.75,1.45) and (3.25,1.45) .. (3.5,1.6)
         		node[circle,pos=0.9, fill=black,inner sep=0pt,minimum size=2pt,]{}
       		node[circle,pos=0.7, fill=black,inner sep=0pt,minimum size=2pt,]{}
       		node[circle,pos=0.5, fill=black,inner sep=0pt,minimum size=2pt,]{};
        \draw (2.7,1.35) .. controls (2.75,1.20) and (3.25,1.20) .. (3.5,1.35)
        		node[circle,pos=0.9, fill=black,inner sep=0pt,minimum size=2pt,]{}
       		node[circle,pos=0.7, fill=black,inner sep=0pt,minimum size=2pt,]{}
       		node[circle,pos=0.5, fill=black,inner sep=0pt,minimum size=2pt,]{};
        \draw (2.7,1.1) .. controls (2.75,.95) and (3.25,.95) .. (3.5,1.1)
        		node[circle,pos=0.9, fill=black,inner sep=0pt,minimum size=2pt,]{}
       		node[circle,pos=0.7, fill=black,inner sep=0pt,minimum size=2pt,]{}
       		node[circle,pos=0.5, fill=black,inner sep=0pt,minimum size=2pt,]{};
  \begin{scope}[shift={(5,0)}]
 \draw (.5,2) node{\scriptsize$C^v$};

  \draw[name path=E3] (0,0)  .. controls (1,.25) and (3,.25) .. node[at start,left] {} (4,0) 
  node[circle,pos=0.9, fill=black,inner sep=0pt,minimum size=2pt, label=above:{\scriptsize$\lambda_\infty$}]{};

  \draw (.75,-0.25) .. controls (1,.25) and (1,0.75) .. (.75,1.75)
  node[at start,below] {} ;
    \draw[name path=E2v, opacity=0] (1.75,-0.25) .. controls (2,.25) and (2,0.75) .. (1.75,1.75);
     \draw[opacity=0] (2.75,-0.25) .. controls (3,.25) and (3,0.75) .. (2.75,1.75);

     \draw (.7,1.6) .. controls (.75,1.45) and (1.25,1.45) .. (1.5,1.6)
	node[at start,left] {}     	
     	 node[circle,pos=0.9, fill=black,inner sep=0pt,minimum size=2pt, label=above:{\scriptsize$\lambda_0$}]{}
       	node[circle,pos=0.7, fill=black,inner sep=0pt,minimum size=2pt,]{}
       node[circle,pos=0.5, fill=black,inner sep=0pt,minimum size=2pt, ]{};        
    \draw (.7,1.35) .. controls (.75,1.20) and (1.25,1.20) .. (1.5,1.35)
       		node[circle,pos=0.9, fill=black,inner sep=0pt,minimum size=2pt,]{}
       		node[circle,pos=0.7, fill=black,inner sep=0pt,minimum size=2pt,]{}
       		node[circle,pos=0.5, fill=black,inner sep=0pt,minimum size=2pt,]{};
     \draw (.7,1.1) .. controls (.75,.95) and (1.25,.95) .. (1.5,1.1)
     		node[circle,pos=0.9, fill=black,inner sep=0pt,minimum size=2pt,]{}
       		node[circle,pos=0.7, fill=black,inner sep=0pt,minimum size=2pt,]{}
       		node[circle,pos=0.5, fill=black,inner sep=0pt,minimum size=2pt,]{};;
 
 \draw[name intersections={of=E2v and E3}] 
     (intersection-1) node[circle, draw, color=black, fill=black,inner sep=0pt,minimum size=2pt,label={above:{\scriptsize$\lambda_v$}}]{};   
	\end{scope}

	\begin{scope}[shift={(3.5,-3.5)}]
	\draw (.5,2) node{\scriptsize$C'$};

  \draw[name path=E3,opacity=0] (0,0)  .. controls (1,.25) and (3,.25) .. node[at start,left] {} (4,0) 
  node[circle,pos=0.9, fill=black,inner sep=0pt,minimum size=2pt, label=above:{\scriptsize$\lambda_\infty$}]{};

  \draw[name path=E2] (.75,-0.25) .. controls (1,.25) and (1,0.75) .. (.75,1.75)
  node[at start,below] {} ;
    \draw[name path=E2v, opacity=0] (1.75,-0.25) .. controls (2,.25) and (2,0.75) .. (1.75,1.75);
     \draw[opacity=0] (2.75,-0.25) .. controls (3,.25) and (3,0.75) .. (2.75,1.75);

     \draw (.7,1.6) .. controls (.75,1.45) and (1.25,1.45) .. (1.5,1.6)
	node[at start,left] {}     	
     	 node[circle,pos=0.9, fill=black,inner sep=0pt,minimum size=2pt, label=above:{\scriptsize$\lambda_0$}]{}
       	node[circle,pos=0.7, fill=black,inner sep=0pt,minimum size=2pt,]{}
       node[circle,pos=0.5, fill=black,inner sep=0pt,minimum size=2pt, ]{};        
    \draw (.7,1.35) .. controls (.75,1.20) and (1.25,1.20) .. (1.5,1.35)
       		node[circle,pos=0.9, fill=black,inner sep=0pt,minimum size=2pt,]{}
       		node[circle,pos=0.7, fill=black,inner sep=0pt,minimum size=2pt,]{}
       		node[circle,pos=0.5, fill=black,inner sep=0pt,minimum size=2pt,]{};
     \draw (.7,1.1) .. controls (.75,.95) and (1.25,.95) .. (1.5,1.1)
     		node[circle,pos=0.9, fill=black,inner sep=0pt,minimum size=2pt,]{}
       		node[circle,pos=0.7, fill=black,inner sep=0pt,minimum size=2pt,]{}
       		node[circle,pos=0.5, fill=black,inner sep=0pt,minimum size=2pt,]{};;
 
 \draw[name intersections={of=E2 and E3}] 
     (intersection-1) node[circle, draw, color=black, fill=black,inner sep=0pt,minimum size=2pt,label={right:{\scriptsize$\lambda_\infty$}}]{};   
	
	\end{scope}
	
     
\draw[->,
line join=round,
decorate, decoration={
    segment length=4,
    amplitude=.9,post=lineto,
    post length=2pt}]
 (2.5,-.5)--(3.5,-1.10);
 
 \draw[->,
line join=round,
decorate, decoration={
    segment length=4,
    amplitude=.9,post=lineto,
    post length=2pt}]
 (2.5+4,-.5)--(3.5+2,-1.10);
 
 \draw[->,
line join=round,
decorate, decoration={
    segment length=4,
    amplitude=.9,post=lineto,
    post length=2pt}]
 (4,1)--(5,1);
\end{tikzpicture}
  \label{fig:keylemma}
  \end{figure}
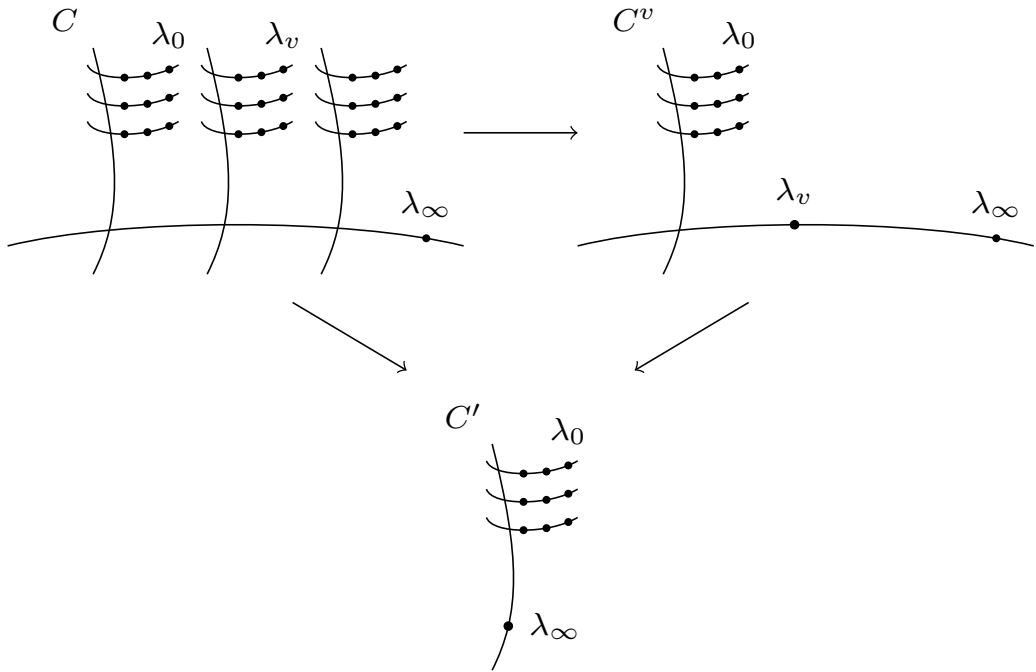
\vspace{5mm}
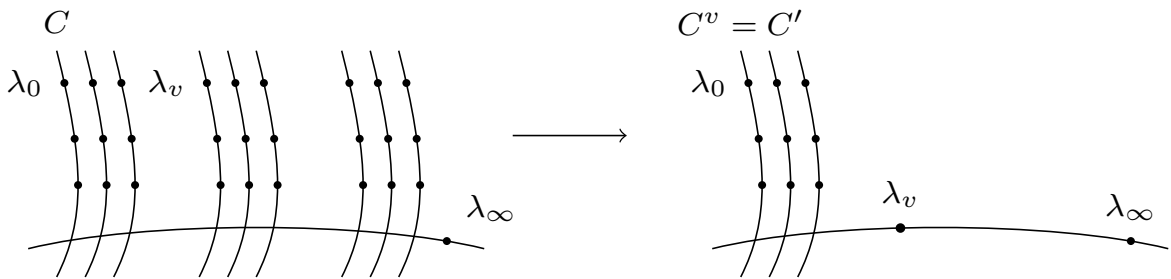
\begin{figure}[H]
\centering
 \caption{Example for a fiber with corresponding flag equal to $\{V_0,V_1, V_3\}$.}
\begin{tikzpicture}

\begin{scope}[shift={(-.5,0)}]
\draw (.5,2) node{\scriptsize$C$};

\begin{scope}[shift={(.25,0)}]
  \draw (0,0)  .. controls (1,.25) and (3,.25) .. node[at start,left] {} (4,0) 
  node[circle,pos=0.9, fill=black,inner sep=0pt,minimum size=2pt, label=60:{\scriptsize$\lambda_\infty$}]{};
  \end{scope}

\begin{scope}[shift={(-.25,0)}]
  \draw (.75,-0.25) .. controls (1,.25) and (1,0.75) .. (.75,1.75)
  node[at start,below] {}
         		node[circle,pos=0.9, fill=black,inner sep=0pt,minimum size=2pt,label=left:{\scriptsize$\lambda_0$}]{}
       		node[circle,pos=0.7, fill=black,inner sep=0pt,minimum size=2pt,]{}
       		node[circle,pos=0.5, fill=black,inner sep=0pt,minimum size=2pt,]{} ;
\end{scope}       		
       		
    \draw (1.75,-0.25) .. controls (2,.25) and (2,0.75) .. (1.75,1.75)
           		node[circle,pos=0.9, fill=black,inner sep=0pt,minimum size=2pt,label=left:{\scriptsize$\lambda_v$}]{}
       		node[circle,pos=0.7, fill=black,inner sep=0pt,minimum size=2pt,]{}
       		node[circle,pos=0.5, fill=black,inner sep=0pt,minimum size=2pt,]{};
       		
\begin{scope}[shift={(.25,0)}]       		
     \draw (2.75,-0.25) .. controls (3,.25) and (3,0.75) .. (2.75,1.75)
            		node[circle,pos=0.9, fill=black,inner sep=0pt,minimum size=2pt,]{}
       		node[circle,pos=0.7, fill=black,inner sep=0pt,minimum size=2pt,]{}
       		node[circle,pos=0.5, fill=black,inner sep=0pt,minimum size=2pt,]{};
       		\end{scope}
     
\begin{scope}[shift={(.25,0)}]
\begin{scope}[shift={(-.25,0)}]
  \draw (.75,-0.25) .. controls (1,.25) and (1,0.75) .. (.75,1.75)
  node[at start,below] {}
         		node[circle,pos=0.9, fill=black,inner sep=0pt,minimum size=2pt,]{}
       		node[circle,pos=0.7, fill=black,inner sep=0pt,minimum size=2pt,]{}
       		node[circle,pos=0.5, fill=black,inner sep=0pt,minimum size=2pt,]{} ;
       		\end{scope}
    \draw (1.75,-0.25) .. controls (2,.25) and (2,0.75) .. (1.75,1.75)
           		node[circle,pos=0.9, fill=black,inner sep=0pt,minimum size=2pt,]{}
       		node[circle,pos=0.7, fill=black,inner sep=0pt,minimum size=2pt,]{}
       		node[circle,pos=0.5, fill=black,inner sep=0pt,minimum size=2pt,]{};
       		\begin{scope}[shift={(.25,0)}]  
     \draw (2.75,-0.25) .. controls (3,.25) and (3,0.75) .. (2.75,1.75)
            		node[circle,pos=0.9, fill=black,inner sep=0pt,minimum size=2pt,]{}
       		node[circle,pos=0.7, fill=black,inner sep=0pt,minimum size=2pt,]{}
       		node[circle,pos=0.5, fill=black,inner sep=0pt,minimum size=2pt,]{};
       		\end{scope}
\end{scope}

\begin{scope}[shift={(.5,0)}]
\begin{scope}[shift={(-.25,0)}]
  \draw (.75,-0.25) .. controls (1,.25) and (1,0.75) .. (.75,1.75)
  node[at start,below] {}
         		node[circle,pos=0.9, fill=black,inner sep=0pt,minimum size=2pt,]{}
       		node[circle,pos=0.7, fill=black,inner sep=0pt,minimum size=2pt,]{}
       		node[circle,pos=0.5, fill=black,inner sep=0pt,minimum size=2pt,]{} ;
       		\end{scope}
    \draw (1.75,-0.25) .. controls (2,.25) and (2,0.75) .. (1.75,1.75)
           		node[circle,pos=0.9, fill=black,inner sep=0pt,minimum size=2pt,]{}
       		node[circle,pos=0.7, fill=black,inner sep=0pt,minimum size=2pt,]{}
       		node[circle,pos=0.5, fill=black,inner sep=0pt,minimum size=2pt,]{};
       		\begin{scope}[shift={(.25,0)}]  
     \draw (2.75,-0.25) .. controls (3,.25) and (3,0.75) .. (2.75,1.75)
            		node[circle,pos=0.9, fill=black,inner sep=0pt,minimum size=2pt,]{}
       		node[circle,pos=0.7, fill=black,inner sep=0pt,minimum size=2pt,]{}
       		node[circle,pos=0.5, fill=black,inner sep=0pt,minimum size=2pt,]{};
       		\end{scope}
\end{scope} 
\end{scope}
    

  \begin{scope}[shift={(5.5,0)}]
 \draw (.5,2) node{\scriptsize$C^v=C'$};

\begin{scope}[shift={(.25,0)}]
  \draw[name path=E3] (0,0)  .. controls (1,.25) and (3,.25) .. node[at start,left] {} (4,0) 
  node[circle,pos=0.9, fill=black,inner sep=0pt,minimum size=2pt, label=above:{\scriptsize$\lambda_\infty$}]{};
  \end{scope}

\begin{scope}[shift={(-.25,0)}]
  \draw (.75,-0.25) .. controls (1,.25) and (1,0.75) .. (.75,1.75)
  node[at start,below] {}
         		node[circle,pos=0.9, fill=black,inner sep=0pt,minimum size=2pt,label=left:{\scriptsize$\lambda_0$}]{}
       		node[circle,pos=0.7, fill=black,inner sep=0pt,minimum size=2pt,]{}
       		node[circle,pos=0.5, fill=black,inner sep=0pt,minimum size=2pt,]{} ;
\end{scope}       		
       		
    \draw[name path=E1v, opacity=0] (1.75,-0.25) .. controls (2,.25) and (2,0.75) .. (1.75,1.75)
           		node[circle,pos=0.9, fill=black,inner sep=0pt,minimum size=2pt,label=left:{\scriptsize$\lambda_v$}]{}
       		node[circle,pos=0.7, fill=black,inner sep=0pt,minimum size=2pt,]{}
       		node[circle,pos=0.5, fill=black,inner sep=0pt,minimum size=2pt,]{};

\begin{scope}[shift={(.25,0)}]
\begin{scope}[shift={(-.25,0)}]
  \draw (.75,-0.25) .. controls (1,.25) and (1,0.75) .. (.75,1.75)
  node[at start,below] {}
         		node[circle,pos=0.9, fill=black,inner sep=0pt,minimum size=2pt,]{}
       		node[circle,pos=0.7, fill=black,inner sep=0pt,minimum size=2pt,]{}
       		node[circle,pos=0.5, fill=black,inner sep=0pt,minimum size=2pt,]{} ;
       		\end{scope}

\end{scope}

\begin{scope}[shift={(.5,0)}]
\begin{scope}[shift={(-.25,0)}]
  \draw (.75,-0.25) .. controls (1,.25) and (1,0.75) .. (.75,1.75)
  node[at start,below] {}
         		node[circle,pos=0.9, fill=black,inner sep=0pt,minimum size=2pt,]{}
       		node[circle,pos=0.7, fill=black,inner sep=0pt,minimum size=2pt,]{}
       		node[circle,pos=0.5, fill=black,inner sep=0pt,minimum size=2pt,]{} ;
       		\end{scope}

\end{scope}
 \draw[name intersections={of=E1v and E3}] 
     (intersection-1) node[circle, draw, color=black, fill=black,inner sep=0pt,minimum size=2pt,label={above:{\scriptsize$\lambda_v$}}]{};

\end{scope}


 \draw[->,
line join=round,
decorate, decoration={
    segment length=4,
    amplitude=.9,post=lineto,
    post length=2pt}]
 (4,1)--(5,1);
\end{tikzpicture}
  \label{fig:keylemma2}
  \end{figure} 
Claim 1 says that $\rho'\colon C'\to D'$ is an isomorphism of \textbf{$(\hat V'\cup\{v\}$)-marked curves.} This implies that $C^v$ and $D^v$ may both be viewed as stabilizations of same stable $\hat V'$-marked curve endowed with an additional $v$-marked section. By the uniqueness of stabilizations, we obtain an isomorphism of stable $(\hat V'\cup \{v\})$-marked curves
\begin{equation}\label{stabiso}\rho^v\colon C^v\stackrel\sim\to D^v.
\end{equation}
\vspace{1mm}

\noindent\emph{\textbf{Claim 2.} This isomorphism preserves the $\hat V$-marking.}
\begin{proof}[Proof of Claim 2] By the uniqueness of contractions, the following diagram commutes: 
$$\xymatrix{C^v\ar[r]^{\rho^v}_\sim\ar[d] & D^v\ar[d]\\C^\infty\ar^{\rho^\infty}_\sim[r]&D^\infty.}$$
The bottom arrow is an isomorphism of $\hat V$-marked curves by assumption. Let $Z_{V'}\subset C^\infty$ denote the union of the images of the $V'$-marked sections and consider the open subscheme $U^\infty:= C^\infty\setminus Z_{V'}$. Let $U^v$ denote the inverse image of $U^\infty$ in $C^v$. By Proposition \ref{contriso}, we obtain an isomorphism 
\begin{equation}\label{infnbhd}
U^v\stackrel\sim\to U^\infty.
\end{equation}

We claim that the $(V\setminus V')$-marked sections of $C^v$ factor through $U^v$. In light of (\ref{infnbhd}), this is equivalent to the same assertion for $C^\infty$ and $U^\infty$. It suffices to check the latter on fibers, so fix $s\in S$. Recall that $C$ is an $\OF$-fern and $V'=V_{n-1}$. Let $\CF_s=\{V_{i_0},\ldots, V_{i_m}\}$ be the flag associated to $C_s$. For $t,u\in V$, we have $\lambda^\infty_t(s)=\lambda^\infty_u(s)$ if and only if $t\cong u\mod V_{i_{m-1}}$. It follows that for $t\in V\setminus V'$, we have $\lambda_t^\infty(s)\neq \lambda_{v'}^\infty(s)$ for all $v'\in V'$. We deduce that $\lambda^\infty_t(s)\cap Z=\emptyset$ for all $t\in V\setminus V'$. The $(V\setminus V')$-marked sections thus factor through $U^v$, as claimed.

Define $W^\infty\subset D^\infty$ and $W^v\subset D^v$ in analog to $U^\infty$ and $U^v$. We obtain a commutative diagram: 
$$\xymatrix{U^v\ar[r]^{\rho^v|_{U^v}}_{\sim}\ar[d]_\wr & W^v\ar[d]_\wr\\U^\infty\ar[r]_\sim^{\rho^\infty|_{U^\infty}}& W^\infty.}$$
 Since the vertical arrows and the bottom horizontal arrow preserve $(V\setminus V')$-marked sections, so does the top arrow. It follows that $\rho^v$ preserves the $(V\setminus V')$-marked sections. Since we already know that $\rho^v$ preserves the $\hat V'$-marked sections, this proves Claim 2.
\end{proof}
Let $Z_{(V\setminus V')}\subset C^v$ denote the union of the images of the $(V\setminus V')$-marked sections. Consider the open subscheme $U:=C^v\setminus Z_{(V\setminus V')}$ and let $U_0$ be the inverse image of $U$ in $C$.\footnote{We use this notation to indicate that $U_0$ is an open neighborhood of the $0$-section in $C$. We will show (Claim 3) that its translates cover $C$.} Again applying Proposition \ref{contriso}, we find that the contraction $C\to C^v$ induces an isomorphism $U_0\stackrel\sim\to U$. 
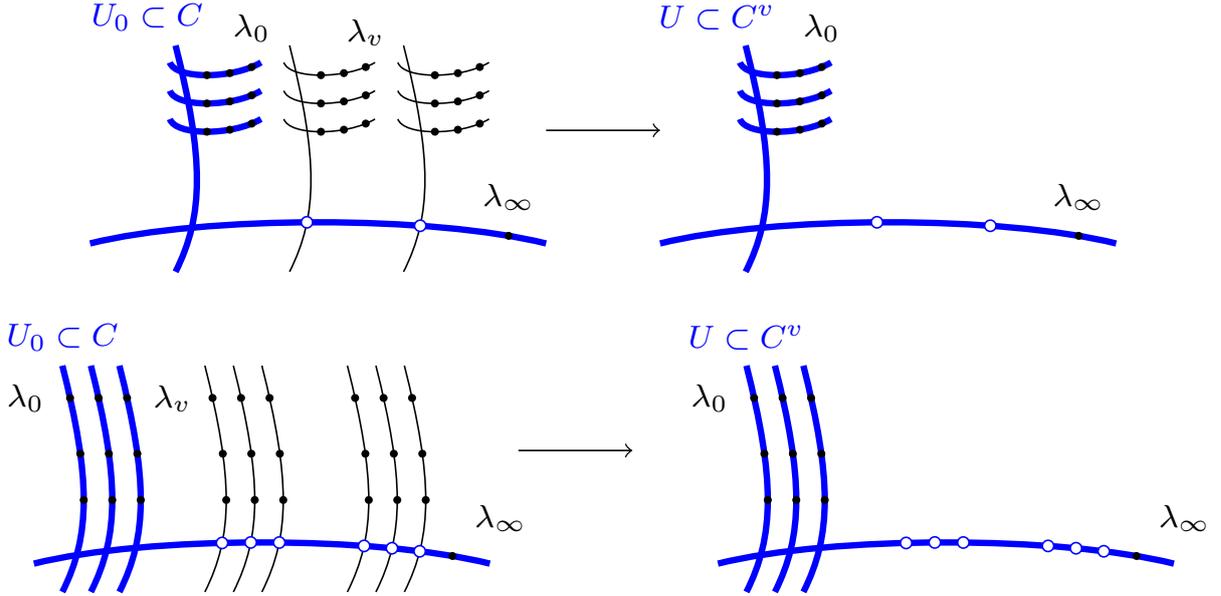
\begin{figure}[H]
\centering
 \caption{Examples of $U_0\stackrel\sim\to U$ for $C$ as in Figures \ref{fig:keylemma} and \ref{fig:keylemma2}. The open subschemes $U_0$ and $U$ are represented by thick lines, and the unfilled points are omitted. Note that, in these examples, the $(V\setminus V')$-translates of $U_0$ cover $C$.}
\begin{tikzpicture}

\draw (.5,2) node{\scriptsize\color{blue}$U_0\subset C$};

  \draw[name path=E03,ultra thick, color=blue] (0,0)  .. controls (1,.25) and (3,.25) .. node[at start,left] {} (4,0) 
  node[circle,pos=0.9, fill=black,inner sep=0pt,minimum size=2pt, label=above:{\color{black}\scriptsize$\lambda_\infty$}]{};

  \draw[ultra thick, color=blue] (.75,-0.25) .. controls (1,.25) and (1,0.75) .. (.75,1.75)
  node[at start,below] {} ;
    \draw[name path= E2] (1.75,-0.25) .. controls (2,.25) and (2,0.75) .. (1.75,1.75);
     \draw[name path=E2'] (2.75,-0.25) .. controls (3,.25) and (3,0.75) .. (2.75,1.75);
     
        \draw[name intersections={of=E03 and E2}] 
     (intersection-1) node[circle, draw, color=blue, fill=white,inner sep=1pt,minimum size=2pt]{};
     
     \draw[name intersections={of=E03 and E2'}] 
     (intersection-1) node[circle, draw, color=blue, fill=white,inner sep=1pt,minimum size=2pt]{};

     \draw[ultra thick, color=blue] (.7,1.6) .. controls (.75,1.45) and (1.25,1.45) .. (1.5,1.6)
	node[at start,left] {}     	
     	 node[circle,pos=0.9, fill=black,inner sep=0pt,minimum size=2pt, label=above:{\color{black}\scriptsize$\lambda_0$}]{}
       	node[circle,pos=0.7, fill=black,inner sep=0pt,minimum size=2pt,]{}
       node[circle,pos=0.5, fill=black,inner sep=0pt,minimum size=2pt, ]{};        
    \draw[ultra thick, color=blue] (.7,1.35) .. controls (.75,1.20) and (1.25,1.20) .. (1.5,1.35)
       		node[circle,pos=0.9, fill=black,inner sep=0pt,minimum size=2pt,]{}
       		node[circle,pos=0.7, fill=black,inner sep=0pt,minimum size=2pt,]{}
       		node[circle,pos=0.5, fill=black,inner sep=0pt,minimum size=2pt,]{};
     \draw[ultra thick, color=blue] (.7,1.1) .. controls (.75,.95) and (1.25,.95) .. (1.5,1.1)
     		node[circle,pos=0.9, fill=black,inner sep=0pt,minimum size=2pt,]{}
       		node[circle,pos=0.7, fill=black,inner sep=0pt,minimum size=2pt,]{}
       		node[circle,pos=0.5, fill=black,inner sep=0pt,minimum size=2pt,]{};;

         \draw (1.7,1.6) .. controls (1.75,1.45) and (2.25,1.45) .. (2.5,1.6)
         		node[circle,pos=0.9, fill=black,inner sep=0pt,minimum size=2pt,label=above:{\scriptsize$\lambda_v$}]{}
       		node[circle,pos=0.7, fill=black,inner sep=0pt,minimum size=2pt,]{}
       		node[circle,pos=0.5, fill=black,inner sep=0pt,minimum size=2pt,]{};
        \draw (1.7,1.35) .. controls (1.75,1.20) and (2.25,1.20) .. (2.5,1.35)
        		node[circle,pos=0.9, fill=black,inner sep=0pt,minimum size=2pt,]{}
       		node[circle,pos=0.7, fill=black,inner sep=0pt,minimum size=2pt,]{}
       		node[circle,pos=0.5, fill=black,inner sep=0pt,minimum size=2pt,]{};
        \draw (1.7,1.1) .. controls (1.75,.95) and (2.25,.95) .. (2.5,1.1)
        		node[circle,pos=0.9, fill=black,inner sep=0pt,minimum size=2pt,]{}
       		node[circle,pos=0.7, fill=black,inner sep=0pt,minimum size=2pt,]{}
       		node[circle,pos=0.5, fill=black,inner sep=0pt,minimum size=2pt,]{};
        
         \draw (2.7,1.6) .. controls (2.75,1.45) and (3.25,1.45) .. (3.5,1.6)
         		node[circle,pos=0.9, fill=black,inner sep=0pt,minimum size=2pt,]{}
       		node[circle,pos=0.7, fill=black,inner sep=0pt,minimum size=2pt,]{}
       		node[circle,pos=0.5, fill=black,inner sep=0pt,minimum size=2pt,]{};
        \draw (2.7,1.35) .. controls (2.75,1.20) and (3.25,1.20) .. (3.5,1.35)
        		node[circle,pos=0.9, fill=black,inner sep=0pt,minimum size=2pt,]{}
       		node[circle,pos=0.7, fill=black,inner sep=0pt,minimum size=2pt,]{}
       		node[circle,pos=0.5, fill=black,inner sep=0pt,minimum size=2pt,]{};
        \draw (2.7,1.1) .. controls (2.75,.95) and (3.25,.95) .. (3.5,1.1)
        		node[circle,pos=0.9, fill=black,inner sep=0pt,minimum size=2pt,]{}
       		node[circle,pos=0.7, fill=black,inner sep=0pt,minimum size=2pt,]{}
       		node[circle,pos=0.5, fill=black,inner sep=0pt,minimum size=2pt,]{};
  \begin{scope}[shift={(5,0)}]
 \draw (.5,2) node{\scriptsize\color{blue}$U\subset C^v$};

  \draw[name path=E3,ultra thick, color=blue] (0,0)  .. controls (1,.25) and (3,.25) .. node[at start,left] {} (4,0) 
  node[circle,pos=0.9, fill=black,inner sep=0pt,minimum size=2pt, label=above:{\color{black}\scriptsize$\lambda_\infty$}]{};

  \draw[ultra thick, color=blue] (.75,-0.25) .. controls (1,.25) and (1,0.75) .. (.75,1.75)
  node[at start,below] {} ;
    \draw[name path=E2v, opacity=0] (1.75,-0.25) .. controls (2,.25) and (2,0.75) .. (1.75,1.75);
     \draw[name path=E2vv,opacity=0] (2.75,-0.25) .. controls (3,.25) and (3,0.75) .. (2.75,1.75);

     \draw[ultra thick, color=blue] (.7,1.6) .. controls (.75,1.45) and (1.25,1.45) .. (1.5,1.6)
	node[at start,left] {}     	
     	 node[circle,pos=0.9, fill=black,inner sep=0pt,minimum size=2pt, label=above:{\color{black}\scriptsize$\lambda_0$}]{}
       	node[circle,pos=0.7, fill=black,inner sep=0pt,minimum size=2pt,]{}
       node[circle,pos=0.5, fill=black,inner sep=0pt,minimum size=2pt, ]{};        
    \draw[ultra thick, color=blue] (.7,1.35) .. controls (.75,1.20) and (1.25,1.20) .. (1.5,1.35)
       		node[circle,pos=0.9, fill=black,inner sep=0pt,minimum size=2pt,]{}
       		node[circle,pos=0.7, fill=black,inner sep=0pt,minimum size=2pt,]{}
       		node[circle,pos=0.5, fill=black,inner sep=0pt,minimum size=2pt,]{};
     \draw[ultra thick, color=blue] (.7,1.1) .. controls (.75,.95) and (1.25,.95) .. (1.5,1.1)
     		node[circle,pos=0.9, fill=black,inner sep=0pt,minimum size=2pt,]{}
       		node[circle,pos=0.7, fill=black,inner sep=0pt,minimum size=2pt,]{}
       		node[circle,pos=0.5, fill=black,inner sep=0pt,minimum size=2pt,]{};;
 
\draw[name intersections={of=E3 and E2v}] 
     (intersection-1) node[circle, draw, color=blue, fill=white,inner sep=1pt,minimum size=2pt]{};
     
\draw[name intersections={of=E3 and E2vv}] 
     (intersection-1) node[circle, draw, color=blue, fill=white,inner sep=1pt,minimum size=2pt]{};
	
	\end{scope}

 
 \draw[->,
line join=round,
decorate, decoration={
    segment length=4,
    amplitude=.9,post=lineto,
    post length=2pt}]
 (4,1)--(5,1);
\end{tikzpicture}
\includestandalone[scale=1.5]{Figures/UandU02}
  \label{fig:U0U}
  \end{figure}
For each $w\in V$, consider the translate $U_w:=\phi_w(U_0)\subset C$.
\vspace{5mm}

\noindent\emph{\textbf{Claim 3.} We have $C=\bigcup_{w\in{V}}U_w.$} 
\begin{proof}[Proof of Claim 3]It suffices to verify the claim on fibers. Let $s\in S$, 
and let $0\leq j<n$ be maximal such that $V_j\in\CF_s$. Write $\oV:=V/V_j$ and $\oV':=V'/V_j$. Then the special points of the $\infty$-component $E_\infty\subset C_s$ are in one-to-one correspondence with $\oV$ by Corollary \ref{specialpoints}. The open subscheme $U_{0,s}$ can be described as follows: The preceding remarks imply that the connected components of $C_s\setminus E_{\infty}$ are indexed by $\oV$. For each $u\in \oV$, we denote the corresponding connected component by $X_u$. Let $p_u:=X_u\cap E_\infty$. We have 
$$U_{0,s}=\Bigl(\bigcup_{u\in\oV'}X_u\cup E_\infty\Bigr)\setminus\Bigl(\bigcup_{u\in \oV\setminus\oV'}\{p_u\}\Bigr).
\footnote{In Figure \ref{fig:U0U} above, the $X_u$ consist of the branches coming off the $\infty$-component. The $p_u$ are simply the non-$\infty$ special points on the $\infty$-component.}$$ 
For any $u\in\oV$ and any lift $\tilde u\in V$ of $u$, we have $\phi_s(\tilde u)(X_0)=X_u$. Since $C_s=\bigl(\bigcup_{u\in\oV}X_u\bigr)\cup E_\infty,$ it follows that the $(V\setminus V')$-translates of $U_{0,s}$ cover $C_s$. In other words, we have $C_s=\bigcup_{w\in V} U_{w,s}$. Since $s$ was arbitrary, this implies the claim.
\end{proof}
For each $w\in V$, define $W_w\subset D$ in analogy to $U_w$. As for $C$, we have $D=\bigcup_{w\in V} W_w$. The isomorphism $\rho^v\colon C^v\stackrel\sim\to D^v$ in (\ref{stabiso}) induces an isomorphism $$\pi_0\colon U_0\stackrel\sim\to W_0.$$
For each $w\in V$, consider the isomorphism 
\begin{equation}\label{transliso}\pi_w:=\psi_w\circ\pi_0\circ\phi_w^{-1}\colon U_w\stackrel\sim\longto W_w.
\end{equation}
\vspace{1mm}

\noindent\emph{\textbf{Claim 4.} The isomorphisms (\ref{transliso}) glue to an isomorphism $\pi\colon C\stackrel\sim\to D$ of $V$-ferns.}

\begin{proof}[Proof of Claim 4]
Consider two elements $u,w\in V$ and let $U_{uw}:=U_u\cap U_w$. We wish to show that $\pi_u|_{U_{uw}}=\pi_w|_{U_{uw}}$. If $\pi_0|_{U_{0t}}=\pi_t|_{U_{0t}}$ for all $t\in V$, then on $U_{uw}$ we have
\begin{eqnarray*}
\pi_u&:=&\psi_{u}\circ\pi_0\circ\phi_u^{-1}=\psi_{u}\circ\pi_{w-u}\circ\phi_u^{-1}
\\&=&\psi_{u}\circ(\psi_{w-u}\circ\pi_0\circ\phi_{w-u}^{-1})\circ\phi_u^{-1}
\\&=&\psi_{w}\circ\pi_{0}\circ\phi_w^{-1}=:\pi_w.
\end{eqnarray*}
We may thus suppose that $u=0$. There is a unique isomorphism $\phi_w^\infty\colon C^\infty\stackrel\sim\to C^\infty$ such that the following diagram commutes
$$\xymatrix{C\ar[d]\ar[r]^\sim_{\phi_w}&C\ar[d]\\
C^\infty\ar[r]^\sim_{\phi^\infty_w}& C^\infty.}
$$ Similarly, there is an isomorphism $\psi_w^\infty\colon D^\infty\stackrel\sim\to D^\infty$ satisfying an analogous diagram.
In addition, by the definition of $\pi_0$ and the uniqueness of contractions, both squares in the diagram
$$\xymatrix{U_{0}\ar[d]\ar[r]^\sim_{\pi_0}&W_0\ar[d]\\
C^v\ar[d]\ar[r]^\sim_{\rho^v}&D^v\ar[d]\\
C^\infty\ar[r]^\sim_{\rho^\infty_w}& D^\infty.}
$$ commute; hence the outer square commutes as well.
Consider the diagram
\begin{equation}\label{cube}
\begin{gathered}\xymatrix{%
&\phi_w^{-1}(U_{0w})\ar@{->}[rr]^(.65){\pi_0}\ar@{->}[dl]\ar@{->}[dd]_(.65){\phi_w}|!{[d];[d]}\hole&&\psi_w^{-1}(W_{0w}) %
\ar@{->}[dl] \ar@{->}[dd]_(.65){\psi_w} 
\\
C^\infty\ar@{->}[rr]^(.65){\rho^\infty}\ar@{->}[dd]_(.65){\phi_w^\infty}&&D^\infty\ar@{->}[dd]_(.65){\psi_w^\infty}\\
&U_{0w}\ar@{->}[rr]^(.65){\pi_0}|!{[r];[r]}\hole\ar@{->}[dl] %
\hole&&W_{0w}
\ar@{->}[dl]\hole
\\
C^\infty\ar@{->}^(.65){\rho^\infty}[rr]&&D^\infty.}\end{gathered}
\end{equation}
The assertion that $\pi_0|_{U_{0w}}=\pi_w|_{U_{0w}}$ is equivalent to the commutativity of rear face of the cube in (\ref{cube}). The above arguments show that the other faces commute. Moreover, the arrows from the rear face to the front face are isomorphisms onto their images by Proposition \ref{contriso}. Hence the rear face also commutes, as desired.
\vspace{2mm}

We thus have $\pi_u|_{U_{uw}}=\pi_w|_{U_{uw}}$ for all $u,w\in V$. The $\pi_w$ therefore glue to an isomorphism $\pi\colon C\stackrel\sim\to D$ of the underlying curves. Since each $\pi_w$ preserves the $(w+\hat V')$-marking, it follows that $\pi$ is an isomorphism of stable $\hat V$-marked curves and hence of $V$-ferns.

\end{proof}
This concludes the proof of Lemma \ref{key}.
\end{proof}
\subsection{Remainder of the proof}
We now conclude the proof of Theorem \ref{flagrep} in the case $\CF:=\OF$, which implies Theorem \ref{rep} by Propositions \ref{redtoflag} and \ref{completetogeneral}. Let $L_\OF$ denote the line bundle over $U_\OF$ whose underlying scheme is $\CC_\OF^\infty\setminus\lambda_{\OF,\infty}(U_\OF)$, as in Proposition \ref{lbstr}. The corresponding fiberwise non-zero linear map $\lambda_\OF\colon V\to L_\OF(U_\OF)$ determines a morphism $$g_\OF\colon U_\OF\to P_V.$$ On the other hand, recall from Section \ref{OVBV} that $U_\OF$ is a locally closed subscheme of $\prod_{0\neq V'\subset V}P_{V'}$, and there is thus a natural morphism
 $$p_\OF\colon U_\OF\into \prod_{0\neq V'\subset V}P_{V'}\to P_V,$$
where the latter morphism is the projection onto the $V$-component.
\begin{lemma}\label{samehom}The morphisms $g_\OF$ and $p_\OF$ are equal.
\end{lemma}
\begin{proof}
Since $\CC_\OF$ is smooth over $\OV$, it follows that $\CC_\OF$ is isomorphic to $\CC_\OF^\infty$ after base change to $\OV$. The restriction of $L_\OF$ to $\OV$ is isomorphic via the morphism $f_V$ defined in (\ref{theMap}) to $\OV\times \BA^1$ with $V$-marked sections corresponding to 
$$\Bigl\{\frac{v}{v_0}\Bigm\vert v\in V\Bigr\}\subset \Gamma(\OV,\CO_{\OV})=RS_{V,0}.$$ 
It follows that $g_\OF|_\OV\colon\OV\to P_V$ is the natural inclusion of $\OV$ as an open subscheme of $P_V$. Since $p_\OF|_{\OV}$ is also the natural inclusion, we deduce that $p_\OF|_{\OV}=g_\OF|_{\OV}$. Since $P_V$ is separated and $U_\OF$ is reduced and $\OV\subset U_\OF$ is dense, it follows that $g_\OF=p_\OF$ (\cite[Corollary 9.9]{GW}).
\end{proof}
As an aside, we observe that there are natural morphisms $g_V,p_V\colon B_V\to P_V$ defined in analogy to $g_\OF$ and $p_\OF$. The same argument as in the above proof yields:
\begin{proposition}\label{PVmorph}
The morphisms $g_V$ and $p_V$ are equal.
\end{proposition}
Let $(C,\lambda,\phi)$ be an $\OF$-fern over an arbitrary scheme $S$ and consider the corresponding morphism $f_C\colon S\to U_\OF$ as defined in (\ref{assocMorph}). 
\begin{lemma}\label{infcomps}
The $\hat V$-marked curves $C^\infty$ and $f_C^*(\CC_{\OF}^\infty)$ are uniquely isomorphic.
\end{lemma}
\begin{proof}The uniqueness results directly from Proposition \ref{i-comp}. In analogy to $L_\OF$, let $L$ denote the line bundle whose underlying scheme is $C^\infty\setminus \lambda_\infty(S)$. Let $\lambda\colon V\to L(S)$ be corresponding fiberwise non-zero linear map, which determines a morphism 
$$g\colon S\to P_V.$$
Let $f_C^*L_\OF$ be the pullback of $L_\OF$ to $S$ along $f_C$. The map $\lambda_\OF$ induces a fiberwise non-zero $\BF_q$-linear map $f_C^*\lambda_\OF\colon V\to f^*_CL_\OF(S)$, which corresponds to the composite $g_\OF\circ f_C$, which equals $p_\OF\circ f_C$ by Lemma \ref{samehom}. By the construction of $f_C$, the morphism $g$ is precisely the $V$-component of $f_C$, i.e., we have $g=p_\OF\circ f_C$. The pairs $(L,\lambda)$ and $(f^*_CL_\OF,f_C^*\lambda_\OF)$ therefore define the same morphism to $P_V$, and it follows that there is an isomorphism of line bundles $\alpha\colon L\stackrel\sim\to f^*_CL_\OF$ inducing a commutative diagram: $$ \xymatrix{V\ar[r]\ar[dr]& L(S)\ar[d]_\wr^\alpha\\&f^*_CL_\OF(S).}$$

We claim that the isomorphism $\alpha$ extends to an isomorphism $C^\infty\stackrel\sim\to f_C^*(\CC^\infty_{\OF})$ of $\hat V$-marked curves. Since $C$ and $\CC_\OF$ are $\OF$-ferns, the schemes $C^\infty$ and $f_C^*(\CC^\infty_\OF)$ are trivial $\BP^1$-bundles given by contracting relative to the $\{0,v,\infty\}$-sections for any choice of $v\in V\setminus V_{n-1}$ by Lemma \ref{infcompF}. There is a unique isomorphism $\overline\beta\colon C^\infty\stackrel\sim\to C^\infty_{\OF,S}$ respecting the $\{0,v,\infty\}$-marking. The restriction to $L$ induces an isomorphism of line bundles  $\beta\colon L\stackrel\sim\to f^*_CL_\OF.$ Since $L$ and $f^*_CL_\OF$ are both trivial and generated by the $v$-marked section, any isomorphism preserving the $v$-section is unique, and hence $\beta=\alpha$. Since $\alpha$ preserves the $V$-marked sections, we conclude that $\overline\beta$ preserves the $\hat V$-marking, and is hence the desired extension of $\alpha$.
\end{proof}

\begin{proposition}\label{pullback}There exists a unique $V$-fern isomorphism $F_C\colon C\stackrel\sim\to f_C^*(\CC_\OF)$ over~$S$.
\end{proposition}
\begin{proof}
The uniqueness follows directly from the uniqueness of morphisms of stable $\hat V$-marked curves. For the existence, we proceed by induction on $\dim V$. Suppose first that $\dim V=1$. Then $U_\OF=\Spec\BF_q$ and $\CC_\OF\cong \BP^1_{\BF_q}$. 
In this case, the proposition follows directly from Proposition \ref{dim1fern}. 
\vspace{2mm}

Now let $\dim V=n>1$. Let $V':=V_{n-1}$ and let $C'$ be the contraction of $C$ with respect to $\hat V'$. Write $\OF':=\OF\cap V'$, and let $f_{C'}$ denote the morphism $S\to U_{\OF'}$ corresponding to $C'$, which is defined analogously to $f_C$. Recall that there is a natural morphism $p\colon U_\OF\to U_{\OF'}$ induced by the forgetful functor, and by construction we have $p\circ f_C=f_{C'}$. By induction, there are unique isomorphisms of $V'$-ferns $$C'\stackrel\sim\to f_{C'}^*(\CC_{\OF'})\cong (f_C^*(\CC_\OF))',$$ with the latter isomorphism given by Proposition \ref{cont'}. Lemmas \ref{key} and \ref{infcomps} then yield the desired isomorphism $F_C\colon C\stackrel\sim\to f_C^*(\CC_\OF)$.
\end{proof}

\begin{proof}[Proof of Theorem \ref{flagrep}]In order to show that $U_\OF$ represents the functor of $\OF$-ferns, we must show that for any scheme $S$ and any $\OF$-fern $(C,\lambda,\phi)$ over $S$, there exists a unique $f\colon S\to U_\OF$ such that $f^*\CC_\OF$ and $C$ are isomorphic $\OF$-ferns. In light of Proposition \ref{pullback}, it only remains to verify the uniqueness of $f$. Consider distinct elements $$f,g\in\Mor(S,U_\OF)\subset\Mor\Bigl(S, \prod_{0\neq V'\subset V}P_{V'}\Bigr)\cong\prod_{0\neq V'\subset V}\Mor(S, P_{V'}).$$ Denote the corresponding pullback $V$-ferns by $f^*\CC_\OF$ and $g^*\CC_\OF$. For each $0\neq V'\subset V$, let $f_{V'}$ (resp. $g_{V'}$) be the $V'$-component of $f$ (resp. $g$). Since $f\neq g$, there exists a $V'$ such that $f_{V'}\neq g_{V'}$. The line bundles together with $V'$-marked sections associated to the $\hat V'$-contractions $(f^*\CC_\OF)'$ and $(g^*\CC_\OF)'$ determine $f_{V'}$ and $g_{V'}$, and are therefore not isomorphic. But any isomorphism of $V$-ferns $f^*\CC_\OF\stackrel\sim\to g^*\CC_\OF$ induces an isomorphism of the associated line bundles preserving marked sections. We deduce by contradiction that $f^*\CC_\OF\not\cong g^*\CC_\OF$.
\end{proof}


\section{Morphisms between moduli and $V$-fern constructions}
Having proved that $B_V$ is a fine moduli space for $\Fern_V$, we now investigate the morphisms of moduli spaces corresponding to the various constructions involving $V$-ferns. We have already described the natural morphism $B_V\to P_V$ (Proposition \ref{PVmorph}) and will see that the remaining constructions correspond to morphisms arising naturally in the functorial interpretations of $Q_V$ and $B_V$ in \cite{P-Sch}. Let $0\neq V'\subsetneq V$ be a subspace. 

\subsection{Contraction}
The forgetful functor $$(\CE_W)_{0\neq W\subset V}\mapsto (\CE_W)_{0\neq W\subset V'}$$ yields a morphism $$\gamma\colon B_V\to B_{V'},$$ which is precisely the morphism induced by the projection 
$$\pi\colon\prod_{0\neq W\subset V}P_W\to \prod_{0\neq W\subset V'}P_W.$$
Contraction of $V$-ferns with respect to $\hat V'$ gives a natural transformation $\Fern_V\to\Fern_{V'}$ and hence also corresponds to a morphism 
$B_V\to B_{V'}.$
Explicitly, one contracts the universal family $\CC_V$ with respect to $\hat V'$ to obtain a $V'$-fern $(\CC_V)'$ over $B_V$, which then corresponds to a unique morphism from $B_V$ to $B_{V'}$ by Theorem \ref{rep}.
\begin{proposition}\label{contrmorph} The morphism $B_V\to B_{V'}$ corresponding to contraction is equal to $\gamma$.
\end{proposition}
\begin{proof}
Let $C$ be a $V$-fern over a scheme $S$ and let $C'$ be the $V'$-fern obtained by contracting $C$. Let 
$$f_C\colon S\to B_V\subset\prod_{0\neq W\subset V}P_W$$
and
$$f_{C'}\colon S\to B_{V'}\subset\prod_{0\neq W\subset V'}P_W$$
be the morphisms corresponding to $C$ and $C'$ respectively, as constructed in Subsection \ref{morphfern}. Recall that for each $0\neq W\subset V$ the $W$-component $S\to P_W$ of $f_C$ depends only on the contraction of $C$ with respect to $\hat W$. When $W\subset V'$ this is isomorphic to the contraction of $C'$ with respect to $\hat W$. The $W$-component of $f_C$ and $f_{C'}$ are thus equal, and it follows that $f_{C'}=\pi\circ f_C$. This implies the proposition.
\end{proof}

\subsection{Grafting and the boundary strata}Let $\oV:=V/V'$ and $\CF:=\{0,V',V\}.$  For each $W\not\subset V'$, let $\oW:=(W+V')/V'$, and consider the natural surjection 
$$\pi_W\colon W\otimes\CO_S\onto \oW\otimes\CO_S.$$
Recall from Proposition \ref{grafting-subquot} and the discussion following it that there is a natural morphism
\begin{eqnarray*}
g\colon B_{V'}\times B_{\oV}&\stackrel\sim\longto &B_\CF\into B_V.\\\nonumber
(\CE'_\bullet,\oCalE_\bullet)&\mapsto& \CE_\bullet
\end{eqnarray*}
Where $\CE_\bullet$ is the tuple defined by
\begin{equation}\label{tuple}\CE_W=\begin{cases}\CE'_W,\mbox{ if $W\subset V'$,}\\
\pi_W^{-1}(\oCalE_\oW),\mbox{ if $W\not\subset V'$}.\end{cases}
\end{equation}

On the other hand, the grafting construction (Section \ref{grafting-section}) yields a natural transformation
$\Fern_V\times\Fern_{V'}\to\Fern_V$ and hence also induces a morphism $B_{V'}\times B_{\oV}\to B_V$. It can be obtained by pulling the universal families $C_{V'}$ and $C_{\oV}$ back to $B_{V'}\times B_{\oV}$, grafting to a $V$-fern, and then taking the unique corresponding morphism $$\tilde g\colon B_{V'}\times B_{\oV}\to B_V.$$
\begin{proposition}\label{graftingmorph} The morphisms $g$ and $\tilde g$ are equal.\end{proposition}
\begin{proof}Let $S$ be a scheme and let $(C',\lambda',\phi')$ and $(\oC,\overline{\lambda},\overline{\phi})$ be a $V'$-fern and $\oV$-fern over $S$ respectively. Denote the graft of $C'$ and $\oC$ by $(C,\lambda,\phi)$.  Let $\CE_\bullet\in B_V(S)$ be the tuple corresponding to $C$. Similarly, let $\CE'_\bullet\in B_{V'}(S)$ and $\OE_\bullet\in B_{\oV}(S)$ be the tuples corresponding to $C'$ and $\oC$. To show that $\tilde g=g,$ we must show that $\CE_\bullet$ is precisely the tuple obtained from $\CE'_\bullet$ and $\OE_\bullet$ via (\ref{tuple}). 

Since $U:=(\Omega_{V'}\times\Omega_\oV)\subset B_{V'}\times B_\oV$ is dense and $B_{V'}\times B_\oV$ is separated and reduced, it suffices to show that $\tilde g|_U=g|_U.$ We may thus assume that $C'$ and $\oC$ are both \textbf{smooth} over $S$. The flag associated to each fiber $C_s$ must be equal to $\CF$, so the restriction $\tilde g|_U$ factors through $\Omega_\CF$.  By \cite[Lemma 10.1]{P-Sch}, the forgetful functor $\CE\bullet\mapsto (\CE_{V'},\CE_V)$ induces an isomorphism from $\Omega_\CF$ onto its image in $P_{V'}\times P_{V}$. It thus suffices to show that $$(\CE_{V'},\CE_V)=\bigl(\CE'_{V'},\pi_V^{-1}(\OE_{\oV})\bigr).$$ 

It follows easily from the fact that the contraction of $C$ with respect to $\hat V'$ is isomorphic to $C'$ that $\CE_{V'}=\CE'_{V'}$. To show that $\CE_V=\pi_V^{-1}(\OE_\oV)$, we first construct a morphism $\rho\colon C\to \oC$. Consider the prestable curve
$$\tilde C:=\oC\sqcup\Bigl(\bigsqcup_{\ov\in\oV}C'\Bigr)$$ as in the construction of $C.$ For each $\ov\in \oV$, we denote the copy of $C'$ indexed by $\ov$ by $(C')_\ov$ and the corresponding $\infty$-section by $\lambda'_{\infty_\ov}$. Let $\pi'\colon C'\to S$ be the structure morphism. Define a morphism $\rho\colon \tilde C\to \oC$ via
\begin{eqnarray*}
\tilde\rho|_{\oC}&:=&\id_\oC\\
\tilde\rho|_{(C')_\ov}&:=&\overline{\lambda}_\ov\circ\pi'
\end{eqnarray*}
By construction, we have $\tilde\rho\circ\overline{\lambda}_\ov=\tilde\rho\circ\lambda'_{\infty_\ov}$ for each $\ov\in \oV$. The universal property of Proposition \ref{clutching}.1 then yields a unique morphism $\rho\colon C\to \oC$ such that the following diagram commutes:
$$\xymatrix{\tilde C\ar[d]\ar[r]^{\tilde\rho}&\oC\\
C\ar[ur]_{\rho}&.}$$
Let $v\in V$, and let $\ov\in \oV$ be its image in $\oV$. It follows directly from the constructions of $\rho$ and the $\hat V$-marking on $C$ that $\rho\circ \lambda_v=\lambda_\ov$. The composite 
$$\xymatrix{C\ar[r]^\rho&\oC\ar[r]&\oC^\infty}$$
 is thus a contraction to the $\infty$-component of $C$. We thus have a commutative diagram
 \begin{equation}\label{SameInfComp}
 \begin{gathered}\xymatrix{\hat V\ar[d]\ar[r]&C^\infty(S)\ar@{=}[d]\\
\hat \oV\ar[r]&\oC^\infty(S).}
\end{gathered}
\end{equation}
Let $\CL$ and $\overline{\CL}$ denote the sheaves of sections of the line bundles associated to $C$ and $\oC$ respectively as in Section \ref{LineBundles}. The diagram (\ref{SameInfComp}) then corresponds to a commutative diagram
$$\xymatrix{V\otimes\CO_S\ar@{->>}[d]^{\pi_V}\ar[r]^-\lambda&\CL\ar@{=}[d]\\
\oV\otimes\CO_S\ar[r]^-{\overline{\lambda}}&\overline{\CL}.}
$$
The sheaves $\CE_V$ and $\OE_{\oV}$ are respectively the kernels of $\lambda$ and $\overline{\lambda}$; hence $\CE_V=\pi_V^{-1}(\OE_{\oV})$, as desired.

\end{proof}
\begin{remark}Let $0\neq V'\subsetneq V$ and $0\neq W\subset V$ be subspaces such that $0\neq W\cap V'\subsetneq W.$ As a consequence of Propositions \ref{contrmorph} and \ref{graftingmorph}, we see that contraction and grafting are compatible in that they induce a commutative diagram
$$\xymatrix{\Fern_{V'}\times\Fern_\oV\ar[d]\ar[r]&\Fern_V\ar[d]\\
\Fern_{W'}\times\Fern_\oW\ar[r]&\Fern_W,}$$
where the horizontal arrows are given by grafting and the vertical arrows by contraction.
\end{remark}
Proposition \ref{graftingmorph} can be generalized as follows. Let $\CF=\{V_0,\ldots V_m\}$ be a flag of $V$. By Proposition \ref{grafting-subquot}, there is a natural closed embedding $$g\colon B_{V_1}\times B_{V_1/V_2}\times\cdots\times B_{V_m/V_{m-1}}\stackrel\sim\longto B_\CF\into B_V.$$ We obtain the following corollary by induction:
\begin{corollary}\label{boundary} The morphism $B_{V_1}\times B_{V_1/V_2}\times\cdots\times B_{V_m/V_{m-1}}\to B_V$ obtained by grafting is equal to $g$.
\end{corollary}

\begin{remark}Let $\CF:=\{V_0,\ldots, V_m\}$. Corollary \ref{boundary} shows that the closed subschemes $B_\CF\into B_V$ represent $V$-ferns that can be obtained by grafting a given tuple $(C_i)_{1\leq i\leq m},$ where $C_i$ is a $(V_i/V_{i-1})$-fern. The boundary stratum $\Omega_\CF$ then corresponds to those tuples consisting of smooth ferns.
\end{remark}



\subsection{Reciprocal maps and the scheme $Q_V$}
In this section we discuss an alternative compactification of $\OV$, denoted by $Q_V$, which is the primary focus of the same paper \cite{P-Sch} in which Pink and Schieder introduce $B_V$. Let $\mrV:=V\setminus\{0\}$. Recall that we defined $S_V$ to be the symmetric algebra of $V$ over $\BF_q.$ Let $K_V$ be the quotient field of $S_V$, and let $R_V$ be the $\BF_q$-subalgebra of $K_V$ generated by all elements of the form $\frac{1}{v}$ with $v\in\mrV$. We view $R_V$ as a graded ring, where each $\frac{1}{v}$ is homogeneous of degree $-1$. Define
$$Q_V:=\Proj(R_V).$$
The natural inclusion $R_V\into RS_V$ induces an open immersion $\OV\into Q_V$.
\vspace{2mm}

 Let $S$ be a scheme and let $\CL$ be an invertible sheaf on $S$.
\begin{definition}\label{recipmap}
We say a map $\rho\colon \mrV\to \CL(S)$ is \emph{reciprocal} if
\begin{enumerate}
\item $\rho(\alpha v)=\alpha^{-1}\rho(v)$ for all $v\in\mrV$ and $\alpha\in\BF_q^\times$, and
\item $\rho(v)\cdot\rho(v')=\rho(v+v')\cdot(\rho(v)+\rho(v'))$ in $\CL^{\otimes 2}(S)$ for all $v, v'\in\mrV$ such that $v+v'\in\mrV.$
\end{enumerate}
\end{definition}
\begin{definition}
Let $i\colon V'\into V$ denote the natural inclusion of a non-zero subspace $V'$ of $V$. Given a reciprocal map $\rho'\colon \mrV'\to\Gamma(S,\CL)$, one defines the \emph{extension by zero} of $\rho'$ to be the reciprocal map
$$i_*\rho'\colon \mrV\to\Gamma(S,\CL),\,\,\,v\mapsto\begin{cases}\rho'(v)&\mbox{if $v\in V',$}\\
0&\mbox{otherwise}.\end{cases}$$
\end{definition}

In \cite{P-Sch}, the scheme $Q_V$ is given a modular interpretation in terms of reciprocal maps. For this, we observe that the natural map $\rho_V\colon\mrV\to R_{V,-1}\cong\Gamma(Q_V,\CO_{Q_V}(1))$ given by $v\mapsto \frac{1}{v}$ is reciprocal and fiberwise non-zero.
\begin{theorem}[\cite{P-Sch}, Theorem 7.10, Proposition 7.11]\label{QVrep} The scheme $Q_V$ with the universal family $(\CO_{Q_V}(1),\rho_V)$ represents the functor which associates to a scheme $S$ over $\BF_q$ the set of isomorphism classes of pairs $(\CL,\rho)$ consisting of an invertible sheaf $\CL$ on $S$ and a fiberwise non-zero reciprocal map $\rho\colon \mrV\to\Gamma(S,\CL)$.  The open subscheme $\OV\subset Q_V$ represents the subfunctor of fiberwise injective reciprocal maps.
\end{theorem}
Pink and Schieder construct a natural morphism
\begin{equation}\label{BVQV}
\pi_Q\colon B_V\to Q_V,
\end{equation}
which is compatible with the natural inclusions $\OV\into B_V$ and $\OV\into Q_V$. This is done as follows (see \cite[Theorem 10.17]{P-Sch} for details): Let $\CE_\bullet\in B_V(S)$. One considers the commutative diagram of invertible sheaves $(V'\otimes \CO_S)/\CE_{V'}$ for all $0\neq V'\subset V,$ with the natural homomorphisms $(V''\otimes\CO_S)/\CE_{V''}\to(V'\otimes\CO_S)/\CE_{V'}$ for each $0\neq V''\subset V'\subset V.$ Dualizing, one obtains a direct system (which is not filtered) and defines 
\begin{equation}\label{RecipLB}\CL:=\varinjlim_{0\neq V'\subset V}\bigl((V'\otimes \CO_S)/\CE_{V'}\bigr)^{-1}.
\end{equation}
 One can show that $\CL$ is an invertible sheaf on $S$.
To define a reciprocal map $\rho\colon \mrV\to\CL(S)$, we associate to each $v\in\mrV$ the $\CO_S$-linear homomorphism $\ell_v\colon (\BF_q\otimes\CO_S)/\CE_{\BF_qv}\to\CO_S,\,\, v\otimes a\mapsto a$. Then $\ell_v\in\bigl((\BF_qv\otimes\CO_S)/\CE_{\BF_qv}\bigr)^{-1},$ and we define $\rho(v)$ to be the image of $\ell_v$ in $\CL$ under the natural homomorphism from $(\BF_qv\otimes\CO_S)/\CE_{\BF_qv}$ to the direct limit. The association $\CE_\bullet\mapsto (\CL,\rho)$ is functorial and induces the map (\ref{BVQV}).
\vspace{2mm}

 Locally, the invertible sheaf $\CL$ has a simpler description:
\begin{lemma}[\cite{P-Sch}, Section 10]\label{locRecipLB}Let $\CE_\bullet\in U_\CF(S)$ for some flag $\CF=\{V_0,\ldots,V_m\}$ of $V$. Then $\CL=\bigl((V_1\otimes\CO_S)/\CE_{V_1}\bigr)^{-1}.$
\end{lemma}
Using Lemma \ref{locRecipLB} and unraveling all of the definitions, we obtain the following:
\begin{lemma}\label{RecipOV}
The restriction $\pi_Q|_\OV$ corresponds to the functor sending the isomorphism class of a pair $(\CM,\lambda)\in\OV(S)$ consisting of an invertible sheaf $\CM$ over a scheme $S$ and a fiberwise injective linear map $\lambda\colon V\to\Gamma(S,\CM)$ to the isomorphism class of the pair $(\CM^{-1},\lambda^{-1})\in Q_V(S)$, where $\lambda^{-1}$ is the fiberwise invertible reciprocal map defined by the formula $\lambda^{-1}(v)=\lambda(v)^{-1}$ for all $v\in\mrV$.
\end{lemma}

The existence of (\ref{BVQV}) suggests that one can construct a line bundle and reciprocal map directly from a given $V$-fern, and we now describe how this is done. The construction is analogous to that of the line bundle and fiberwise non-zero linear map associated to a $V$-fern from subsection \ref{LineBundles}. Let $(C,\lambda,\phi)$ be a $V$-fern over a scheme $S$. Let $C^0$ denote the contraction to the $0$-component (Section \ref{contractionicomp}) and let $\hat L:=C^0\setminus \lambda_0(S)$. The marking $\lambda^0\colon \hat V\to C^0(S)$ induces a map
$$\rho\colon\mrV\to \hat L(S).$$

\begin{proposition}\label{VfernRecip}There is natural structure 
of line bundle on $\hat L$ such that
\begin{enumerate}[label={(\alph*)}]
\item the zero section of $\hat L$ is $\lambda^0(\infty)$, and
\item the induced map $\rho\colon\mrV\to\hat\CL(S)$, where $\hat\CL$ denotes the sheaf of sections of $\hat L$, is a reciprocal map.
\end{enumerate}
\end{proposition}
\begin{proof}We use the notation $(-)^0$ to denote the contraction of a $V$-fern to the $0$-component. Observe first that the same argument as in the proof of Proposition \ref{basechangelb} shows that the construction of $\hat L$ is compatible with base change. Thus, by pullback of line bundle structure along the morphism $f_C\colon S\to B_V$ corresponding to $C$, it suffices to prove the proposition for the universal $V$-fern $\CC_V$. Let $\hat L_V:=(\CC_V)^0\setminus\{\lambda^0_V(B_V)\}$.
Fix a flag $\CF=\{V_0,\ldots, V_m\}$ and let $\hat L_\CF$ be the pullback of $\hat L_V$ along the open immersion $U_\CF\into B_V$. We will first endow $\hat L_\CF$ with a line bundle structure satisfying the properties in the statement of the proposition. We then show that these glue to the desired line bundle structure on $\hat L_V$.
\vspace{2mm}

For any $v\in V_1\setminus\{0\}$, the $0$- and $v$- and $\infty$-sections of $(\CC_\CF)^0$ are disjoint, so $(\CC_\CF)^0$ is a trivial $\BP^1$-bundle over $U_\CF$. Choose any isomorphism 
$$\overline{\eta}_\CF\colon (\CC_\CF)^0\stackrel\sim\to \BP^1_{U_\CF}$$
 sending the $0$-section to $(0:1)$ and the $\infty$-section to $(1:0)$, and consider the restriction 
 $$\eta_\CF\colon\hat L_\CF\stackrel\sim\to\BA^1_{U_\CF}\subset\BP^1_{U_\CF}$$ mapping to the standard affine chart around $(1:0)$. Furthermore, let $L_\CF$ be the inverse image under $\overline{\eta}_\CF$ of the standard affine chart around $(0:1)$ so that $(\CC_\CF)^0=L_\CF\cup \hat L_\CF.$
\vspace{2mm}

We endow $L_\CF$ and $\hat L_\CF$ with the line bundle structures induced by pulling back the trivial line bundle structure on $\BA^1_{U_\CF}$ via $\overline{\eta}_\CF$. Consider the map
$$\rho_\CF\colon \mrV\to \hat L_\CF(U_\CF).$$
In order to show that $\rho_\CF$ is a reciprocal map, it suffices by the reducedness of $U_\CF$ to show that
$$\rho_s\colon \mrV\to \hat L_s:=\hat L_\CF\times_{U_\CF}\Spec k(s)$$ is reciprocal for each $s\in U_\CF$. Fix $s\in U_\CF$. There is a unique flag $\CF_s\subset \CF$ such that $s\in \Omega_{\CF_s}$. Let $i\geq 1$ be minimal such that $V_i\in\CF_s$. In a Zariski open neighborhood $U\subset U_\CF$ of $s$, the contraction to the $0$-component $(\CC_\CF)^0$ is equal to the $V_i$-fern $(\CC_\CF)_i$ obtained by contracting with respect to $\hat V_i$. We also have $L_\CF|_U= L_i|_U$, where $L_i$ is the line bundle associated to $(\CC_\CF)_i$ as in Subsection \ref{LineBundles}. By Proposition \ref{fibnonzero}, the marking $\lambda_i\colon V_i\to L_\CF|_U(U)$ is a fiberwise invertible $\BF_q$-linear map. By construction, the composite of $\rho_\CF$ with the restriction $\hat\CL_\CF(U_\CF)\to\hat\CL_\CF(U)$ is precisely the extension by zero of the reciprocal map from $\mrV_i$ to $\hat \CL_\CF(U)$ defined by $v\mapsto \frac{1}{\lambda_i(v)}$. It follows that $\rho_s$ is reciprocal, as desired.
\vspace{2mm}

It remains to show that the line bundle structures on each $\hat L_\CF$ glue to a line bundle structure on $\hat L_V$. This follows from the fact that the isomorphisms $\eta_\CF$ used to define the (trivial) line bundle structure on $\hat L_\CF$ extend to isomorphisms $\overline{\eta}_\CF\colon (\CC_\CF)^0\stackrel\sim\to\BP^1_{U_\CF}$ so that for two flags $\CF$ and $\CG$, the automorphism $\eta_\CG\circ\eta_{\CF}^{-1}\colon\BA^1_{U_{\CF\cap\CG}}\stackrel\sim\to\BA^1_{U_{\CF\cap\CG}}$ extends to an automorphism of $\BP^1_{U_{\CF\cap\CG}}$ and is hence linear.
\end{proof}
As a corollary of the above proof, we obtain the following:
\begin{corollary}\label{RestOV}
Let $(C,\lambda,\phi)$ be a smooth $V$-fern over $S$. Let $(L,\lambda)$ be the line bundle and fiberwise invertible linear map from $V$ to $L(S)$ associated to $C$ as in Section \ref{LineBundles} and let $(\hat L,\rho)$ be as in Proposition \ref{VfernRecip}. Then $(\hat L,\rho)\cong (L^{-1},\lambda^{-1}).$ In addition, we have $\pi_Q|_\OV=\hat\pi_Q|_\OV.$
\end{corollary}
\begin{proof}
The latter statement follows from Lemma \ref{RecipOV}.
\end{proof}

The association $(C,\lambda,\phi)\mapsto (\hat L,\rho)$ determines a morphism 
\begin{equation}\label{hatBVQV}\hat\pi_Q\colon B_V\to Q_V
\end{equation}
\begin{proposition}\label{QVmorph}
The morphisms (\ref{BVQV}) and (\ref{hatBVQV}) are equal.
\end{proposition}
\begin{proof}Since $\OV\subset B_V$ is dense and $B_V$ is reduced and separated, it suffices to show that $\pi_Q|_\OV=\hat\pi_Q|_\OV$, which is Corollary \ref{RestOV}.
\end{proof}
\begin{remark}Using the fact that the corresponding diagram in \cite[Theorem 10.17]{P-Sch} commutes, we deduce that the morphisms corresponding to line bundle constructions involving $V$-ferns of Propositions \ref{lbstr} and \ref{VfernRecip} fit into the following commutative diagram:
$$\xymatrix{P_V&B_V\ar[l]\ar[r]&Q_V\\&\OV\ar@{^(->}[ul]\ar@{^(->}[ur]\ar@{^(->}[u]&.}$$

\end{remark}




\end{document}


\begin{tikzpicture}
  \draw (0,0)  .. controls (1,.25) and (3,.25) .. node[at start,left] {\scriptsize $C$} (4,0) 
  node[circle,pos=0.9, fill=black,inner sep=0pt,minimum size=2pt, label=above:{\scriptsize$\lambda_\infty$}]{}
   node[circle,pos=0.1, fill=black,inner sep=0pt,minimum size=2pt, label=above:{\scriptsize$\lambda_0$}]{}
       	node[circle,pos=.15, fill=black,inner sep=0pt,minimum size=2pt,]{}
       node[circle,pos=.2, fill=black,inner sep=0pt,minimum size=2pt, ]{}
       
          node[circle,pos=0.3, fill=black,inner sep=0pt,minimum size=2pt, ]{}
       	node[circle,pos=0.35, fill=black,inner sep=0pt,minimum size=2pt,]{}
       node[circle,pos=0.4, fill=black,inner sep=0pt,minimum size=2pt, ]{}
       
        node[circle,pos=0.5, fill=black,inner sep=0pt,minimum size=2pt, ]{}
       	node[circle,pos=0.55, fill=black,inner sep=0pt,minimum size=2pt,]{}
       node[circle,pos=0.6, fill=black,inner sep=0pt,minimum size=2pt, ]{};

\end{tikzpicture}